\documentclass[reqno,11pt,letterpaper]{amsart}
\usepackage{amsmath,amssymb,amsthm,graphicx,mathrsfs,url}
\usepackage[usenames,dvipsnames]{color}
\usepackage[colorlinks=true,linkcolor=Red,citecolor=Green]{hyperref}
\usepackage{enumitem}
\usepackage[russian, francais, english]{babel}
\usepackage{csquotes}

\def\?[#1]{\textbf{[#1]}\marginpar{\Large{\textbf{??}}}}

\newtheorem{theo}{Theorem}
\newtheorem{corr}[theo]{Corollary}
\newtheorem{prop}{Proposition}[section]

\newtheorem{lemm}[prop]{Lemma}
\theoremstyle{definition}
\newtheorem{defi}[prop]{Definition}
\newtheorem{rem}[prop]{Remark}

\numberwithin{equation}{section}

\newcommand{\rr}{\mathbb{R}}

\newcommand{\dd}{\mathrm{d}}
\newcommand{\Lie}{\mathcal{L}}

\newcommand{\Cinf}{\mathcal{C}^\infty}
\newcommand{\e}{\mathrm{e}}
\newcommand{\Lien}{{\Lie_{X}^\nabla}}
\newcommand{\Lienz}{{\Lie^{\nabla(z)}_X}}
\newcommand{\Lienzo}{{\Lie^{\nabla(z_0)}_X}}

\renewcommand{\H}{\mathcal{H}}
\newcommand{\grdet}[1]{{\det}_{\mathrm{gr}, #1}}
\newcommand{\grtr}[1]{{\tr}_{\mathrm{gr}, #1}}
\newcommand{\sdet}[1]{{\det}_{\mathrm{s}, #1}}
\newcommand{\str}[1]{{\tr}_{\mathrm{s}, #1}}
\newcommand{\strf}{{\tr}_\mathrm{s}^\flat}
\newcommand{\grtrf}{{\tr}_\mathrm{gr}^\flat}
\newcommand{\I}{\mathcal{I}}
\newcommand{\dom}{\mathcal{O}}
\newcommand{\Rep}{\mathrm{Rep}}
\newcommand{\G}{\mathcal{G}}
\newcommand{\Lient}{\Lie_{X_t}^\nabla}
\renewcommand{\div}{\mathrm{div}}
\newcommand{\crit}{\mathrm{Crit}}
\newcommand{\ind}{\mathrm{ind}}
\newcommand{\cs}{\mathrm{cs}}
\DeclareMathOperator{\grad}{grad}

\newcommand{\Lientilde}{\Lie_{\widetilde{X}}^\nabla}
\newcommand{\Lientildem}{\Lie_{-\widetilde{X}}^\nabla}

\DeclareMathOperator{\Res}{Res}

\let\Im=\Imag

\DeclareMathOperator{\Op}{Op}

\DeclareMathOperator{\rank}{rank}

\let\Re=\Real

\DeclareMathOperator{\supp}{supp}

\DeclareMathOperator{\WF}{WF}

\DeclareMathOperator{\tr}{tr}
\DeclareMathOperator{\id}{Id}

\title[Dynamical torsion for contact Anosov flows]{Dynamical torsion for contact Anosov flows}

\author[Y.~Chaubet]{Yann Chaubet}

\address{Universit\'e Paris-Sud, D\'epartement de Math\'ematiques, 91400
Orsay, France}
\email{yann.chaubet@math.u-psud.fr}

\author[N.V.~Dang]{Nguyen Viet Dang}

\address{Institut Camille Jordan (U.M.R. CNRS 5208), Universit\'e Claude Bernard Lyon 1, B\^atiment Braconnier, 43, boulevard du 11 novembre 1918, 
69622 Villeurbanne Cedex }

\email{dang@math.univ-lyon1.fr}

\begin{document}
\maketitle

\begin{abstract}
We introduce a new object, the dynamical torsion, which extends the potentially ill-defined value at $0$ of the Ruelle zeta function of a contact Anosov flow twisted by an acyclic representation of the fundamental group.
We show important properties of the dynamical torsion: it is invariant under deformations among contact Anosov flows, it is holomorphic in the representation and it has the same logarithmic derivative as some refined combinatorial torsion of Turaev. This shows that the ratio between this torsion and the Turaev torsion is locally constant on the space of acyclic representations. 

 In particular, for contact Anosov flows path connected to the geodesic flow of some hyperbolic manifold among contact Anosov flows, we relate the leading term of the Laurent expansion of $\zeta$ at the origin, the Reidemeister torsion and the torsions of the finite dimensional complexes of the generalized resonant states of both flows for the resonance $0$. This extends previous work of~\cite{dang2018fried} on the Fried conjecture near geodesic flows of
hyperbolic $3$--manifolds, to hyperbolic manifolds of any odd dimension.
\end{abstract}

\section{Introduction}
Let $M$ be a closed odd dimensional manifold and $(E, \nabla)$ be a flat vector bundle over $M$. The parallel transport of the connection $\nabla$ induces a conjugacy class of representation $\rho \in \mathrm{Hom}(\pi_1(M), \mathrm{GL}(\mathbb{C}^d))$. Moreover, $\nabla$ defines a differential on the complex $\Omega^\bullet(M,E)$ of $E$-valued differential forms
on $M$ and thus cohomology groups $H^\bullet(M,\nabla) = H^\bullet(M,\rho)$ (note that we use the notation $\nabla$ also for the twisted differential induced by $\nabla$ whereas it can be denoted by $\dd^\nabla$ in other references). We will say that $\nabla$ (or $\rho$) is 
acyclic if those cohomology groups are trivial. If $\rho$ is unitary (or equivalently,
if there exists a hermitian structure on $E$ preserved by $\nabla$) and acyclic, Reidemeister \cite{reidemeister1935homotopieringe} introduced a combinatorial invariant $\tau_{\mathrm{R}}(\rho)$ 
of the pair $(M,\rho)$, the so-called \textit{Franz-Reidemeister torsion} (or R-torsion), which is a positive number. This allowed him to classify lens spaces in dimension $3$; this result was then extended in higher dimension by Franz \cite{franz1935torsion} and De Rham \cite{de1936nouveaux}.

 On the analytic side, Ray-Singer \cite{ray1971r} introduced another invariant $\tau_{\mathrm{RS}}(\rho)$, the \textit{analytic torsion}, defined via the derivative at $0$ of the spectral zeta function of the Laplacian given by the Hermitian metric on $E$ and some Riemannian metric on $M$. They conjectured the equality of the analytic and Reidemeister torsions. This conjecture was proved independently by Cheeger \cite{cheeger1979analytic} and M\"uller \cite{muller1978analytic}, assuming only that $\rho$ is unitary (both R-torsion and analytic torsion have a natural extension if $\rho$ is unitary and not acyclic). The Cheeger-M\"uller theorem was extended to unimodular flat vector bundles by M\"uller \cite{muller1993analytic} and to arbitrary flat vector bundles by Bismut-Zhang \cite{bismut1992extension}.

In the context of hyperbolic dynamical systems, Fried \cite{fried1987lefschetz} was interested in the link between the R-torsion and the Ruelle zeta function of an Anosov flow $X$ which is defined by
$$
\zeta_{X, \rho}(s) = \prod_{\gamma \in \mathcal{G}_X^\#} \det \Bigl(1- \varepsilon_\gamma \rho([\gamma])\e^{-s \ell(\gamma)}\Bigr), \quad \Re(s) \gg 0,
$$
where $\mathcal{G}_X^\#$ is the set of primitive closed orbits of $X$, $\ell(\gamma)$ is the period of $\gamma$ and $\varepsilon_\gamma = 1$ if the stable bundle of $\gamma$ is orientable and $\varepsilon_\gamma = -1$ otherwise. Using Selberg's trace formula, Fried could relate the behavior of $\zeta_{X,\rho}(s)$ near $s=0$ with $\tau_\textrm{R}$, as follows.
\begin{theo}[Fried \cite{fried1986analytic}]\label{thm:Fried}
Let $M = SZ$ be the unit tangent bundle of some closed oriented hyperbolic manifold $Z$, and denote by $X$ its geodesic vector field on $M$. Assume that $\rho : \pi_1(M) \to O(d)$ is an acyclic and unitary representation. Then $\zeta_{X,\rho}$ extends meromorphically to $\mathbb{C}$. Moreover, it is holomorphic near $s=0$ and
\begin{equation}\label{eq:morgado}
|\zeta_{X,\rho}(0)|^{(-1)^r} = \tau_\textrm{R}(\rho),
\end{equation}
where $2r+1 = \dim M$, and $\tau_\textrm{R}(\rho)$ is the Reidemeister torsion of $(M, \rho)$.
\end{theo}

In \cite{fried1987lefschetz}, Fried conjectured that the same holds true for negatively curved locally symmetric spaces. 
This was proved by Moscovici-Stanton \cite{moscovici1991r}, Shen \cite{shen2017analytic}. 

For analytic Anosov flows, the meromorphic continuation of $\zeta_{X,\rho}$ was proved by Rugh \cite{rugh1996generalized} in dimension $3$ and by Fried \cite{fried1995meromorphic} in higher dimensions. Then Sanchez-Morgado \cite{sanchez1993lefschetz,sanchez1996r} proved in dimension $3$ that if $\rho$ is acyclic, unitary, and satisfies that $\rho([\gamma]) - \varepsilon_\gamma^j$ is invertible for $j\in \{0,1\}$ for some closed orbit $\gamma$, then (\ref{eq:morgado}) is true.

For general smooth Anosov flows, the meromorphic continuation of $\zeta_{X,\rho}$ was proved by Giuletti-Liverani-Pollicott \cite{giulietti2013anosov} and alternatively by Dyatlov--Zworski~\cite{dyatlov2013dynamical}.
The Axiom A case was treated by Dyatlov--Guillarmou in \cite{dyatlov2018afterword}. 
Quoting the commentary from Zworski \cite{zworski2018commentary} on 
Smale's seminal paper~\cite{smale1967differentiable}, equation~(\ref{eq:morgado}) "would link dynamical, spectral and topological quantities. [$\dots$] In the case of smooth manifolds of variable negative curvature, equation (\ref{eq:morgado}) remains completely open".
However in~\cite{dyatlov2017ruelle},
the authors were able to prove the following.
\begin{theo}[Dyatlov--Zworski]\label{th:DZruelle}
Suppose $(\Sigma,g)$ is a negatively curved orientable Riemannian surface. Let $X$ denote the associated geodesic vector field on the unitary cotangent bundle $M = S^*\Sigma$. 
Then for some $c \neq 0$, we have as $s \to 0$
\begin{equation}
\zeta_{X,\mathbf{1}}(s)=cs^{|\chi(\Sigma)|}\left(1+\mathcal{O}(s) \right),
\end{equation}
where $\mathbf{1}$ is the trivial representation $\pi_1(S^*\Sigma) \to \mathbb{C}^*$ and $\chi(\Sigma)$ is the Euler characteristic of $\Sigma$. In particular, the length spectrum $\left\{\ell(\gamma),~\gamma\in \mathcal{G}_X^\#\right\}$
determines the genus.
\end{theo}

This result was generalized in a recent preprint of Ceki\'c--Paternain~\cite{CekicPaternain} to volume
preserving Anosov flows in dimension $3$.

In the same spirit and using similar microlocal methods,
Guillarmou-Rivi\`ere-Shen and the second author~\cite{dang2018fried} showed
\begin{theo}[D--Rivi\`ere--Guillarmou--Shen]\label{thm:dgrs}
Let $\rho$ be an acyclic representation of $\pi_1(M)$. Then the map
$$
X \mapsto \zeta_{X,\rho}(0)
$$ 
is locally constant on the open set of smooth vector fields which are Anosov and for which $0$ is not a Ruelle resonance, that is, $0 \notin \Res(\Lien)$. If $X$ preserves a smooth volume form and $\dim(M) = 3$, equation (\ref{eq:morgado}) holds true if $b_1(M) \neq 0$ or under the same assumption used in \cite{sanchez1996r}.
\end{theo}

Let us comment on the notion of Ruelle resonance to explain the assumptions in the above Theorem.
All recent works on the analytic continuation of the Ruelle zeta
function are important byproducts of new functional methods to study hyperbolic flows. 
They rely on the construction 
of spaces of anisotropic distributions adapted to the dynamics, initiated by Kitaev \cite{kitaev99}, Blank--Keller--Liverani \cite{blank2002ruelle}, Baladi \cite{baladi2005anisotropic,baladi2018dynamical}, Baladi--Tsujii \cite{baladi2007anisotropic}, Gou\"ezel-Liverani \cite{gouezel2006banach}, Liverani \cite{liverani2005fredholm}, Butterley-Liverani \cite{butterley2007smooth,butterley2013robustly},
and many others where we refer 
to the recent book~\cite{baladi2018dynamical} for precise references. 
These spaces allow to define a suitable notion of spectrum for the operator $\Lien = \nabla \iota_X + \iota_X \nabla$, 
where $\iota$ is the interior product, acting on $\Omega^\bullet(M,E)$. This spectrum is the set of so-called Pollicott--Ruelle resonances $\Res(\Lien)$, which forms a discrete subset of $\mathbb{C}$ and contains all zeros and poles of $\zeta_{X,\rho}$.
Faure--Roy--Sj\"ostrand \cite{faure2008semi}, Faure--Sj\"ostrand \cite{faure2011upper} initiated the use of
microlocal methods to describe these anisotropic spaces of distributions giving a purely microlocal approach to study Ruelle resonances.
This was further developped by Dyatlov and Zworski to study Ruelle zeta functions.

However, if $0 \in \Res(\Lien)$ then
the results of~\cite{dang2018fried} no longer apply since the zeta function $\zeta_{X,\rho}$ might have a pole or zero at $s=0$ (recall zeros and poles of $\zeta_{X,\rho}$ are contained in $\Res(\Lien)$). One goal of this article is to remove the assumption that $0$ is not a Ruelle resonance. 
In the spirit of Theorem \ref{th:DZruelle} and the Fried conjecture, we can state a Theorem which follows from more general results of the present paper (see \S \ref{sec:mainresults}).
\begin{theo}\label{corhyperbolic}
Let $(Z,g_0)$ be a compact hyperbolic manifold of dimension $q$ and $\rho$ be the lift to $S^*Z$ of some acyclic unitary representation $\pi_1(Z) \to \mathrm{GL}(\mathbb{C}^d)$. Then for every metric $g$ which is path connected to $g_0$ in the space of negatively curved metrics,
there exists $m(g, \rho)\in \mathbb{Z}$ s.t.
\begin{equation}\label{eg:ChaubetDang}
\left\vert\zeta_{X_g, \rho}(s)\right\vert^{(-1)^q}  =
 |s|^{(-1)^qm(g, \rho)}\underset{\text{R-torsion}}{\underbrace{\tau_\textrm{R}(\rho)}}
\left \vert \frac{\tau(C^\bullet(X_{g_0},\rho))}{\tau(C^\bullet(X_g, \rho))}\right\vert \left(1+\mathcal{O}(s) \right),
\end{equation}
where $X_g$ denotes the geodesic vector field of $g$ 
and $\tau(C^\bullet(X_g, \rho))$ is the refined torsion of the finite dimensional space of resonant states for the resonance $0$ of $(X_g,\rho)$.
\end{theo}
In the above statement, the vector field $X_g$ generates a contact Anosov flow on the contact manifold
$ S_g^*Z = \{(x,\xi) \in T^*Z,~|\xi|_g = 1\}$~\footnote{This means concretely that changing the metric $g$ on $Z$ affects both the contact form $\vartheta$ and Reeb field $X$ on $S^*Z$}.
The finite dimensional torsion $\tau(C^\bullet(X_g, \rho))$ will be described in \S\ref{sec:mainresults} below.
%
\section{Main results}\label{sec:mainresults}

There are two restrictions in Theorem~\ref{thm:dgrs} of~\cite{dang2018fried}.
The first restriction is that 
$$|\zeta_{X,\rho}(0)|^{(-1)^r} = \tau_\textrm{R}(\rho) $$ is an equality of positive real numbers and
the representation $\rho$ is unitary. For arbitrary acyclic representations $\rho:\pi_1(M)\to \mathrm{GL}(\mathbb{C}^d)$, one could wonder if the phase of the complex number $\zeta_{X,\rho}(0)$ contains topological information. For instance, if it can be compared
with some complex valued torsion defined for general acyclic representations $\rho:\pi_1(M)\to \mathrm{GL}(\mathbb{C}^d)$.
The second restriction concerns the assumption
that $0$ is not a Ruelle resonance. Apart from the low dimension cases studied in~\cite{dang2018fried}, 
this assumption is particularly hard to control and is difficult to check for explicit examples. 

Our goal in the present work is to partially overcome these two obstacles.
In the case where $X$ induces a contact flow, which means
that $X=X_\vartheta$ is the Reeb vector field of some contact form $\vartheta$ on $M$, we deal with these difficulties by introducing a \textit{dynamical torsion} $\tau_\vartheta(\rho)$ which is a new object defined for any acyclic $\rho$ and which coincides with $\zeta_{X,\rho}(0)^{\pm 1}$ if $0 \notin \Res(\Lien)$.
Before stating our main results, let us 
introduce the two main characters of our discussion in the next two subsections.

\subsection{Refined versions of torsion}
The Franz--Reidemeister torsion $\tau_\mathrm{R}$ is given by the modulus 
of some alternate product of determinants and is therefore real valued. One cannot get a canonical object by removing 
the modulus since one has to make some choices to define the combinatorial torsion, and the 
ambiguities in these choices affect the alternate product of determinants.
To remove indeterminacies arising in the definition of  the combinatorial torsion, Turaev \cite{turaev1986reidemeister,turaev1990euler,turaev1997torsion} introduced in the acyclic case a refined version of the combinatorial R-torsion, the \textit{refined combinatorial torsion}. It is a complex number $\tau_\frak{e,o}(\rho)$ which depends on additional combinatorial data, namely an Euler structure $\frak{e}$ and a cohomological orientation $\frak{o}$ of $M$, and which satisfies $|\tau_\frak{e,o}(\rho)| = \tau_\mathrm{R}(\rho)$ if $\rho$ is acyclic and unitary.
We refer the reader to subsection~\ref{subsec:eulerstruct} for precise definitions.
Later, Farber-Turaev \cite{farber2000poincare} extended this object to non-acyclic representations. In this case, $\tau_\frak{e,o}(\rho)$ is an element of the determinant line of cohomology $\det H^\bullet(M,\rho)$.

Motivated by the work of Turaev, but from the analytic side,
Braverman-Kappeler \cite{braverman2007refined, braverman2008refined, braverman2007ray} introduced a refined version of the Ray-Singer analytic torsion called \textit{refined analytic torsion} $\tau_{\mathrm{an}}(\rho)$. It is complex valued in
the acyclic case. Their construction heavily relies on the existence of a chirality operator $\Gamma_g$, that is,
$$\Gamma_g : \Omega^\bullet(M,E) \to \Omega^{n-\bullet}(M,E), \quad \Gamma_g^2 = \id,$$
which is a renormalized version of the Hodge star operator associated to some metric $g$.
They showed that the ratio
$$
\rho \mapsto \displaystyle{\frac{\tau_{\mathrm{an}}(\rho)}{\tau_{\mathfrak{e,o}}(\rho)}}
$$
is a holomorphic function on the representation variety given by an explicit local expression, up to a local constant of modulus one. This result is an extension of the Cheeger-M\"uller theorem. Simultaneously, Burghelea-Haller \cite{burghelea2007complex} introduced a complex valued analytic torsion, which is closely related to the refined analytic torsion~\cite{braverman2007comparison} when it is defined; see \cite{huang2007refined} for comparison theorems.

\subsection{Dynamical torsion}
We now assume that $X = X_\vartheta$ is the Reeb vector field of some contact form $\vartheta$ on $M$. Let us briefly describe the construction of the dynamical torsion. In the spirit of \cite{braverman2007refined}, we use a chirality operator associated to the contact form $\vartheta$, $$\Gamma_\vartheta : \Omega^\bullet(M,E) \to \Omega^{n-\bullet}(M,E), \quad \Gamma_\vartheta^2= \id,$$
cf. \S\ref{sec:deftors}, analogous to the usual Hodge star operator associated to a Riemannian metric. Let $C^\bullet \subset \mathcal{D}^{'\bullet}(M,E)$ be the finite dimensional space of Pollicott-Ruelle generalized resonant states of $\Lien$ for the resonance $0$, that is,
 $$
 C^\bullet = \Bigl\{u \in \mathcal{D}^{'\bullet}(M,E), ~\WF(u) \subset E_u^*,~\exists N \in \mathbb{N},~\left(\Lien\right)^Nu = 0\Bigr\},
 $$
 where $\WF$ is the H\"ormander wavefront set, $E_u^* \subset T^*M$ is the unstable cobundle of $X$~\footnote{ 
 the annihilator of $E_u \oplus \mathbb{R}X$ where $E_u\subset TM$ denotes the unstable bundle of the flow}, cf.  \S\ref{sec:policott}, and $\mathcal{D}'(M,E)$ denotes the space of $E$-valued currents. Then $\nabla$ induces a differential on $C^\bullet$ which makes it a finite dimensional 
cochain complex. Then a result from \cite{dang2017topology} implies 
that the complex $(C^\bullet,\nabla)$ 
is {acyclic} if we assume that $\nabla$ is. 
Because $\Gamma_\vartheta$ commutes 
with $\Lien$, it induces a chirality operator on $C^\bullet$. Therefore 
we can compute the torsion $\tau(C^\bullet, \Gamma_\vartheta)$ of the finite dimensional complex $\left(C^\bullet,\nabla\right)$ with respect to $\Gamma_\vartheta$, as described in \cite{braverman2007refined} (see \S\ref{sec:torsion}). Then we define the \emph{dynamical torsion} $\tau_\vartheta$ as the product
$$
\tau_{\vartheta}(\rho)^{(-1)^q} 
=\pm\underset{\text{finite dimensional torsion}}{{\underbrace{\tau(C^\bullet, \Gamma_\vartheta)^{(-1)^q} }}}
\times
\underset{\text{renormalized Ruelle zeta function at }s=0}{
\underbrace{ \lim_{s\rightarrow 0}s^{-m(X, \rho)}\zeta_{X,\rho}(s)}} \in \mathbb{C} \setminus 0,
$$
where the sign $\pm$ will be given later, $m(X,\rho)$ is the order of $\zeta_{X,\rho}(s)$ at $s=0$ and $q=\frac{\dim(M)-1}{2}$ is the dimension of the unstable bundle of $X$. Note that the order $m(X,\rho)\in \mathbb{Z}$ is a priori not stable under perturbations of $(X,\rho)$, in fact both terms in the product may not be invariant under small changes of $\vartheta$ whereas the dynamical torsion $\tau_{\vartheta}$ has interesting invariance
properties as we will see below.

\subsection{Statement of the results.}

We denote by $\mathrm{Rep}_{\mathrm{ac}}(M,d)$ the set of acyclic representations $\pi_1(M) \to \mathrm{GL}(\mathbb{C}^d)$ and by $\mathcal{A}\subset \Cinf(M, TM)$ the space of contact forms on $M$ whose Reeb vector field induces an Anosov flow. This is an open subset of the space of contact forms. For any $\vartheta \in \mathcal{A}$, we denote by $X_\vartheta$ its Reeb vector field.
Recall that we want to 
study the value at $0$ without taking the modulus. As in Fried's case, $\zeta_{X,\rho}(0)$ might be ill--defined since $0\in \Res(\Lien)$ and this was the reason for introducing the more general object $\tau_\vartheta(\rho)$. 
Our goal is to compare this new complex number with the refined torsion. As a first step towards this,
our first result shows $\tau_\vartheta(\rho) $ is invariant by small perturbations of the contact form $\vartheta\in \mathcal{A}$.
\begin{theo}\label{thm:main1}
Let $(M,\vartheta)$ be a contact manifold such that the Reeb vector field of $\vartheta$ induces an Anosov flow. Let $(\vartheta_\tau)_{\tau \in (-\varepsilon, \varepsilon)}$ be a smooth family in $\mathcal{A}$.
Then $\partial_\tau \log \tau_{\vartheta_\tau}(\rho) = 0$ for any $\rho \in \Rep_\mathrm{ac}(M,d)$.
\end{theo}
\begin{rem} In the case where the representation $\rho$ is not acyclic, we can still define $\tau_\vartheta(\rho)$ as an element of the determinant line $\det H^\bullet(M,\rho)$, cf Remark \ref{rem:notacyclic}. This element is invariant under perturbations of $\vartheta \in \mathcal{A}$, cf. Remark \ref{rem:notacyclic2}.
\end{rem}
This  result implies that the map $\vartheta\in\mathcal{A}\mapsto\tau_{\vartheta}(\rho)$ is locally constant for all $\rho\in \mathrm{Rep}_{\mathrm{ac}}(M,d)$. To apply Theorem \ref{thm:dgrs} in the case of contact Anosov flows, we need to make small perturbations near a contact Anosov flow s.t. $0 \notin \Res(\Lien) $. Assume we have a $C^1$ family of contact Anosov flows $(X_t)_{t\in [0,1]}$ s.t. $0$ is not a resonance of $(X_0,X_1)$, but if $0 \in \Res(\mathcal{L}_{X_u}^\nabla) $ for some $u\in (0,1)$ then we cannot claim that $\zeta_{X_0,\rho}(0)=\zeta_{X_1,\rho}(0)$ using Theorem \ref{thm:dgrs}. 
In the present case, the assumption $0 \notin \Res(\Lien)$ is no longer needed 
and we can make more general perturbations provided we stay 
within the set of contact Anosov flows.

Our second result aims to compare $\tau_\vartheta$ with Turaev's refined version of the Reidemeister torsion
$\tau_{\frak{e},\frak{o}}$, which depends on some choice of Euler structure $\frak{e}$ and orientation $\frak{o}$. An analog of the Fried conjecture would be to prove the equality $\tau_\vartheta(\rho)=\tau_{\frak{e},\frak{o}}(\rho)  $ for some $(\frak{e},\frak{o} )$ and for all $\rho\in \mathrm{Rep}_{\mathrm{ac}}(M,d)$ (this would imply $|\tau_R(\rho)| = |\zeta_{X,\rho}(0)|^{\pm1}$ if $\rho$ is acyclic and unitary and if $0 \notin \Res(\Lien)$). We prove a weaker result, which shows that the derivatives in $\rho\in \mathrm{Rep}_{\mathrm{ac}}(M,d)$  of $\log\tau_\vartheta(\rho)$ and $\log\tau_\frak{e,o}(\rho)$ coincide.

\begin{theo}\label{thm:main2}
Let $(M,\vartheta)$ be a contact manifold such that the Reeb vector field of $\vartheta$ induces an Anosov flow.
Then 
$\rho\in \mathrm{Rep}_{\mathrm{ac}}(M,d)\mapsto \tau_\vartheta(\rho)$ is holomorphic~\footnote{$\mathrm{Rep}_{\mathrm{ac}}(M,d)$ is a variety over $\mathbb{C}$ see subsection~\ref{ss:holoconnections} for the right notion of holomorphicity} and 
there exists an Euler structure $\frak{e}$ such that for any cohomological orientation $\frak{o}$ and any smooth family $(\rho_u)_{u \in(- \varepsilon, \varepsilon)}$ of $\Rep_\mathrm{ac}(M,d)$,
$$
\partial_u \log \tau_{\vartheta}(\rho_u) = \partial_u \log \tau_{\frak{e}, \frak{o}}(\rho_u)
$$
Moreover, if $\dim M = 3$ and $b_1(M) \neq 0$, the map $\rho \mapsto \tau_\vartheta(\rho) / \tau_{\frak{e}, \frak{o}}(\rho)$ is of modulus one on the connected components of $\Rep_{\mathrm{ac}}(M,d)$ containing an acyclic and unitary representation.
\end{theo}
In~\cite{dang2018fried}, for $\rho$ acyclic, the authors proved that $0\notin \Res(\Lien)$ implies that $X\mapsto \zeta_{X,\rho}(0)$ is locally constant. Then the equality $\vert \zeta_{X,\rho}(0)\vert=\tau_{\mathrm{R}}(\rho)$ was proved indirectly by working near analytic Anosov flows in dimension $3$ or near geodesic flows of hyperbolic $3$-manifolds, where the equality is known by the works of Sanchez Morgado and Fried, 
relying on the fact that $\zeta_{X,\rho}(0)$ remains constant by small perturbations of the vector field $X$.
Whereas in the above Theorem, for any contact Anosov flow in any odd dimension, we directly compare the $\log$ derivatives of the dynamical and refined torsions as holomorphic functions on the representation variety. We do not 
need to work near some vector field $X$ for which the equality $\vert \zeta_{X,\rho}(0)\vert=\tau_{\mathrm{R}}(\rho)$ is already known. 

Finally, our third result aims to describe how $\partial_u \log\tau_\vartheta(\rho_u)$ depends on the choice of the 
contact Anosov vector field $X_\vartheta$.
\begin{theo}\label{thm:main3}
Let $(M,\vartheta)$ be a contact manifold such that the Reeb vector field of $\vartheta$ induces an Anosov flow. Let $(\rho_u)_{|u|\leq \varepsilon}$ be a smooth family in $\Rep_\mathrm{ac}(M,d)$. Then for any $\eta \in \mathcal{A}$ 
$$
\partial_u\log \tau_{\eta}(\rho_u) = \partial_u \log \tau_{\vartheta}(\rho_u) + \partial_u\log \underset{\textrm{topological}}{\underbrace{\det\left(\left\langle\rho_u,\cs(X_\vartheta, X_\eta)\right\rangle\right)}}
$$
as differential $1$-forms on $\mathrm{Rep}_{\mathrm{ac}}(M,d)$ and
where $\cs(X_\vartheta, X_\eta) \in H_1(M, \mathbb{Z})$ is the Chern-Simons class of the pair of vector fields $(X_\vartheta, X_\eta)$.
\end{theo}
The underbraced term  is topological since it is defined as the pairing of the representation $\rho$ with
the Chern--Simons class $\mathrm{cs}(X_\vartheta, X_\eta)\in H_1(M, \mathbb{Z})$ which measures the obstruction
to find a homotopy among non singular vector fields connecting $X_\vartheta$ and $X_\eta$~\footnote{Note that taking the determinant $\det\left( \langle \rho, \mathrm{cs}(X_\vartheta, X_\eta) \rangle\right)$ does not depend on the choice of representative of $\mathrm{cs}(X_\vartheta, X_\eta)$ in $\pi_1(M)$}. In particular, if $\vartheta$ and $\eta$ are connected by some path in $\mathcal{A}$, then $\mathrm{cs}(Y_\eta,X_\vartheta)=0$ which yields $\det\left\langle\rho,\bigl(\cs(X_\vartheta, X_\eta)\bigr)\right\rangle=1$ hence $\partial_u\log \tau_{\eta}(\rho_u) = \partial_u \log \tau_{\vartheta}(\rho_u)$ for any acyclic $\rho$. We refer the reader to subsection~\ref{subsec:chernsimons} for the definition of Chern-Simons classes. 

Because the dynamical torsion is constructed with the help of the dynamical zeta function $\zeta_{X,\rho}$, we deduce from the above theorem some informations about the behavior of $\zeta_{X,\rho}(s)$ near $s=0$, as follows.
\begin{corr}
Let $M$ be a closed odd dimensional manifold. Then for every connected open subsets $\mathcal{U}\subset \mathrm{Rep}_{\mathrm{ac}}(M,d)$ and $\mathcal{V}\subset \mathcal{A}$, there exists a constant $C$
such that for every Anosov contact form $\vartheta \in \mathcal{V}$ and every representation $\rho \in \mathcal{U}$,
\begin{equation}
\zeta_{X_\vartheta,\rho}(s)^{(-1)^q} = Cs^{(-1)^qm(\rho,X_\vartheta)}  \frac{\tau_{\frak{e}_{X_\vartheta},\frak{o}}(\rho)}{\tau\left(C^\bullet\left(\vartheta,\rho\right), \Gamma_\vartheta\right)} \left(1+\mathcal{O}(s) \right),
\end{equation}
where $X_\vartheta$ is the Reeb vector field of $\vartheta$, $(E_\rho, \nabla_\rho)$ is the flat vector bundle over $M$ induced by $\rho$,  $C^\bullet\left(\vartheta,\rho\right) \subset \mathcal{D}^{'\bullet}(M,E_\rho)$ is the space of generalized resonant states for the resonance $0$ of $\Lie_{X_\vartheta}^{\nabla_\rho}$ and $m(X_\vartheta, \rho)$ is the vanishing order of $\zeta_{X_\vartheta, \rho}(s)$ at $s=0$.
\end{corr}

\subsection{Methods of proof}
Let us briefly sketch the proof of Theorems \ref{thm:main1} and \ref{thm:main2} which relies essentially on two variational arguments: we compute the variation of $\tau_\vartheta(\nabla)$ when we perturb the contact form $\vartheta$ and the connection $\nabla$. As we do so, the space $C^\bullet(\vartheta, \nabla)$ of Pollicott-Ruelle resonant states of $\Lie_{X_\vartheta}^\nabla$ for the resonance $0$ may radically change. Therefore, it is convenient to consider the space $C^\bullet_{[0, \lambda]}(\vartheta,\nabla)$ instead, which consists of the generalized resonant states for $\Lie_{X_\vartheta}^\nabla$ for resonances $s$ such that $|s| \leqÊ\lambda$, where $\lambda \in (0,1)$ is chosen so that $\{|s|Ê= \lambda\} \cap \Res(\Lie_{X_\vartheta}^\nabla) = \emptyset$. Then  using \cite[Proposition 5.6]{braverman2007refined} and multiplicativity of torsion, 
one can show that
\begin{equation}
\tau_\vartheta(\nabla) = \pm \tau \left(C^\bullet_{[0,\lambda]}(\vartheta,\nabla), \Gamma_\vartheta \right) \zeta_{X_\vartheta,\rho}^{(\lambda, \infty)}(0)^{(-1)^q},
\end{equation}
where $\zeta_{X_\vartheta,\rho}^{(\lambda, \infty)}$ is a renormalized version of $\zeta_{X_\vartheta,\rho}$ (we remove all the poles and zeros of $\zeta_{X_\vartheta,\rho}$ within $\{s\in \mathbb{C},~|s|\leq\lambda\}$), see \S\ref{sec:deftors}.
Thus we can work with the space $C^\bullet_{[0,\lambda]}(\vartheta, \nabla)$, which behaves nicely under perturbations of $X$ thanks to Bonthonneau's construction of uniform anisotropic Sobolev spaces for families of Anosov flows \cite{bonthonneau2018flow}, and also under perturbations of $\nabla$.

Now consider a smooth family of contact forms $(\vartheta_t)_t$ for $|t| < \varepsilon$ such that their Reeb vector fields $(X_t)_{t}$ induce Anosov flows. Then Theorem \ref{theo:invariance} says that for any acyclic $\nabla$, the map $t \mapsto \tau_{\vartheta_t}(\nabla)$ is differentiable and its derivative vanishes. This follows from a result of \cite{braverman2007refined} which allows to compute the variation of the torsion of a finite dimensional complex when the chirality operator is perturbed, and on a variation formula of the map $t \mapsto \zeta_{X_t,\rho}(s)$ for $\Re(s)$ big enough obtained in \cite{dang2018fried}.

Next, consider a smooth family of flat connections $z \mapsto \nabla(z)$, where $z$ is a complex number varying in a small neighborhood of the origin and write $\nabla(z) = \nabla + z \alpha + o(z)$ where $\alpha \in \Omega^1(M, \mathrm{End}(E))$. Then we show in \S\ref{sec:variationconnexion}, in the same spirit as before, that $z \mapsto \tau_{\vartheta}(\nabla(z))$ is complex differentiable and its logarithmic derivative reads
$$
\partial_z|_{z = 0} \log \tau_{\vartheta}(\nabla(z)) = - \strf \alpha K \e^{-\varepsilon \Lie_{X_\vartheta}^\nabla},
$$
where $\varepsilon > 0$ is small enough, $\strf$ is the super flat trace, cf. \S\ref{subsec:flattrace}, and $K : \Omega^\bullet(M,E) \to \mathcal{D}^{'\bullet}(M,E)$ is a cochain contraction, that is, it satisfies $\nabla K + K \nabla = \id_{\Omega^\bullet(M,E)}$. On the other hand, we can compute, using the formalism of \cite{dang2017spectral},
$$
\partial_z|_{z = 0} \log \tau_{\frak{e}_\vartheta,\frak{o}}(\nabla(z)) = - \strf \alpha \widetilde{K} \e^{-\varepsilon \Lie_{-\widetilde{X}}^\nabla} - \int_e \tr \alpha,
$$
where $\frak{e}_\vartheta$ is an Euler structure canonically associated to $\vartheta$, $\widetilde{K}$ is another cochain contraction, $\widetilde{X}$ is a Morse-Smale gradient vector field and $e \in C_1(M,\mathbb{Z})$ is a singular one-chain representing the Euler structure $\frak{e}_\vartheta$, cf. \S\ref{sec:turaevtorsion}. Now using the fact that $K$ and $\widetilde{K}$ are cochain contractions, one can see that
$$
\alpha\left( K \e^{-\varepsilon \Lie_{X_\vartheta}^\nabla} - \widetilde{K} \e^{-\varepsilon \Lie_{\widetilde{X}}^\nabla}\right) = \alpha R_\varepsilon + [\nabla, \alpha G_\varepsilon],
$$
where $R_\varepsilon$ is an operator of degree -1 whose kernel is, roughly speaking, the union of graphs of the maps $\e^{-\varepsilon X_u}$, where $(X_u)_u$ is a non-degenerate family of vector fields interpolating $X_\vartheta$ and $\widetilde{X}$, cf. \S\ref{subsec:homotopy}, and $G_\varepsilon$ is some operator of degree -2. Therefore we obtain by cyclicity of the flat trace
\begin{equation}\label{eq:variationintro}
\partial_z|_{z=0} \log \frac{\tau_\vartheta(\nabla(z))}{\tau_{\frak{e}_\vartheta,\frak{o}}(\nabla(z))} = \strf \alpha R_\varepsilon - \int_e \tr \alpha = 0,
\end{equation}
where the last equality comes from differential topology arguments. Using the analytical structure of the representation variety, we may deduce from (\ref{eq:variationintro}) the claim of Theorem \ref{thm:main2}. 
Theorem~\ref{thm:main3} then follows from the invariance of the dynamical torsion under small perturbations of the flow, the fact that $\tau_{\frak{e,o}}(\rho) = \tau_{\frak{e',o}}(\rho) \langle \det \rho, h \rangle$ for any other Euler structure $\frak{e}'$, where $h \in H_1(M,\mathbb{Z})$ satisfies $\frak{e} = \frak{e}' + h$ (we have that $H_1(M, \mathbb{Z})$ acts freely and transitively on the set of Euler structures, cf. \S\ref{sec:turaevtorsion}), and the fact that, in our notations, $\frak{e}_\eta - \frak{e}_\vartheta = \cs(X_\vartheta,X_\eta)$ for any other contact form $\eta$.

\hfill
\\ 

\subsection{Related works} Some analogs of our dynamical torsion were introduced by
Burghelea--Haller~\cite{burghelea2008torsion} for vector fields which admit a Lyapunov closed $1$--form generalizing previous works by
Hutchings \cite{hutchings}, Hutchings--Lee  \cite{hutchings1999circle1,hutchings1999circle} dealing with Morse--Novikov flows. 
In that case, the dynamical torsion depends on a choice of Euler structure and is a partially defined function on $\Rep_{\mathrm{ac}}(M,d)$; if $d=1$, it is shown in \cite{burghelea2008dynamics} that it extends to a rational map on the Zariski closure of $\Rep_{\mathrm{ac}}(M,1)$ which coincides, up to sign, with Turaev's refined combinatorial torsion (for the same choice of Euler structure). This follows from previous works of
Hutchings--Lee \cite{hutchings1999circle1,hutchings1999circle} who introduced some topological invariant involving circle-valued Morse functions.
In both works, the considered object has the form
$$\text{ Dynamical zeta function} (0)~ \times \text{ Correction term }$$
where the correction term is the torsion of some finite dimensional complex whose chains are generated by the critical points of the vector field.
The chosen Euler structure gives a distinguished basis of the complex and thus a well defined torsion. This is one of the main differences with our work since in the Anosov case, there are no such choices of distinguished currents in $C^\bullet$. However, the chirality operator allows us to overcome this problem as described above.

We also would like to mention some interesting related works of Rumin--Seshadri~\cite{Rumin} where 
they relate
some dynamical zeta function involving the Reeb flow and some analytic contact torsion on $3$--dimensional Seifert CR manifolds. 

Finally, very recently, Spilioti \cite{spilioti2020twisted} and M\"uller \cite{muller2020on} were able to compare the Ruelle zeta function for odd dimensional compact hyperbolic manifolds with some of the complex valued torsions mentioned above.
%
\\Ê
%

\subsection{Plan of the paper.} The paper is organized as follows. In \S\ref{sec:torsion}, we give some preliminaries about torsion of finite dimensional complexes computed with respect to a chirality operator. In \S\ref{sec:geometry}, we present our geometrical setting and conventions. In \S\ref{sec:policott}, we introduce \textcolor{black}{Pollicott}-Ruelle resonances. In \S\ref{sec:deftors}, we compute the refined torsion of a space of generalized eigenvectors for nonzero resonances and we define the dynamical torsion. In \S\ref{sec:invariance}, we prove that our torsion is \textcolor{black}{unsensitive} to small perturbations of the dynamics. In \S\ref{sec:variationconnexion}, we compute the variation of our torsion with respect to the \textcolor{black}{connection}. In \S\ref{sec:turaevtorsion}, we introduce
Euler structures which are some topological tools used to fix ambiguities of the refined torsion. 
In \S\ref{sec:turaevvariation}, we introduce the refined combinatorial torsion of Turaev using Morse theory and we compute its variation with respect to the connection. We finally compare it to the dynamical torsion in \S\ref{sec:comparison}. 
\vspace{-5pt}
{\small
\subsection*{Acknowledgements.}

We warmly thank Nalini Anantharaman, Yannick Bonthonneau, Mihajlo Ceki\'c, Alexis Drouot, Semyon Dyatlov, Malo J\'ez\'equel, 
Thibault Lefeuvre, Julien March\'e, Marco Mazzucchelli, Claude Roger, Nicolas Vichery, Jean Yves Welschinger, Steve Zelditch, for asking questions about this work or for interesting discussions related to the paper. 
Particular thanks are due to
Colin Guillarmou who went through the whole paper, 
helped us correct many errors and is always 
a source of inspiration. 
We thank the organizers of 
the microlocal analysis 
program in MSRI for the invitation to speak about our result.
N.V.D is very grateful to Gabriel Rivi\`ere for his friendship, many inspiring maths discussions, his many advices and 
for the series of works
which made the present paper possible. 
Finally, N.V.D acknowledges the incredible patience and love of his wife and daughter, who created the right atmosphere at home  
which made this possible.
}

\section{Torsion of finite dimensional complexes}\label{sec:torsion}
We recall the definition of the refined torsion of a finite dimensional acyclic complex computed with respect to a chirality operator, following \cite{braverman2007refined}. Then we compute the variation of the torsion of such a complex when the differential is perturbed.

\subsection{The determinant line of a complex}\label{subsec:determinantline} For a non zero complex vector space $V$, the determinant line of $V$ is the line defined by $\det(V) = \bigwedge^{\dim V} V$. We declare the determinant line of the trivial vector space $\{0\}$ to be $\mathbb{C}$. If $L$ is a 1-dimensional vector space, we will denote by $L^{-1}$ its dual line.
Any basis $(v_1,\dots,v_n)$ of $V$ defines a nonzero element
$v_1\wedge \dots\wedge v_n\in \det(V)$. Thus elements of the determinant line of $\det(V)$ should be thought of as equivalence classes of oriented basis of $V$.

Let
$$(C^\bullet, \partial) : 0 \overset{\partial}{\longrightarrow} C^0 \overset{\partial}{\longrightarrow} C^1 \overset{\partial}{\longrightarrow} \cdots \overset{\partial}{\longrightarrow} C^n \overset{\partial}{\longrightarrow} 0$$
be a finite dimensional complex, i.e. $\dim C^j < \infty$ for all $j=0, \dots, n.$ 
We define the \emph{determinant line} of the complex $C^\bullet$ 
by
$$
\det(C^\bullet) = \bigotimes_{j=0}^n \det(C^j)^{(-1)^j}.$$

Let $H^\bullet(\partial)$ be the cohomology of $(C^\bullet, \partial)$, that is
$$
H^\bullet(\partial) = \bigoplus_{j=0}^n H^j(\partial), \quad H^j(\partial) = \frac{\ker(\partial : C^j \to C^{j+1})}{\mathrm{ran} (\partial : C^{j-1}\to C^j)}.
$$
We will say that the complex $(C^\bullet, \partial)$ is acyclic if $H^\bullet(\partial) = 0$. In that case, $\det H^\bullet(\partial)$ is canonically isomorphic to $\mathbb{C}$.


It remains to define the fusion homomorphism that we will later need to 
define the torsion of a finite dimensional based complex~\cite[\S2.3]{farber2000poincare}.
For any finite dimensional vector spaces $V_1, \dots, V_r$, we have a fusion isomorphism
$$
\mu_{V_1, \dots, V_r} : \det(V_1) \otimes \cdots \otimes \det(V_r) \to \det(V_1 \oplus \cdots \oplus V_r)
$$
defined by
$$
\mu_{V_1, \dots, V_r} \Bigl( v_1^1 \wedge \cdots \wedge v_1^{m_1} \otimes \cdots \otimes v_r^1 \wedge \cdots \wedge v_r^{m_r} \Bigr) = v_1^1 \wedge \cdots \wedge v_1^{m_1} \wedge \cdots \wedge v_r^1 \wedge \cdots \wedge v_r^{m_r},
$$
where $m_j = \dim V_j$ for $j \in \{1, \dots, r\}$.

\subsection{Torsion of finite dimensional acyclic complexes.}\label{subsec:torsionfinitedim}
In the present paper, we want to think of torsion of finite dimensional
{acyclic} complexes as
a map $\varphi_{C^\bullet} $ from the determinant line of the complex to $\mathbb{C}$.
We have a canonical isomorphism 
\begin{equation}\label{eq:isocoh}
\varphi_{C^\bullet} : \det(C^\bullet) \overset{\sim}{\longrightarrow} 
\mathbb{C},
\end{equation}
defined as follows. Fix a decomposition
$$
C^j = B^j  \oplus A^j, \quad j=0, \dots, n,
$$
with $B^j = \ker(\partial) \cap C^j$ and $B^j = \partial(A^{j-1}) = \partial(C^{j-1})$ for every $j$. Then $\partial|_{A^j} : A^j \to B^{j+1}$ is an isomorphism for every $j$. 

Fix non zero elements $c_j \in \det C^j$ and $a_j \in \det A^j$ for any $j$.
Let $\partial(a_j) \in \det B^{j+1}$ denote the image of $a_j$ under the isomorphism $\det A^j \to \det B^{j+1}$ induced by the isomorphism $\partial|_{A^j} : A^j \to B^{j+1}$. Then for each $j = 0, \dots, n$, there exists a unique 
$\lambda_j \in \mathbb{C}$ such that
$$
c_j =\lambda_j \mu_{B^j, A^j}\Bigl(\partial(a_{j-1})  \otimes a_j\Bigr),
$$
where $\mu_{B^j, A^j}$ is the fusion isomorphism defined in \S\ref{subsec:determinantline}. Then define the isomorphism $\varphi_{C^\bullet}$ by
$$
\varphi_{C^\bullet} ~: ~ c_0 \otimes c_1^{-1} \otimes \cdots \otimes c_n^{(-1)^n} \mapsto  (-1)^{N(C^\bullet)} \prod_{j=0}^n\lambda_j^{(-1)^j} \in \mathbb{C},
$$
where
$$
N(C^\bullet) = \frac{1}{2} \sum_{j=0}^n \dim A^j \left( \dim A^j + (-1)^{j+1}\right).
$$
One easily shows that $\varphi_{C^\bullet}$ is 
independent of the choices of $a_j$ \cite[Lemma 1.3]{turaev2001introduction}. The number $\tau(C^\bullet, c) = \varphi_{C^\bullet}(c)$ is called the \textit{refined torsion} of $(C^\bullet, \partial)$ with respect to the element $c$.

The torsion will depend on the choices of $c_j \in \det C^j$. Here the sign convention (that is, the choice of the prefactor $(-1)^{N(C^\bullet)}$ in the definition of $\varphi_{C^\bullet}$) follows Braverman--Kappeler \cite[\S2]{braverman2007refined} and is consistent with Nicolaescu~\cite[\S1]{nicolaescu2003reidemeister}. This prefactor was introduced by Turaev and differs from \cite{turaev1986reidemeister}. See \cite{{nicolaescu2003reidemeister}} for the motivation for the choice of sign.

\begin{rem}\label{rem:notacyclic0}
If the complex $(C^\bullet, \partial)$ is not acyclic, we can still define a torsion $\tau(C^\bullet,c)$, which is this time an element of the determinant line $\det H^\bullet(\partial)$, cf. \cite[\S2.4]{braverman2007refined}.
\end{rem}
 
\subsection{Torsion with respect to a chirality operator}
We saw above that torsion depends on the choice of an element of the determinant line.
A way to fix the value of the torsion 
without choosing an explicit basis
is to use a chirality operator as in \cite{braverman2007refined}.
 Take $n = 2r+1$ an odd integer and consider a complex $(C^\bullet, \partial)$ of length $n$. We will call a \emph{chirality} operator an operator $\Gamma : C^\bullet \to C^\bullet$ such that $\Gamma^2 = \id_{C^\bullet}$, and 
$$\Gamma(C^j) = C^{n-j}, \quad j=0,\dots,n.$$
$\Gamma$ induces isomorphisms $\det(C^j) \to \det(C^{n-j})$ that 
we will still denote by $\Gamma$. If $\ell \in L$ is a non zero element of a complex line, we will denote by $\ell^{-1} \in L^{-1}$ the unique element such that $\ell^{-1}(\ell) = 1$. Fix non zero elements $c_j \in \det(C^j)$ for $j \in \{0, \dots, r\}$ and define
$$c_\Gamma = (-1)^{m(C^\bullet)}c_0 \otimes c_1^{-1} \otimes \cdots \otimes c_r^{(-1)^r} \otimes (\Gamma c_r)^{(-1)^{r+1}} \otimes (\Gamma c_{r-1})^{(-1)^{r}} \otimes \cdots \otimes (\Gamma c_0)^{-1},$$
where 
$$m(C^\bullet) = \frac{1}{2} \sum_{j=0}^r \dim C^j\left(\dim C^j + (-1)^{r+j}\right).$$
\begin{defi}\label{def:torsion}
The element $c_\Gamma$ is independent of the choices of $c_j$ for $j \in \{0, \dots, r\}$; the \emph{refined torsion of $(C^\bullet, \partial)$ with respect to $\Gamma$} is the element
$$\tau(C^\bullet, \Gamma) = \tau(C^\bullet, c_\Gamma).$$
\end{defi}
We also have the following result which is \cite[Lemma 4.7]{braverman2007refined} in the acyclic case 
about the multiplicativity of torsion.
\begin{prop}\label{prop:multtorsion}
Let $(C^\bullet, \partial)$ and $(\tilde{C}^\bullet, \tilde\partial)$ be two acyclic complexes of same length endowed with two chirality operators $\Gamma$ and $\tilde\Gamma$. Then
$$
\tau(C^\bullet\oplus \tilde C^\bullet, \Gamma \oplus \tilde \Gamma) = \tau(C^\bullet, \Gamma) \tau(\tilde C^\bullet, \tilde \Gamma).
$$
\end{prop}

\subsection{Computation of the torsion with the contact signature operator}\label{subsec:signature}
Let 
$$B = \Gamma \partial + \partial \Gamma : C^\bullet \to C^\bullet.$$
$B$ is called the \emph{signature operator}. 
Let $B_+ = \Gamma \partial$ and $B_- = \partial \Gamma$. Denote 
$$C^j_{\pm} = C^j \cap \ker(B_\mp), \quad j=0, \dots, n.$$ We have that $B_{\pm}$ preserves $C^\bullet_{\pm}$. Note that $B_+(C^{j}_+) \subset C^{n-j-1}_+$, so that $B_+(C^j_+ \oplus C^{n-j-1}_+) \subset C^j_+ \oplus C^{n-j-1}_+$. Note that if $B$ is invertible on $C^\bullet$, $B_+$ is invertible on $C^\bullet_+$. If $B$ is invertible, we can compute the refined torsion of $(C^\bullet, \partial)$
using the following
\begin{prop}\cite[Proposition 5.6]{braverman2007refined}\label{prop:formulatorsion}
Assume that $B$ is invertible. Then $(C^\bullet, \partial)$ is acyclic so that $\det(H^\bullet(\partial))$ is canonically isomorphic to $\mathbb{C}$. Moreover,
$$\tau(C^\bullet, \Gamma) = (-1)^{r\dim C^r_+} \det\left(\Gamma \partial|_{C^r_+}\right)^{(-1)^{r}} \prod_{j=0}^{r-1} \det\left(\Gamma \partial|_{C^j_+ \oplus C^{n-j-1}_+}\right)^{(-1)^{j}}.$$
\end{prop}

\subsection{Super traces and determinants}\label{subsec:graded}
Let $V^\bullet = \bigoplus_{j=0}^p V^j$ is a graded finite dimensional vector space and $A : V^\bullet \to V^\bullet$ be a degree preserving linear map. We define the \textit{super trace} and the \textit{super determinant} of $A$ by
$$ 
\begin{aligned}
\str{V^\bullet} A = \sum_{j=0}^p (-1)^j \tr_{V^j} A, \quad \quad
\sdet{V^\bullet} A = \prod_{j=0}^p ({\det}_{V^j} A)^{(-1)^j}.
\end{aligned}
$$
We also define the \textit{graded trace} and the \textit{graded determinant} of $A$ by
$$ 
\begin{aligned}
\grtr{V^\bullet} A = \sum_{j=0}^p (-1)^jj \tr_{V^j} A, \quad \quad
\grdet{V^\bullet} A = \prod_{j=0}^p ({\det}_{V^j} A)^{(-1)^jj}.
\end{aligned}
$$

\subsection{Analytic families of differentials}\label{subsec:analyticfamiliesofdifferentials}

The goal of the present subsection is to give a variation formula for the torsion of a finite dimensional complex when we vary
the differential. This formula plays a crucial role in the variation formula of the dynamical torsion, when the representation is perturbed. Indeed, we split the dynamical torsion as the product of the torsion $\tau\left(C^\bullet(\vartheta,\rho),\Gamma_{\vartheta} \right)$ of some finite dimensional space of Ruelle resonant states and a renormalized value at $s=0$ of the dynamical zeta function $\zeta_{X,\rho}(s)$.
Then the following  formula allows us to deal with the variation of $\tau\left(C^\bullet(\vartheta,\rho), \Gamma_{\vartheta} \right)$.

Let $(C^\bullet, \partial)$ be an acyclic finite dimensional complex of finite odd length $n$. If $S : C^\bullet : C^\bullet$ is a linear operator, we will say that it is of degree $s$ if $S(C^k) \subset C^{k+s}$ for any $k$. If $S$ and $T$ are two operators on $C^\bullet$ of degrees $s$ et $t$ respectively then the supercommutator of $S$ and $T$ by
$$
[S,T] = ST - (-1)^{st}TS.
$$
Cyclicity of the usual trace gives $\str{C^\bullet}[S,T] = 0$ for any $S,T$.

Let $U$ be a neighborhood of the origin in the complex plane and $\partial(z),$ ~$z\in U$, be a family of acyclic differentials on $C^\bullet$ which is complex differentiable at $z=0$, that is,
\begin{equation}\label{eq:partialz}
\partial(z) = \partial + za + o(z)
\end{equation}
for some operator $a : C^\bullet \to C^\bullet$ of degree $1$. Note that $\partial(z) \circ \partial(z) = 0$ 
implies that the supercommutator
\begin{equation}\label{eq:apartial}
[\partial, a] = \partial a + a \partial = 0.
\end{equation}
We will denote by $C^\bullet(z)$ the complex $(C^\bullet, \partial(z))$. Finally let $k : C^\bullet \to C^\bullet$ be a cochain contraction, that is a linear map of degree $1$ such that
\begin{equation}\label{eq:Kpartiall}
\partial k + k \partial = \id_{C^\bullet}.
\end{equation}
The existence of such map is ensured by the acyclicity of $(C^\bullet, \partial)$.

\begin{lemm}\label{lem:5.1}
In the above notations, for any chirality operator $\Gamma$ on $C^\bullet$, the
map $z \mapsto \tau(C^\bullet(z), \Gamma)$ is complex differentiable at $z=0$ and
$$
\left.\frac{\dd}{\dd z}\right|_{z=0} \log \tau(C^\bullet(z), \Gamma) = - \str{C^\bullet}(ak).
$$
\end{lemm}
Note that this implies in particular that $\str{C^\bullet}(ak)$ does not depend on the chosen cochain contraction $k$. This is expected since if $k'$ is another cochain contraction,
$$
[\partial, akk'] = \partial a k k' + a k k' \partial = a(k-k')
$$
by (\ref{eq:apartial}), and the supertrace of a supercommutator vanishes.

\begin{proof}
First note that for non zero elements $c, c' \in \det C^\bullet$, we have
\begin{equation}\label{eq:changeofbasis}
\tau(C^\bullet(z), c) = [c:c'] \cdot \tau(C^\bullet(z), c'),
\end{equation}
where $[c:c'] \in \mathbb{C}$ satisfies $c = [c:c'] \cdot c'$. 

For every $j=0, \dots, n,$ fix a decomposition
$$
C^j = A^j \oplus B^j,
$$
where $B^j = \ker \partial \cap C^j$ and $A^j$ is any complementary of $B^j$ in $C^j$. Fix some basis $a^1_j, \dots, a^{\ell_j}_j$ of $A^j$; then $\partial a^1_j, \dots, \partial a^{\ell_j}_j$ is a basis of $B^{j+1}$ by acyclicity of $(C^\bullet, \partial)$. Now let
$$
c_j = a^1_j \wedge \cdots \wedge a^{\ell_j}_j \wedge \partial a^1_{j-1} \wedge \cdots \wedge \partial a^{\ell_{j-1}}_{j-1} \in \det C^j,
$$
and 
$$
c = c_0 \otimes (c_1)^{-1} \otimes c_2 \otimes \cdots \otimes (c_n)^{(-1)^n} \in \det C^\bullet.
$$
Now by definition of the refined torsion, we have for $|z|$ small enough
\begin{equation}\label{eq:taucprod}
\tau(C^\bullet(z), c) = \pm \prod_{j=0}^n \det\bigl(A_j(z)\bigr)^{(-1)^{j+1}}
\end{equation}
where the sign $\pm$ is independent of $z$ and $A_j(z)$ is the matrix sending the basis
$$
a^1_j, \dots, a^{\ell_j}_j, \partial a^1_{j-1}, \dots, \partial a^{\ell_{j-1}}_{j-1}
$$
to the basis
$$
a^1_j, \dots, a^{\ell_j}_j, \partial(z) a^1_{j-1}, \dots, \partial(z) a^{\ell_{j-1}}_{j-1}
$$
(which is indeed a basis of $C^j$ for $|z|$ small enough). Let $k : C^\bullet \to C^\bullet$ of degree $-1$ defined by
$$
k \partial a_j^m = a_j^m, \quad k a_j^m = 0,
$$
for every $j$ and $m \in \{0, \dots, \ell_j\}.$ Then $k \partial + \partial k = \id_{C^\bullet}$ and
$$
\det A_j(z) = {\det}_{\partial B^{j-1} \oplus B^j} \bigl(\partial(z) k \oplus \id \bigr).
$$
Now (\ref{eq:partialz}) and (\ref{eq:taucprod}) imply the desired result, because $\tau(C^\bullet(z), \Gamma) = [c_\Gamma :c] \cdot \tau(C^\bullet(z), c)$ by (\ref{eq:changeofbasis}).
\end{proof}

\section{Geometrical setting and notations}\label{sec:geometry}

We introduce here our geometrical conventions and notations. In particular, we adopt the formalism of Harvey--Polking~\cite{harvey1979fundamental} which will be convenient to compute flat traces and relate the variation of the Ruelle zeta function with topological objects.

\subsection{Twisted cohomology}
We consider $M$ an oriented closed connected manifold of odd dimension $n = 2r+1$. Let $E \to M$ be a flat vector bundle over $M$ of rank $d\geq 1$. For $k\in\{0, \dots, n\}$, we will denote the bundle $\Lambda^k T^*M$ by $\Lambda^k$ for simplicity. We will denote 
by $\Omega^k(M,E) = \Cinf(M, \Lambda^k \otimes E)$ the space of $E$ valued $k$-forms. 
We set
$$\Omega^\bullet(M,E) = \bigoplus_{k=0}^n \Omega^k(M,E).$$
Let $\nabla$ be a flat connection on $E$. We view the connection as a degree $1$ operator (as an operator of the graded vector space $\Omega^\bullet(M,E)$)
$$
\nabla : \Omega^k(M,E) \to \Omega^{k+1}(M,E), \quad k = 0,\dots, n.
$$

The flatness of the connection reads $\nabla^2 = 0$ and thus we obtain a cochain complex $\bigl(\Omega^\bullet(M,E), \nabla\bigr)$.  We will assume that the connection $\nabla$ is acyclic, that is, the complex $\bigl(\Omega^\bullet(M,E), \nabla\bigr)$ is acyclic, or equivalently, the cohomology groups 
$$H^k(M,\nabla) = \frac{\Bigl\{u \in \Omega^k(M,E) \ : \ \nabla u = 0\Bigr\}}{\Bigl\{ \nabla v  \ : \ v \in \Omega^{k-1}(M,E)\Bigr\}}, \quad k=0, \dots, n,$$
are trivial.

\subsection{Currents and Schwartz kernels}\label{subsec:currents}
Let
$$\mathcal{D}^{'\bullet}(M,E) = \bigoplus_{k=0}^n \mathcal{D}^{'}(M, \Lambda^k \otimes E)$$
the space of $E$-valued currents. 
Let $E^\vee$ denote the dual bundle of $E$. 
We will identify $\mathcal{D}^{'k}(M,E)$ and the topological 
dual of $\Omega^{n-k}(M,E^\vee)$ via the non degenerate bilinear pairing
$$
\langle \alpha, \beta \rangle = \int_M \alpha \wedge \beta, \quad \alpha \in \Omega^k(M,E), \quad \beta \in \Omega^{n-k}(M,E^\vee),
$$ 
where $\wedge$ is the usual wedge product between $E$-valued forms and $E^\vee$-valued forms. 

A continuous linear operator $G~: \Omega^\bullet(M,E) \to \mathcal{D}^{'\bullet}(M, E)$ is called homogeneous if for some $p \in \mathbb{Z}$, we have $G\Bigl(\Omega^k(M,E)\Bigr) \subset \mathcal{D}^{'k+p}(M,E)$ for every $k=0, \dots, n$;
the number $p$ is called the degree of $G$ and is denoted by $\deg G$. 
In that case, the Schwartz kernel theorem gives us a twisted current $\mathcal{G}\in \mathcal{D}^{'n+p}(M \times M, \pi_1^*E^\vee \otimes \pi_2^*E)$ satisfying
$$
\langle G u, v \rangle_M = \langle \G, \pi_1^* u \wedge \pi_2^*v \rangle_{M \times M}, \quad u \in \Omega^k(M,E), \quad v \in \Omega^{n-k-p}(M,E^\vee),
$$
where $\pi_1$ and $\pi_2$ are the projections of $M \times M$ onto its first and second factors
respectively.

\subsection{Integration currents}\label{subsec:integrationcurrents}
Let $N$ be an oriented submanifold of $M$ of dimension $d$, possibly with boundary. The associated integration current $[N] \in \mathcal{D}^{'n-d}(M)$ is given by
$$
\bigl \langle [N], \omega \bigr\rangle = \int_N i_N^* \omega, \quad \omega \in \Omega^d(M),
$$
where $i_N : N \to M$ is the inclusion. We have classically
\begin{equation}\label{eq:boundary}
\dd [N] = (-1)^{n-d+1}[\partial N].
\end{equation}

For $f \in \mathrm{Diff}(M)$, we will set $\mathrm{Gr}(f) = \{(f(x),x), ~x \in M\}$ the graph of $f$. Note that $\mathrm{Gr}(f)$ is a $n$-dimensional submanifold of $M\times M$ which is canonically oriented since $M$ is. Therefore, we can consider the integration current over $\mathrm{Gr}(f)$. By definition, we have for any $\alpha, \beta \in \Omega^\bullet(M)$
$$
\bigl \langle [\mathrm{Gr}(f)], \pi_1^*\alpha \wedge \pi_2^*\beta\bigr \rangle = \int_M f^*\alpha \wedge \beta.
$$
In particular, $[\mathrm{Gr}(f)]$ is the Schwartz kernel of $f^*:\Omega^\bullet(M) \to \Omega^\bullet(M)$.

\subsection{Flat traces}\label{subsec:flattrace}
Let $G : \Omega^\bullet(M,E) \to \mathcal{D}^{'\bullet}(M,E)$ be an operator of degree $0$. We denote its Schwartz kernel by $\mathcal{G}$ and we define
$$
\WF'(\mathcal{G}) = \bigl\{(x,y,\xi, \eta),~ (x,y, \xi, -\eta) \in \WF(\mathcal{G}) \bigr\} \subset T^*(M\times M),
$$
where $\WF$ denotes the classical H\"ormander wavefront set, cf \cite[\S8]{hor1}. We will also use the notation $\WF(G) = \WF(\mathcal{G})$ and $\WF'(G) = \WF'(\mathcal{G})$. Assume that 
\begin{equation}\label{eq:wfcondition}
\WF'(\mathcal{G}) \cap \Delta(T^*M) = \emptyset, \quad \Delta(T^*M) = \{(x,x, \xi, \xi),~ (x,\xi) \in T^*M\}.
\end{equation}
Let $\iota : M \to M \times M, x \mapsto (x,x)$ be the diagonal inclusion. Then by \cite[Theorem 8.2.4]{hor1} the pull back $\iota^*\mathcal{G} \in \mathcal{D}^{'n}(M, E^\vee \otimes E)$ is well defined and we define the \textit{super flat trace} of $G$ by
$$
\strf G = \langle \tr \iota^* \mathcal{G}, 1 \rangle,
$$
where $\tr$ denotes the trace on $E^\vee \otimes E$. We will also use the notation
$$
\grtrf G = \strf NG,
$$
where $N : \Omega^\bullet(M,E) \to \Omega^\bullet(M,E)$ is the number operator, that is, $NÊ\omega = k \omega$ for every $\omega \in \Omega^k(M,E)$.

The notation $\strf$ is motivated by the following. Let $A : \Cinf(M,F) \to \mathcal{D}^{'}(M,F)$ be an operator acting on sections of a vector bundle $F$. If $A$ satisfies (\ref{eq:wfcondition}), we can also define a flat trace $\tr^\flat A$ as in \cite[\S2.4]{dyatlov2013dynamical}. Now if $G : \Omega^\bullet(M,E) \to \mathcal{D}^{'\bullet}(M,E)$  is an operator of degree $0$, it gives rise to an operator $G_k : \Cinf(M,F_k) \to \mathcal{D}^{'}(M, F_k)$ for each $k=0, \dots, n$, where $F_k = \Lambda^k \otimes E$. Then the link between the two notions of flat trace mentioned above is given by
$$
\strf G = \sum_{k=0}^n (-1)^k \tr^\flat G_k.
$$

If $\Gamma \subset T^*M$ is a closed conical subset, we let
\begin{equation}\label{eq:d'gamma}
\mathcal{D}^{'\bullet}_\Gamma(M,E) = \left\{ u \in \mathcal{D}^{'\bullet}(M,E), \WF(u) \subset \Gamma \right\}
\end{equation}
be the space of $E$-valued current whose wavefront set is contained in $\Gamma$, endowed with its usual topology, cf. \cite[\S8]{hor1}. If $\Gamma$ is a closed conical subset of $T^*(M\times M)$ not intersecting the conormal to the diagonal
$$N^*\Delta(T^*M) = \{(x,x,\xi,-\xi), ~(x,\xi) \in T^*M\},$$
then the flat trace is continuous as a map $\mathcal{D}_\Gamma^{'\bullet}(M \times M, \pi_1^*E^\vee \otimes \pi_2^*E) \to \mathbb{R}$.

\subsection{Cyclicity of the flat trace}
Let $G, H : \Omega^\bullet(M,E) \to \mathcal{D}^{'\bullet}(M,E)$ be two homogeneous operators. We denote by $\G, \H$ their respective kernels. If $\Gamma \subset T^*(M\times M)$ is a closed conical subset, we define
$$
\Gamma^{(1)} = \{(y,\eta),~\exists x \in M, (x,y,0,\eta) \in \Gamma\}, \quad \Gamma^{(2)} = \{(y, \eta), \exists x \in M, (x,y,-\eta, 0) \in \Gamma\}.
$$
Then under the assumption
$$
\WF(\G)^{(2)} \cap \WF(\H)^{(1)} = \emptyset,
$$
the operator $F = G\circ H$ is well defined by \cite[Theorem 8.2.14]{hor1}
and its Schwartz kernel $\mathcal{F}$ satisfies the wave front set estimate~:
$$ \WF\left(\mathcal{F} \right) \subset \bigl\{ (x,y, \xi,\eta)~|~ \exists (z, \zeta), (x,z,\xi,\zeta)\in \WF'\left(\G \right) \text{ and }(z,y, \zeta,\eta)\in 
 \WF \left(\H \right) \bigr\} .$$ If both compositions $G \circ H$ and $H \circ G$ are defined, we will denote by
$$
[G,H] = G \circ H - (-1)^{\deg G \deg H}H \circ G
$$
the graded commutator of $G$ and $H$. We have the following
\begin{prop}\label{p:flattracesupercomm}
Let $G,H$ be two homogeneous operators with $\deg G + \deg H = 0$ and such that both compositions $G \circ H$ and $H \circ G$ are defined and satisfy the bound (\ref{eq:wfcondition}). Then we have
$$
\strf ~Ê[G,H] = 0.
$$
\end{prop}
The above result follows from the cyclicity of the $L^2$-trace, the approximation result \cite[Lemma 2.8]{dyatlov2013dynamical}, the relation
$$
\strf ~Ê[G,H] = \tr^\flat \left[(-1)^NF, G\right],
$$
where $N$ is the number operator and $\tr^\flat$ is the flat trace with the convention from \cite{dyatlov2013dynamical}, see \S\ref{subsec:flattrace}, and the fact that the map $(G,H) \mapsto G \circ H$ is continuous
$$
\mathcal{D}^{'\bullet}_{\Gamma}(M\times M,\pi_1^*E^\vee \otimes \pi_2^*E) \times \mathcal{D}^{'\bullet}_{\widetilde{\Gamma}}(M\times M,\pi_1^*E^\vee \otimes \pi_2^*E) \to \mathcal{D}^{'\bullet}_{\Upsilon}(M\times M,\pi_1^*E^\vee \otimes \pi_2^*E)
$$
for any closed conical subsets $\Gamma, \widetilde{\Gamma} \subset T^*(M\times M)$ such that $\Gamma^{(2)}\cap {\widetilde{\Gamma}}^{(1)} = \emptyset$, and where $\Upsilon$ is a closed conical subset given in \cite[8.2.14]{hor1}.

\subsection{Perturbation of holonomy}
Let $\gamma : [0,1] \to M$ 
be a smooth curve and $\alpha \in \Omega^1(M,\mathrm{End}(E))$. Let $P_t$ (resp. $\tilde{P}_t$) be the parallel transport $E_{\gamma(0)} \to E_{\gamma(t)}$ of $\nabla$ (resp. $\tilde{\nabla} = \nabla + \alpha$) along $\gamma|_{[0,t]}$. Then
\begin{equation}\label{lem:holonomy}
\tilde{P}_t = P_t \exp\left(-\int_0^t P_{-\tau} \alpha(\dot{\gamma}(\tau)) P_\tau \dd \tau \right).
\end{equation}
The above formula will be useful in some occasion. For simplicity, we will denote for any $A \in \Cinf(M,\mathrm{End}(E))$
$$
\int_\gamma A = \int_0^t P_{-\tau} A(\gamma(\tau)) P_\tau \dd \tau \in \mathrm{End} \left(E_{\gamma(0)}\right)
$$
so that $\tilde{P}_1 = P_1 \exp\left( - \int_\gamma \alpha(X) \right)$.

\section{Pollicott-Ruelle resonances}\label{sec:policott}

\subsection{Anosov dynamics}\label{subsec:anosov}
Let $X$ be a smooth vector field on $M$ and denote by $\varphi^t$ its flow. We will assume that $X$ generates an Anosov flow, that is, there exists a splitting of the tangent space $T_xM$ at every $x \in M$
$$
T_xM = \mathbb{R} X(x) \oplus E_s(x) \oplus E_u(x),
$$
where $E_u(x), E_s(x)$ are subspaces of $T_xM$ depending continuously on $x$ and invariant by the flot $\varphi^t$, such that for some constants $C, \nu > 0$ and some smooth metric $|\cdot |$ on $TM$ one has
$$
\begin{aligned}
&|(\dd \varphi^t)_x v_s| \leq C \e^{-\nu t} |v_s|, \quad t \geq0, \quad v_s \in E_s(x), \\
&|(\dd \varphi^t)_x v_u| \leq C \e^{-\nu |t|} |v_u|, \quad t \leq0, \quad v_u \in E_u(x).
\end{aligned}
$$
We will use the dual decomposition $T^*M = E_0^* \oplus E_u^* \oplus E_s^*$ where $E_0^*, E_u^*$ and $E_s^*$ are defined by
\begin{equation}\label{eq:dualdec}
E_0^*(E_s \oplus E_u) = 0, \quad E_s^*(E_0 \oplus E_s) = 0, \quad E_u^*(E_0 \oplus E_u) = 0.
\end{equation}

\subsection{Pollicott-Ruelle resonances}\label{subsec:policott}
Let $\iota_X$ denote the interior product with $X$ and
$$\Lie_X^\nabla = \nabla \iota_X + \iota_X \nabla : \Omega^\bullet(M,E) \to \Omega^\bullet(M,E)$$
be the Lie derivative along $X$ acting on $E$-valued forms. Locally, the action of $\Lien$ is given by the following. Take $U$ a domain of a chart and write $\nabla = \dd + A$ where $A \in \Omega^1(M, \mathrm{End}(E))$. Take $w_1, \dots, w_\ell$ (resp. $e_1, \dots, e_d$) some local basis of $\Lambda^k$ (resp. $E$) on $U$. Then for any $1Ê\leq i \leq \ell$ and $1 \leq j \leq d$,
$$
\Lien \left(f w_i \otimes e_j \right) = (Xf) w_i \otimes e_j + f (\Lie_X w_i) \otimes e_j + f w_i \otimes A(X) e_j, \quad f \in \Cinf(U),
$$
where $\Lie_X$ is the standard Lie derivative acting on forms. In particular, $\Lien$ is a differential operator of order $1$ acting on sections of the bundle $\Lambda^\bullet T^*M \otimes E$, whose principal part is diagonal and given by $X$.

Denote by $\Phi^t_k$ the induced flow on the vector bundle $\Lambda^kT^*M \otimes E \to M$, that is, 
$$\Phi^t_k(\beta \otimes v) = {^T}{(\dd \varphi^{t})}^{-1}_x\beta \otimes P^\nabla_t(x)v, \quad x \in M, \quad (\beta,v) \in \Lambda^k (T_x^*M) \times E_x, \quad t \in \rr,$$
where $P^\nabla_t(x)$ is the parallel transport induced by $\nabla$ along the curve $\{\varphi^s(x), \ s \in [0,t]\}$. This induces a map 
$$\e^{t\Lie_X^\nabla} : \Omega^\bullet(M,E) \to \Omega^\bullet(M,E).$$

For $\Re(s)$ big enough, the operator $\Lien + s$ acting on $\Omega^\bullet(M,E)$ is invertible with inverse 
\begin{equation}\label{eq:resolvent}
(\Lien + s)^{-1} = \int_0^\infty \e^{-t\Lien}\e^{-st} \dd t.
\end{equation}
The results of~\cite{faure2011upper} generalize to the flat bundle case as in~\cite[\S3]{dang2017topology} and
the resolvent $\left(\Lien + s\right)^{-1}$, viewed as a family of operators $\Omega^\bullet(M,E) \to \mathcal{D}^{'\bullet}(M,E)$, admits a meromorphic continuation to $s \in \mathbb{C}$ with poles of finite multiplicites; we will still denote by $\left(\Lien + s\right)^{-1}$ this extension. Those poles are the \emph{Pollicott-Ruelle resonances} of $\Lien$, and we will denote this set by $\Res(\Lien)$.

\subsection{Generalized resonant states}
Let $s_0 \in \Res(\Lien)$. By \cite[Proposition 3.3]{dyatlov2013dynamical} we have a Laurent expansion
\begin{equation}\label{eq:laurent}
\bigl(\Lien+s\bigr)^{-1} = Y_{s_0}(s) + \sum_{j=1}^{J(s_0)} (-1)^{j-1}\frac{\bigl(\Lien + s_0\bigr)^{j-1}\Pi_{s_0}}{(s-s_0)^{j}}\end{equation}
where $Y_{s_0}(s)$ is holomorphic near $s = s_0$, and
\begin{equation}\label{eq:projector}
\Pi_{s_0}= \frac{1}{2\pi i} \int_{C_\varepsilon(s_0)} \left(\Lien+s\right)^{-1} \dd s : \Omega^\bullet(M,E) \to \mathcal{D}^{'\bullet}(M,E)
\end{equation}
is an operator of finite rank. Here $C_\varepsilon(s_0) = \{|z-s_0| = \varepsilon\}$ with $\varepsilon > 0$ small enough is a small circle around $s_0$ such that $\Res(\Lien) \cap \{|z-s_0| \leq \varepsilon\}= \{s_0\}$.
Moreover the operators $Y_{s_0}(s)$ and $\Pi_{s_0}$ extend to continuous operators 
\begin{equation}\label{eq:ypi}
Y_{s_0}(s), \Pi_{s_0} : \mathcal{D}^{'\bullet}_{E_u^*}(M,E) \to \mathcal{D}^{'\bullet}_{E_u^*}(M,E).
\end{equation} 
The space
$$C^\bullet(s_0) = \mathrm{ran}(\Pi_{s_0}) \subset \mathcal{D}^{' \bullet}_{E_u^*}(M,E)$$
is called the space of generalized resonant states of $\Lien$ associated to the resonance $s_0$. 
\subsection{The twisted Ruelle zeta function}\label{subsec:theruelle}
Fix a base point $x_\star \in M$ and identify $\pi_1(M)$ with $\pi_1(M, x_\star)$. Let $\mathrm{Per}(X)$ be the set of periodic orbits of $X$. For every $\gamma \in \mathrm{Per}(X)$ we fix some base point $x_\gamma \in \mathrm{Im}(\gamma)$ and an arbitrary path $c_\gamma$ joining $x_\gamma$ to $x_\star$. This path defines an isomorphism $\psi_\gamma : \pi_1(M, x_\gamma) \cong \pi_1(M)$ and we can thus define every $\gamma \in \mathrm{Per}(X)$
$$\rho_\nabla([\gamma]) = \rho_\nabla(\psi_\gamma [\gamma]).$$
The \emph{twisted Ruelle zeta function} associated to the pair $(X,\nabla)$ is defined by
\begin{equation}\label{eq:zeta}
\zeta_{X,\nabla}(s) = \prod_{\gamma \in \mathcal{G}_X} \det\left(\id - \rho_\nabla([\gamma])\e^{-s\ell(\gamma)}\right), \quad \Re(s) > C,
\end{equation}
where $\mathcal{G}_X$ is the set of all primitive closed orbits of $X$ (that is, the closed orbits that generate their class in $\pi_1(M)$), $\ell(\gamma)$ is the length of the orbit $\gamma$ and $C>0$ is some big constant depending on $\rho$ and $X$ satisfying
\begin{equation}\label{eq:boundrhogamma}
\|\rho_{\nabla}([\gamma])\| \leq \exp(C \ell(\gamma)), \quad \gamma \in \mathcal{G}_X,
\end{equation}
for some norm $\|Ê\cdot \|$ on $\mathrm{End}(E_{x_\star})$.

For every closed orbit $\gamma$, we have 
\begin{equation}\label{eq:orunstable}
|\det(I - P_\gamma)| = (-1)^q \det(I - P_\gamma),
\end{equation}
for some $q \in \mathbb{Z}$ not depending on $\gamma$, where $P_\gamma$ is the linearized Poincar\'e return map of $\gamma$, that is $P_\gamma = \dd_x\varphi^{-\ell(\gamma)}|_{E_s(x) \oplus E_u(x)}$ for $x \in \mathrm{Im}(\gamma)$ (if we choose another point in $\mathrm{Im}(\gamma)$, the map will be conjugated to the first one). This condition is always true when $\varphi^t$ is contact, in which case we have $q = \dim E_s$.

Giuletti-Pollicott-Liverani and Dyatlov-Zworski \cite{giulietti2013anosov, dyatlov2013dynamical} showed that $\zeta_{X,\nabla}$ has a meromorphic continuation to $\mathbb{C}$ whose poles and zeros are contained in $\Res(\Lien)$; moreover, the order of $\zeta_{X, \nabla}$ near a resonance $s_0 \in \Res(\Lien)$ is given by
\footnote{Actually, it follows from \cite{dyatlov2013dynamical} that $\displaystyle{m(s_0) = (-1)^q \sum_{k=0}^{n-1} (-1)^k m_k^0(s_0})$, where $m_k^0(s_0)$ is the dimension of $\Pi_{s_0}\left(\Omega^k(M,E) \cap \ker \iota_X \right)$. We can however repeat the arguments using the identity $\displaystyle{\det(\id - P_\gamma) = -\sum_{k=0}^n (-1)^k k \tr \Lambda^k \dd_x \varphi^{-\ell(\gamma)}}$ instead of the identity $\displaystyle{\det(\id - P_\gamma)= \sum_{k=0}^{n-1} (-1)^k \tr \Lambda^k P_\gamma}$ (see \cite[\S2.2]{dyatlov2013dynamical}), and study the action of $\Lien$ on the bundles $\Lambda^k T^*M \otimes E$ rather than its action on the bundles $\bigl(\Lambda^k T^*M \cap \ker \iota_X\bigr) \otimes E$, to obtain (\ref{eq:orderres}).}

\begin{equation}\label{eq:orderres}
m(s_0) = (-1)^{q+1}\sum_{k=0}^n (-1)^kkm_k(s_0),
\end{equation}
where $m_k(s_0)$ is the rank of the spectral projector $\Pi_{s_0}|_{\Omega^k(M,E)}$.




\subsection{Topology of resonant states}\label{subsec:topologyruelle}
Since $\nabla$ commutes with $\Lien$, it induces a differential on the complexes $C^\bullet(s_0)$ for any $s_0 \in \Res(\Lien)$. It is shown in \cite{dang2017topology} that the complexes $\bigl(C^\bullet(s_0), \nabla \bigr)$ are acyclic whenever $s_0 \neq 0$. Moreover, for $s_0 = 0$, the map
$$\Pi_{s_0=0} : \Omega^\bullet(M,\nabla) \longrightarrow C^\bullet(s_0=0)$$
is a quasi-isomorphism, that is, it induces isomorphisms at the level of cohomology groups. Since we assumed $\nabla$ to be acyclic, the complex $\bigl(C^\bullet(s_0=0), \nabla \bigr)$ is also acyclic.

\section{The dynamical torsion of a contact Anosov flow}\label{sec:deftors}
From now on, we will assume that the flow $\varphi^t$ is contact, that is, there exists a smooth one form $\vartheta \in \Omega^1(M)$ such that  $\vartheta \wedge (\dd \vartheta)^r$ is a volume form on $M$, $\iota_X \vartheta = 1$ and $\iota_X \dd \vartheta = 0$. The purpose of this section is to define the dynamical torsion of the pair $(\vartheta,\nabla)$. We first introduce a chirality operator $\Gamma_\vartheta$ acting on $\Omega^\bullet(M,E)$ which is defined thanks to the contact structure. Then the dynamical torsion is a renormalized version of the twisted Ruelle zeta function corrected by the torsion of the finite dimensional space of the generalized resonant states for resonance $s_0 = 0$ computed with respect to $\Gamma_\vartheta$.

This construction was inspired by the work of Braverman-Kappeler on the refined analytic torsion \cite{braverman2007refined}.

\subsection{The chirality operator associated to a contact structure}\label{subsec:chirality}

Let $V_X \to M$ denote the bundle $T^*M \cap \ker \iota_X$.  
Note that for $k \in \{0,\dots, n\}$, we have the decomposition
\begin{equation}\label{eq:decomp}
\Lambda^kT^*M = \Lambda^{k-1}V_X \wedge \vartheta \oplus \Lambda^k V_X.
\end{equation}
Indeed, if $\alpha \in \Lambda^k T^*M$ we may write
$$
\alpha = \underset{\in \Lambda^{k-1}V_X \wedge \vartheta}{\underbrace{(-1)^{k+1} \iota_X \alpha \wedge \vartheta}} + \underset{\in \Lambda^{k}V_X}{\underbrace{\alpha - (-1)^{k+1} \iota_X \alpha \wedge \vartheta.}}
$$
Let us introduce the Lefschetz map
$$
\begin{matrix}
\mathscr{L} :&\Lambda^\bullet V_X&\to &\Lambda^{\bullet+2} V_X \\
 &u &\mapsto &u \wedge \dd \vartheta.
\end{matrix}
$$
Since $\dd \vartheta$ is a symplectic form on $V_X$, the maps $\mathscr{L}^{r-k}$ induce bundle isomorphisms~
\begin{equation}\label{eq:iso}
\mathscr{L}^{r-k} : \Lambda^{k}V_X \overset{\sim}{\longrightarrow} \Lambda^{2r-k}V_X, \quad k = 0, \dots, r,
\end{equation}
see for example \cite[Theorem 16.3]{libermann}.
Using the above Lefschetz isomorphisms, we are now ready to 
introduce our chirality operator. 
\begin{defi}\label{def:chirality}
The chirality operator associated to the contact form $\vartheta$ is the operator $\Gamma_\vartheta : \Lambda^\bullet T^*M \to \Lambda^{n-\bullet} T^*M$ defined by 
$\Gamma_\vartheta^2 = 1$ and
\begin{equation}\label{e:chirality}
\Gamma_\vartheta(f \wedge \vartheta + g) = \mathscr{L}^{r-k} g \wedge \vartheta + \mathscr{L}^{r-k+1} f, \quad f \in \Lambda^{k-1}V_X, \quad g \in \Lambda^kV_X, \quad k \in \{0, \dots, r\},
\end{equation}
where we used the decomposition (\ref{eq:decomp}).
\end{defi}
Note that in particular one has for $k \in \{r+1, \dots, n\}$,
$$
\Gamma_\vartheta (f \wedge \vartheta + g) = \left(\mathscr{L}^{k-r}\right)^{-1} g \wedge \vartheta + \left(\mathscr{L}^{k-1-r}\right)^{-1} f.
$$

\subsection{The refined torsion of a space of generalized eigenvectors}
The operator $\Gamma_\vartheta$ acts also on $\Omega^\bullet(M,E)$ by acting trivially on $E$-coefficients. Since $\Lie_X \vartheta = 0$, $\Gamma_\vartheta$ and $\Lien$ commute so that $\Gamma_\vartheta$ induces a chirality operator
$$
\Gamma_\vartheta : C^\bullet(s_0) \to C^{n-\bullet}(s_0)
$$
 for every $s_0 \in \Res(\Lien)$. Recall from \S\ref{subsec:topologyruelle} that the complexes $\bigl(C^\bullet(s_0), \nabla \bigr)$ are acyclic. The following formula motivates the upcoming definition of the dynamical torsion.
 
\begin{prop}\label{prop:computetorsion}
Let $s_0 \in \Res(\Lien) \setminus \{0,1\}$. We have
$$\tau(C^\bullet(s_0), \Gamma_\vartheta)^{-1} = (-1)^{Q_{s_0}}\grdet{C^\bullet (s_0)}\Lien$$
where 
$$Q_{s_0}= \sum_{k=0}^r (-1)^k(r+1-k) \dim C^k(s_0)$$
and $\tau(C^\bullet(s_0), \Gamma_\vartheta) \in \mathbb{C}\setminus 0$ is the refined torsion of the acyclic complex $\bigl(C^\bullet(s_0), \nabla \bigr)$ with respect to the chirality $\Gamma_\vartheta$, cf Definition \ref{def:torsion}.
\end{prop}
Let us first admit the above proposition; the proof will be given in \S\S\ref{subsec:signatureinvertible},\ref{subsec:proofprop}.

\subsection{Spectral cuts}\label{subsec:cut}
If $\I \subset [0, 1)$ is an interval, we set
$$\Pi_\I= \sum_{\substack{s_0 \in \Res(\Lien)\\ |s_0| \in \I}} \Pi_{s_0} \quad \text{ and } \quad C^{\bullet}_\I= \bigoplus_{\substack{s_0 \in \Res(\Lien)\\ |s_0| \in \I} }C^\bullet(s_0).$$
Note that $\Lien + s$ acts on $C^\bullet(s_0)$ for every $s_0 \in \Res(\Lien)$ as $-s_0 \id + J$ where $J$ is nilpotent. We thus have for $s \notin \Res(\Lien)$
\begin{equation}\label{eq:detlie}
\grdet{C^\bullet_\I}\left(\Lien+s \right)^{(-1)^{q+1}} = \prod_{\substack{s_0 \in \Res(\Lien)\\ |s_0| \in \I}} (s-s_0)^{m(s_0)},
\end{equation}
where ${\det}_{\mathrm{gr}}$ is the graded determinant, cf. \S \ref{subsec:graded}.

Let $\lambda \in [0, 1)$ such that $\Res(\Lien) \cap \{s \in \mathbb{C} : |s| = \lambda\} = \emptyset$.
Now define the meromorphic function
\begin{equation}\label{eq:zetalambda}
\zeta^{(\lambda, \infty)}_{X, \nabla}(s) ={\zeta_{X,\nabla}(s)}\grdet{C^{\bullet}_{[0, \lambda]}}\left(\Lien+s \right)^{(-1)^q}.
\end{equation}
Then (\ref{eq:orderres}) and (\ref{eq:detlie}) show that $\zeta^{(\lambda, \infty)}_{X, \nabla}$ has no pole nor zero in $\{|s|\leq \lambda\}$, so that the number $\zeta^{(\lambda, \infty)}_{X, \nabla}(0)$ is well defined.

\subsection{Definition of the dynamical torsion}\label{subsec:defdyn}

Let $0 < \mu < \lambda < 1$ such that for every $s_0 \in \Res(\Lien)$, one has $|s_0|\neq \lambda, \mu$. Using Proposition \ref{prop:multtorsion} and Proposition \ref{prop:computetorsion} we obtain, with notations of \S\ref{subsec:cut},
$$
\tau\bigl(C^\bullet_{[0, \lambda]}, \Gamma_\vartheta\bigr) = (-1)^{Q_{(\mu, \lambda]}} \left(\grdet{C^\bullet_{(\mu, \lambda]}} \Lien\right)^{-1} \tau\bigl(C^\bullet_{[0, \mu]}, \Gamma_\vartheta\bigr),
$$
where for an interval $\I$ we set 
$$Q_\I = \sum_{\substack{s_0 \in \Res(\Lien)\\ |s_0| \in \I}} Q_{s_0}.$$
This allows us to give the following

\begin{defi}[Dynamical torsion]\label{def:dyntors}
The number
\begin{equation}\label{eq:dyntors}
\tau_\vartheta(\nabla) = (-1)^{Q_{[0,\lambda]}} \zeta^{(\lambda, \infty)}_{X, \nabla}(0)^{(-1)^q} \cdot \tau\bigl(C^\bullet_{[0, \lambda]}, \Gamma_\vartheta\bigr) \in \mathbb{C}\setminus 0
\end{equation}
is independent of the spectral cut $\lambda \in (0, 1)$. We will call this number the \emph{dynamical torsion} of the pair $(\vartheta, \nabla)$.
\end{defi}

\begin{rem}
If $c_{X, \nabla}s^{m(0)}$ is the leading term of the Laurent expansion of $\zeta_{X, \nabla}(s)$ at $s=0$, then taking $\lambda$ small enough actually shows that
\begin{equation}\label{eq:lambdapetit}
\tau_{\vartheta}(\nabla) = (-1)^{Q_0}c_{X,\nabla}^{(-1)^q} \cdot \tau\bigl(C^\bullet(0), \Gamma_\vartheta\bigr).
\end{equation}
In particular, if $0 \notin \Res(\Lien)$,
\begin{equation}\label{eq:notresonance}
\tau_\vartheta(\nabla) = \zeta_{X, \nabla}(0)^{(-1)^q}.
\end{equation}
Note that we could have taken (\ref{eq:lambdapetit}) as a definition of the dynamical torsion; however (\ref{eq:dyntors}) is more convenient to study the regularity of the $\tau_{\vartheta}(\nabla)$ with respect to $\vartheta$ and $\nabla$. 
\end{rem}

\begin{rem}\label{rem:notacyclic}
This definition actually makes sense even if $\nabla$ is not acyclic. Indeed, in that case, formula (\ref{eq:dyntors}) defines an element of the determinant line $ \det H^\bullet\bigl(C^\bullet_{[0,\lambda]} \nabla\bigr)$, cf. Remark \ref{rem:notacyclic0}. Under the identification $H^\bullet(M,\nabla) = H^\bullet\bigl(C^\bullet_{[0,\lambda]} \nabla\bigr)$ given by the quasi-isomorphism $\Pi_{[0,\lambda]} : \Omega^\bullet(M,E) \to C^\bullet_{[0,\lambda]}$ (cf \S\ref{subsec:topologyruelle}), we thus get an element of $\det H^\bullet(M,\nabla)$.
\end{rem}
%

The rest of this section is devoted to the proof of Proposition \ref{prop:computetorsion}.

\subsection{Invertibility of the contact signature operator}\label{subsec:signatureinvertible}
To prove Proposition \ref{prop:computetorsion} we shall use \S\ref{subsec:signature} and introduce the \emph{contact signature operator}
$$B_\vartheta = \Gamma_\vartheta \nabla + \nabla \Gamma_\vartheta : \mathcal{D}'^\bullet(M,E) \to \mathcal{D}'^\bullet(M,E),$$
where $\Gamma_\vartheta$ acts trivially on $E$.
We fix in what follows some $s_0 \in \Res(\Lien) \setminus \{0,1\}$ and we denote $C^\bullet(s_0)$ by $C^\bullet$ for simplicity. We also set $C^\bullet_0 = C^\bullet \cap \ker(\iota_X).$ 

The following result will put us in position to apply Proposition \ref{prop:formulatorsion}.

\begin{lemm}\label{lem:signinvertible}
The operator $B_\vartheta$ is invertible $C^\bullet \to C^\bullet$.
\end{lemm}

\begin{proof}
We set 
$$
C^\bullet_\mathrm{even} = \bigoplus_{k \text{ even}} C^k, \quad C^\bullet_\mathrm{odd} = \bigoplus_{k \text{ odd}} C^k.
$$
Then $B_\vartheta$ preserves the decomposition $C^\bullet = C^\bullet_\mathrm{even} \oplus C^\bullet_\mathrm{odd}$. Note that because $\Gamma_\vartheta^2 = 1$, we have $B_\vartheta|_{C^\bullet_\mathrm{even}} = \Gamma_\vartheta B_\vartheta|_{C^\bullet_\mathrm{odd}} \Gamma_\vartheta$. It thus suffices to show that $B_\vartheta$ is injective on $C^\bullet_\mathrm{even}$. Let $\beta \in C^\bullet_\mathrm{even}$ such that $B_\vartheta \beta = 0$. Write
$$
\beta = \sum_{k=0}^r \beta_{2k} \in C^\bullet_\mathrm{even},
$$
with
$$
\beta_{2k} = f_{2k-1}\wedge \vartheta + g_{2k}, \quad f_{2k-1} \in C^{2k-1}_0, \quad g_{2k} \in C^{2k}_0, \quad k = 0, \dots, r.
$$
Then $B_\vartheta \beta = 0$ writes, since $\Gamma_\vartheta \nabla(C^k) \subset C^{n-k-1}$ and $\nabla \Gamma_\vartheta (C^k) \subset C^{n-k+1}$,
\begin{equation}\label{B=0}
\Gamma_\vartheta \nabla \beta_{2k} + \nabla \Gamma_\vartheta \beta_{2(k+1)} = 0, \quad k = 0, \dots, r.
\end{equation}
Because $\nabla$ does not leave the decomposition (\ref{eq:decomp}) stable, we need to introduce an operator $\Psi : C^\bullet_0 \to C^{\bullet + 1}_0$ which mimics the action of $\nabla$. We define
\begin{equation}\label{eq:defpsi}
\Psi \mu = \nabla \mu - (-1)^k \Lien \mu \wedge \vartheta, \quad \mu \in C^k_0.
\end{equation}
Because $\Lie_X \dd \vartheta = 0$,  the map $\Psi$ satisfies the simple relation
\begin{equation}\label{eq:B}
\Psi\left(\mu\wedge \dd\vartheta^j\right)= \left(\Psi\mu\right)\wedge \dd\vartheta^j , \quad  \mu \in C^\bullet_0, \quad j \in \mathbb{N},
\end{equation}
that is, $\Psi$ commutes with $\mathscr{L}$. Also, observe that
\begin{equation}\label{A}
\Psi^2 \mu = - \Lien \mu \wedge \dd \vartheta, \quad \mu \in C^\bullet_0.
\end{equation}
Indeed, using the fact that $\Lien$ and $\nabla$ commute,
\begin{equation*}
\begin{aligned}
\Psi^2\mu &= \nabla\left(\nabla \mu - (-1)^k \Lien \mu \wedge \vartheta \right)-(-1)^{k+1}\left(\mathcal{L}_X^\nabla\left(\nabla \mu - (-1)^k \Lien \mu \wedge \vartheta \right)\right)\wedge\vartheta \\
&={\nabla^2\mu}+(-1)^{k+1}\nabla \left(\Lien\mu\wedge \vartheta \right)+(-1)^k \Lien\nabla\mu \wedge\vartheta-{\Lien^2\mu\wedge\vartheta\wedge\vartheta} \\
&= (-1)^{k+1}(-1)^k\Lien\mu\wedge \dd \vartheta.
\end{aligned}
\end{equation*}

Assume first that $k\leq r/2 - 1$. Then $2k+2 \leq r$; we can thus write, with (\ref{eq:defpsi}) in mind,
$$\begin{aligned}
\Gamma_\vartheta \nabla \beta_{2k} &=\Gamma_\vartheta \Bigl(\nabla f_{2k-1}\wedge\vartheta  - f_{2k-1}\wedge \dd\vartheta+\nabla g_{2k} \Bigr) \\
&= \Gamma_\vartheta \Bigl(\Psi f_{2k-1} \wedge \vartheta - \Lien f_{2k-1}\wedge\vartheta\wedge\vartheta - f_{2k-1} \wedge \dd \vartheta + \Psi g_{2k} + \Lien g_{2k} \wedge \vartheta \Bigr) \\
&= \Bigl(\Psi f_{2k-1} + \Lien g_{2k}\Bigr) \wedge \dd \vartheta^{r-2k} + \Bigl(\Psi g_{2k}- f_{2k-1}\wedge \dd\vartheta\Bigr) \wedge \dd \vartheta^{r-2k-1} \wedge \vartheta.
\end{aligned} $$  

Similarly we find by (\ref{eq:B})
\begin{equation}\label{eq:nablagammasign}
\begin{aligned}
\nabla \Gamma_\vartheta \beta_{2k +2} &= \nabla \Bigl(f_{2k+1} \wedge \dd \vartheta^{r-2k - 1} + g_{2k+2} \wedge \dd \vartheta^{r-2k-2} \wedge \vartheta \Bigr)\\
&= \Bigl(\Psi  f_{2k+1}- \Lien f_{2k+1} \wedge \vartheta\Bigr)  \wedge \dd \vartheta^{r-2k-1} + \Bigl(\Psi  g_{2k+2} \wedge \vartheta + g_{2k+2} \wedge \dd \vartheta\Bigr) \wedge \dd \vartheta^{r-2k-2}.
\end{aligned}
\end{equation}
Thus (\ref{B=0}) writes, with the decompostion (\ref{eq:decomp}) in mind,
\begin{equation}\label{i}
\Bigl(\Psi f_{2k+1} + g_{2k+2}\Bigr) \wedge \dd \vartheta^{r-2k-1} + \Bigl(\Psi f_{2k-1} + \Lien g_{2k}\Bigr) \wedge \dd \vartheta^{r-2k} = 0
\end{equation}
and
\begin{equation}\label{ii}
\Bigl(-\Lien f_{2k+1} \wedge \dd \vartheta +\Psi g_{2k+2} \Bigr)\wedge \dd \vartheta^{r-2k-2} + \Bigl(\Psi g_{2k}- f_{2k-1} \wedge \dd \vartheta\Bigr) \wedge \dd \vartheta^{r-2k-1}  = 0.
\end{equation}
Then applying $\Psi $ to (\ref{ii}) gives, with (\ref{A}) and (\ref{eq:B}),
$$
\Bigl(-\Psi \Lien f_{2k+1} -\Lien g_{2k+2}\Bigr) \wedge \dd \vartheta^{r-2k-1} -\Lien g_{2k} \wedge \dd \vartheta^{r-2k} - \Psi f_{2k-1}\wedge \dd \vartheta^{r-2k} = 0.
$$
Note that $\Psi$ commutes with $\Lien$ and thus with $\Lien^{-1}$ (which exists since $s_0 \neq 0$). Then applying $\Lien^{-1}$ to the above relation we get
$$
\Bigl(-\Psi  f_{2k+1} - g_{2k+2} \Bigr) \wedge \dd \vartheta^{r-2k-1} -g_{2k} \wedge \dd \vartheta^{r-2k} - \Lien^{-1}\left(\Psi f_{2k-1}\wedge \dd \vartheta^{r-2k}\right) = 0.
$$
Injecting this in (\ref{i}), we obtain
$$
\Bigl(\left(\Lien-\id\right)g_{2k} + \bigl(\id -\Lien^{-1}\bigr)\Psi f_{2k-1}\Bigr) \wedge \dd \vartheta^{r-2k} = 0.
$$
Since $\mathscr{L}^{r-2k}$ is injective on $C^{2k}_0$ and $\Lien-\id$ is invertible (since $s_0 \neq 1$), this yields
\begin{equation}\label{eq:g2kk}
\Lien g_{2k} + \Psi f_{2k-1} = 0.
\end{equation}
Applying $\Lien^{-1} \Psi$ to the above equation we get
\begin{equation}\label{g2k}
\Psi g_{2k} - f_{2k-1} \wedge \dd \vartheta = 0
\end{equation}
by (\ref{eq:B}); thus (\ref{ii}) gives
\begin{equation*}\label{g2k+2}
\Bigl(\Psi g_{2k+2} - \Lien f_{2k+1} \wedge \dd \vartheta \Bigr)\wedge \dd \vartheta^{r-2k-2} = 0.
\end{equation*}
Now repeating this process with $k$ replaced by $k-1$ 
we obtain $\left(\Psi g_{2k} - \Lien f_{2k-1} \wedge \dd \vartheta\right)\wedge \dd\vartheta^{r-2k}=0$. This implies with (\ref{g2k}) that
$$
\bigl(\id -\Lien\bigr)f_{2k-1} \wedge \dd \vartheta^{r-2k+1} = 0,
$$
which leads to $f_{2k-1} = 0$ since $\mathscr{L}^{r-(2k-1)}$ is injective on $C^{2k-1}_0$ and $\Lien-\id$ is invertible on $C^\bullet$; thus $g_{2k}$ = 0 by (\ref{eq:g2kk}), since $\Lien$ is invertible.
We therefore obtained
$$
\beta_{2k} = 0, \quad k \leq r/2 - 1.
$$

Next assume $k \geq (r+1)/2$. Set $\tilde k = r-k$ and
$$
\tilde{\beta}_{2\tilde k +1} = \Gamma_\vartheta \beta_{2k} \in C^{2\tilde k  +1}_0, \quad \tilde{\beta}_{2\tilde k -1} = \Gamma_\vartheta \beta_{2k+2} \in C^{2\tilde k -1}_0.
$$
Then (\ref{B=0}) writes
$$
\Gamma_\vartheta \nabla \tilde{\beta}_{2\tilde k -1} + \nabla \Gamma_\vartheta \tilde{\beta}_{2\tilde k +1} = 0. 
$$
Since $2\tilde k + 1 \leq r$ and we can do exactly as before to get $\tilde{\beta}_{2\tilde k -1} = 0$ which leads to $\beta_{2k+2} = 0$. Therefore we obtained
$$
\beta_{2k} = 0, \quad k \geq (r+1)/2 + 1.
$$

Therefore it remains to show that $\beta_{2p} = 0$ and $ \beta_{2(p+1)} = 0,$ where $p = \lfloor r/2 \rfloor$. 
We will assume that $r = 2p + 1$ is odd and put $p'= p +1$ (the case $r$ even is similar). Then (\ref{B=0}) 
implies, since $\beta_{2k} = 0$ for every $k \neq p,p'$,
\begin{equation}\label{final}
\nabla \Gamma_\vartheta \beta_{2p'} + \Gamma_\vartheta \nabla \beta_{2p} = 0, \quad \Gamma_\vartheta \nabla \beta_{2p'} = 0, \quad \nabla \Gamma_\vartheta \beta_{2p} = 0.
\end{equation}
We can compute, keeping (\ref{eq:defpsi}) in mind,
$$
\begin{aligned}
\nabla \Gamma_\vartheta \beta_{2p'} &= \nabla\left(\mathscr{L}^{-1}g_{2p'} \wedge \vartheta + f_{2p'-1}\right) \\
&= \Psi \mathscr{L}^{-1}g_{2p'} \wedge \vartheta + \Lien \mathscr{L}^{-1} g_{2p'} \wedge \vartheta \wedge \vartheta + {{\mathscr{L}^{-1}g_{2p'} \wedge \dd \vartheta}} + \Psi f_{2p'-1} - \Lien f_{2p'-1} \wedge \vartheta,
\end{aligned}
$$
and 
$$
\begin{aligned}
\Gamma_\vartheta \nabla \beta_{2p} &= \Gamma_\vartheta \Bigl( \Psi g_{2p} + \Lien g_{2p} \wedge \vartheta + \Psi f_{2p-1} \wedge \vartheta - \Lien f_{2p-1} \wedge \vartheta \wedge \vartheta - f_{2p-1} \wedge \dd \vartheta \Bigr) \\
&= \Psi g_{2p} \wedge \vartheta - f_{2p-1} \wedge \dd \vartheta \wedge \vartheta + \Lien g_{2p} \wedge \dd \vartheta + \Psi f_{2p-1} \wedge \dd \vartheta.
\end{aligned}
$$
Therefore the first equation of (\ref{final}) implies, since $\mathscr{L}^{-1}g_{2p'} \wedge \dd \vartheta = g_{2p'},$
\begin{equation}\label{(i)}
\Psi \mathscr{L}^{-1}g_{2p'} - \Lien f_{2p'-1}-f_{2p-1} \wedge \dd \vartheta + \Psi g_{2p} = 0
\end{equation}
and 
\begin{equation} \label{eq:(k)}
g_{2p'} + \Psi f_{2p'-1} + \Psi f_{2p-1} \wedge \dd \vartheta + \Lien g_{2p} \wedge \dd \vartheta= 0.
\end{equation}
Applying $\Lien^{-1}\Psi $ to (\ref{(i)}) leads to 
$$
-g_{2p'} -\Psi  f_{2p'-1} - \Psi \Lien^{-1}f_{2p-1} \wedge \dd \vartheta + -g_{2p} \wedge \dd \vartheta=0.
$$
Therefore,
\begin{equation}\label{eq:(j)}
\left(\left(\id - \Lien^{-1}\right)\Psi f_{2p-1} + \left(\Lien-\id\right) g_{2p}\right)\wedge \dd \vartheta.
\end{equation}
As before this gives $\Psi f_{2p-1} + \Lien g_{2p} = 0$ and thus with (\ref{eq:(k)}) one gets
\begin{equation}\label{final2}
\Lien g_{2p} + \Psi f_{2p-1} = 0, \quad g_{2p'} + \Psi f_{2p'-1} = 0.
\end{equation}

Next compute
$$\nabla \Gamma_\vartheta \beta_{2p} = g_{2p} \wedge \dd \vartheta^2 + \Psi  f_{2p-1} \wedge \dd \vartheta^2 + \Psi g_{2p}\wedge \vartheta \wedge \dd \vartheta - \Lien f_{2p+1} \wedge \vartheta \wedge \dd \vartheta^2$$
Therefore the third part of (\ref{final}) gives (we take the $\wedge \vartheta$ component of the above equation)
$$
-\Lien f_{2p-1} \wedge \dd \vartheta^2 + \Psi g_{2p} \wedge \dd \vartheta =0.
$$
Applying $\Lien^{-1}\Psi$ to (\ref{final2}) we get $\Psi  g_{2p} = f_{2p-1}\wedge  \dd \vartheta$; we therefore obtain that $f_{2p-1}= 0$ by injectivity of $\mathscr{L}^2$ on $C^{r-2}_0$. Thus $g_{2p} = 0$ by (\ref{eq:(j)}). 

Finally compute
$$
\nabla \beta_{2p'} = \Psi f_{2p'-1} \wedge \vartheta + \Psi g_{2p'} + \Lien g_{2p'} \wedge \vartheta = 0.
$$
Therefore the second part of (\ref{final}) implies (since $\Gamma_\varthetaÊ\nabla \beta_{2 p'} = 0$ is equivalent to $\nabla \beta_{2p'} = 0$)
$$
\Psi f_{2p'-1} + \Lien g_{2p'} = 0.
$$
Therefore by (\ref{final2}) we get $\left(\Lien-\id\right)g_{2p'} = 0$, and thus $g_{2p'} =0$. Using (\ref{(i)}) we conclude that $\Lien f_{2p'-1} = 0$ which leads to $f_{2p'-1} = 0.$ 
\end{proof}

\subsection{Proof of Proposition \ref{prop:computetorsion}}\label{subsec:proofprop}
We start from Proposition \ref{prop:formulatorsion} which gives us, in view of Lemma \ref{lem:signinvertible},
\begin{equation}\label{eq:product}
\tau(C^\bullet, \Gamma_\vartheta) = (-1)^{r\dim C^r_+} \det\left(\Gamma_\vartheta \nabla|_{C^r_+}\right)^{(-1)^{r}} \prod_{j=0}^{r-1} \det\left(\Gamma_\vartheta \nabla|_{C^j_+ \oplus C^{n-j-1}_+}\right)^{(-1)^{j}}.
\end{equation}
where we set as in \S\ref{subsec:signature}
$$C^\bullet_+ = C^\bullet \cap \ker( \nabla \Gamma_\vartheta ), \quad C^\bullet_- = C^\bullet \cap \ker( 
 \Gamma_\vartheta \nabla).$$ 
We first note that for $k \in \{0, \dots, r\}$ and $\beta \in \Omega^k(M,E)$, one has
\begin{equation}\label{eq:b}
\begin{split}
\nabla \Gamma_\vartheta \beta = &\ {\mathscr{L}^{r-k}\Bigl(\nabla \beta - (-1)^k\iota_X\nabla \beta \wedge \vartheta + \mathscr{L}\bigl(\iota_X\nabla\iota_X \beta - \iota_X \beta \bigr) \Bigr) \wedge \vartheta} \\
& \quad \quad +(-1)^k\mathscr{L}^{r-k+1}\Bigl(\beta - \nabla \iota_X\beta +(-1)^k\iota_X\left(\beta -\nabla\iota_X\beta\right)\wedge \vartheta \Bigr), \\
\Gamma_\vartheta\nabla \beta = &\mathscr{L}^{r-k-1}\Bigl(\nabla \beta - (-1)^k \iota_X \nabla \beta\wedge \vartheta \Bigr)\wedge \vartheta +(-1)^k \mathscr{L}^{r-k}\bigl(\iota_X \nabla \beta\bigr),
\end{split}
\end{equation} 
where $\mathscr{L}^{j-r} = (\mathscr{L}^{r-j}|_{\Lambda^jV_X})^{-1}$ for $0 \leq j \leq r.$ Indeed, using the decomposition (\ref{eq:decomp}),
$$
\begin{aligned}
\Gamma_\vartheta \beta
&= (-1)^{k+1} \iota_X \beta\wedge \dd \vartheta^{r-k+1} + \Bigl(\beta + (-1)^k\iota_X\beta \wedge \vartheta\Bigr) \wedge \dd \vartheta^{r-k} \wedge \vartheta \\
&= (-1)^{k+1} \iota_X \beta\wedge \dd \vartheta^{r-k+1} + \beta \wedge \dd \vartheta^{r-k} \wedge \vartheta,
\end{aligned}
$$
which leads to
$$
\begin{aligned}
\nabla \Gamma_\vartheta \beta &= (-1)^{k+1} \nabla \iota_X \beta \wedge \dd \vartheta^{r-k+1} + \nabla \beta \wedge \dd \vartheta^{r-k} \wedge \vartheta + (-1)^k \beta \wedge \dd \vartheta^{r-k+1}  \\
&= (-1)^{k+1}\Bigl((-1)^{k+1}\iota_X \nabla \iota_X \beta \wedge \vartheta \wedge \dd \vartheta^{r-k+1}\Bigr) + (-1)^{k+1} \Bigl(\nabla \iota_X \beta + (-1)^k \iota_X \nabla \iota_X \beta \wedge \vartheta\Bigr) \wedge \dd \vartheta^{r-k+1} \\ 
&\quad \quad \quad + \Bigl(\nabla\beta - (-1)^k \iota_X \nabla \beta \wedge \vartheta\Bigr) \wedge \dd \vartheta^{r-k} \wedge \vartheta +(-1)^k \Bigl(\beta + (-1)^k\iota_X\beta \wedge \vartheta\Bigr) \wedge \dd \vartheta^{r-k+1}  \\
& \quad \quad \quad - \iota_X \beta \wedge \dd \vartheta^{r-k+1} \wedge \vartheta,
\end{aligned}
$$
which is exactly the first part of (\ref{eq:b}). The second part follows directly from the decomposition (\ref{eq:decomp}).

Let us introduce, for $k \in \{0, \dots, r\}$, the operator $J_k : C^k \to C^k$ defined by
\begin{equation}\label{eq:defjk}
J_k \beta = f\wedge \vartheta - (-1)^k \Psi f
\end{equation}
for any $\beta = f \wedge \vartheta + g \in C^k$ with $f \in C^{k-1}_0$ and $g \in C^k_0$, and where $\Psi$ is defined in (\ref{eq:defpsi}). Then we claim that $J_k$ takes it values in $C^k_+$. Indeed, we have for any $f \in C^{k-1}_0$ and $g \in C^k_0$,
$$
\begin{aligned}
\nabla \Gamma_\vartheta (f\wedge \vartheta + g) &= \nablaÊ\Bigl( g \wedge \dd \vartheta^{r-k} \wedge \vartheta + f \wedge \dd \vartheta^{r-k+1} \Bigr) \\
&= \Psi g \wedge \dd \vartheta^{r-k}\wedge \vartheta  + (-1)^k g \wedge \dd \vartheta^{r-k+1} \\
&\quad \quad \quad +  \Psi f \wedge \dd \vartheta^{r-k+1} + (-1)^{k+1} \Lien f \wedge \dd \vartheta^{r-k+1} \wedge \vartheta,
\end{aligned}
$$
which implies that $\beta = f \wedge \vartheta + g$ lies in $C^k_+$ if and only if
\begin{equation}\label{eq:iff+}
\Bigl(\Psi g + (-1)^{k+1} \Lien f \wedge \dd \vartheta\Bigr) \wedge \dd \vartheta^{r-k} = 0 \quad \text{ and } \quad \Bigl(\Psi f + (-1)^k g \Bigr) \wedge \dd \vartheta^{r-k+1} = 0.
\end{equation}
But now note that if $\beta = f \wedge \vartheta + g = J_k\beta' = f' \wedge \vartheta - (-1)^k \Psi f'$ for some $\beta' = f' \wedge \vartheta + g'$ then $f =f'$ and $g = - (-1)^k\Psi f$, and thus $\beta$ satisfies the second part of (\ref{eq:iff+}). We also obtain $\Psi g = -(-1)^k \Psi^2 f = - (-1)^k \Lien f \wedge \dd \vartheta$ by (\ref{A}), so the first part of (\ref{eq:iff+}) is also satisfied.

Therefore $J_k : C^k \to C^k_+$; moreover it is obvious that $J_k$ is a projector. Therefore we can consider the restricted projection $J_k|_{C^k_+} : C^k_+ \to C^k_+$, we will still denote by $J_k.$

The next lemma will be helpful to compute the determinants lying in the product (\ref{eq:product}).
\begin{lemm}\label{lem:k<r}
Take $k \in \{0, \cdots, r-1\}.$ Then for any $\beta = f \wedge \vartheta + g \in C^k_+$ with $f \in C^{k-1}_0$ and $g \in C^k_0$, one has
$$(\Gamma_\vartheta \nabla)^2 \beta = \Lien\left(\Lien-\id\right)\beta - \left(\Lien-\id\right)J_k \beta.$$
\end{lemm}

\begin{proof}
Since $k<r$ we can write, thanks to (\ref{eq:b}),
$$\Gamma_\vartheta \nabla \beta = \nabla \beta \wedge \vartheta \wedge \dd \vartheta^{r-k-1} + (-1)^k \iota_X \nabla \beta \wedge \dd \vartheta^{r-k}.$$
Therefore
\begin{equation*}
\begin{split}
\nabla \Gamma_\vartheta \nabla \beta &= -(-1)^k\nabla \beta \wedge \dd \vartheta^{r-k} + (-1)^k \nabla \iota_X \nabla\beta \wedge \dd \vartheta^{r-k} \\
&{= (-1)^k\left(\Lien - \id\right)\nabla \beta \wedge \dd \vartheta^{r-k}} \\
&= {\Bigl(\iota_X\nabla\iota_X\nabla \beta - \iota_X\nabla\beta \Bigr) \wedge \vartheta \wedge \dd \vartheta^{r-k}} \\
&\quad \quad +(-1)^k(\Lien - \id)\Bigl(\nabla \beta - (-1)^k \iota_X\nabla \beta \wedge \vartheta\Bigr)\wedge \dd \vartheta^{r-k},
\end{split}
\end{equation*}
where we used $\nabla \iota_X  \nabla \beta = \Lien \nabla \beta$ and $\iota_X \nabla \iota_X \nabla \beta = \Lien \iota_X \nabla \beta$. Since $\beta \in C^k_+$ one has with (\ref{eq:b})
$$\Bigl(\nabla \beta -(-1)^k \iota_X\nabla \beta \wedge \vartheta\Bigr)\wedge \dd \vartheta^{r-k} = \Bigl(\iota_X \beta - \iota_X \nabla \iota_X \beta\Bigr)\wedge \dd \vartheta^{r-k+1}.$$
This leads to
\begin{equation*}
\begin{split}
\nabla \Gamma_\vartheta \nabla \beta &= {\Bigl(\iota_X \nabla \iota_X \nabla \beta - \iota_X \nabla \beta \Bigr) \wedge \vartheta \wedge \dd \vartheta^{r-k}}\\
& \quad \quad + (-1)^k\left(\Lien-\id\right)\Bigl(\iota_X \beta - \iota_X \nabla \iota_X \beta\Bigr) \wedge \dd \vartheta^{r-k+1}.
\end{split}
\end{equation*}
Since $\iota_X \nabla \iota_X \nabla \beta - \iota_X \nabla \beta = \left(\Lien-\id\right)\iota_X \nabla \beta$ and $\iota_X \beta - \iota_X \nabla \iota_X \beta = \bigl(\id-\Lien\bigr)\iota_X\beta$, we obtain
$$
\nabla \Gamma_\vartheta \nabla \beta = \left(\Lien-\id\right)\iota_X \nabla \beta\wedge \vartheta \wedge \dd \vartheta^{r-k} + (-1)^k \left(\Lien - \id\right)\left(\id-\Lien\right)\iota_X \beta \wedge \dd \vartheta^{r-k+1},
$$
and thus by definition of $\Gamma_\vartheta$
\begin{equation}\label{eq:b^2}
\Gamma_\vartheta \nabla \Gamma_\vartheta \nabla \beta = -(-1)^k\bigl(\id-\Lien\bigr)^2\iota_X \beta \wedge \vartheta + \left(\Lien-\id\right)\iota_X\nabla \beta.
\end{equation}
Now, writing $\beta = f\wedge \vartheta + g$ where $\iota_X f = 0$ and $\iota_X g = 0$, we have
\begin{equation}\label{eq:betafg}
\begin{split}
\nabla \beta &= \nabla f \wedge \vartheta - (-1)^k f \wedge \dd \vartheta + \nabla g,\\
\iota_X \nabla \beta &= \Lien f\wedge \vartheta +(-1)^k \nabla f + \Lien g, \\
\iota_X \beta \wedge \vartheta &= -(-1)^k f\wedge \vartheta.
\end{split}
\end{equation}
Injecting those relations in (\ref{eq:b^2}) we get 
\begin{equation*}
\begin{split}
\Gamma_\vartheta \nabla \Gamma_\vartheta \nabla \beta &= \Lien \bigl(\Lien - \id\bigr)(f \wedge \vartheta + g) \\ 
&\quad \quad - \bigl(\Lien - \id\bigr)\Bigl(f \wedge \vartheta - (-1)^k\bigl(\nabla f +(-1)^k \Lien f \wedge \vartheta\bigr)\Bigr),
\end{split}
\end{equation*}
which concludes in view of (\ref{eq:defpsi}) and (\ref{eq:defjk}).
\end{proof}

We now deal with the case $k=r$.
\begin{lemm}\label{lem:k=r}
One has, for $\beta \in C^r_+$,
$$\Gamma_\vartheta \nabla \beta = (-1)^r\Bigl(\bigl(\Lien - \id\bigr)\beta +  (\id - J_r) \beta\Bigr).$$
\end{lemm}
\begin{proof}
We have
$$\Gamma_\vartheta \nabla \beta = \mathscr{L}^{-1}\bigl(\nabla \beta - (-1)^r \iota_X \nabla \beta \wedge \vartheta \bigr) + (-1)^r \iota_X \nabla \beta.$$
Since $\beta \in C^r_+$ we have with (\ref{eq:b}) that $\nabla \beta - (-1)^r \iota_X \nabla \beta \wedge \vartheta = (\iota_X \beta - \iota_X \nabla \iota_X \beta)\wedge \dd \vartheta$. Therefore, 
$$\Gamma_\vartheta \nabla \beta = (\iota_X \beta- \iota_X \nabla \iota_X \beta) \wedge \vartheta + (-1)^r \iota_X \nabla \beta.$$
We now conclude as in the previous lemma, using (\ref{eq:betafg}).
\end{proof}

We are now in position to finish the proof of Proposition \ref{prop:computetorsion}. We will set, for $0 \leq k \leq n$, 
$$m_k = \dim C^k, \quad m_k^0 = \dim C^k_0, \quad m_k^{\pm} = \dim C^k_\pm.$$

First take $k \in \{0, \cdots, r-1\}$. First take $k \in \{0, \cdots, r-1\}$. Because $B_\vartheta$ is invertible on $C^\bullet$, $\Gamma_\vartheta \nabla$ induces an isomorphism $C^k_+ \to C^{n-k-1}_+.$ Take any basis $\gamma$ of $C^k_+$. Then $\Gamma_\vartheta \nabla \gamma$ is a basis of $C^{n-k-1}_+$ and the matrix of $\Gamma_\vartheta \nabla|_{C^k_+ \oplus C^{n-k+1}_+}$ in the basis $\gamma \oplus \Gamma_\vartheta \nabla \gamma$ is 
\begin{equation}\label{eq:matrix}
\begin{pmatrix}
0 & \bigl[(\Gamma_\vartheta \nabla)^2\bigr]_\gamma \\ 
\id & 0
\end{pmatrix},
\end{equation}
where $\bigl[(\Gamma_\vartheta \nabla)^2\bigr]_\gamma$ is the matrix of $(\Gamma_\vartheta \nabla)^2|_{C^k_+}$ in the basis $\gamma$. Define 
$$\tilde J_k = \id - J_k : C^k_+ \to C^k_+.$$
Then $\tilde{J}_k$ is a projector (since $J_k$ is) and Lemma \ref{lem:k<r} implies that $J_k$ (and thus $\tilde J_k$) commutes with $\Lien$. Moreover one has
$$
\left(\Gamma_\vartheta \nabla\right)^2|_{\ker \tilde J_k} = \bigl(\Lien - \id\bigr)^2, \quad \left(\Gamma_\vartheta \nabla\right)^2|_{\mathrm{ran} \tilde J_k} = \Lien \bigl(\Lien -\id\bigr).
$$
As a consequence,
$$\det\left((\Gamma_\vartheta \nabla)^2|_{C^k_+}\right) = \bigl[s_0(1+s_0)\bigr]^{m_k^+ - m_{k-1}^0}(1+s_0)^{2m_{k-1}^0} = {s_0}^{m_{k}^+-m_{k-1}^0}(1+s_0)^{m_k^+ + m_{k-1}^0},$$
because on $C^\bullet$ (and in particular on $C^k_+$), one has $\Lien = -s_0 \id + \nu$ where $\nu$ is nilpotent, and one has $\dim \ker \tilde J_k = \dim \mathrm{ran} J_k = m_{k-1}^0$. Indeed, by (\ref{eq:defjk}) we can view $J_k$ as 
a map $C^{k-1}_0 \to C^k_+$, which is obviously injective. We finally obtain with (\ref{eq:matrix})
\begin{equation}\label{eq:detk}
\det\left(\Gamma_\vartheta \nabla |_{C^k_+ \oplus C^{n-k+1}_+}\right) = (-1)^{m_k^+} {s_0}^{m_{k}^+-m_{k-1}^0}(1+s_0)^{m_k^+ + m_{k-1}^0}.
\end{equation}

We now deal with the case $k=r$. Lemma \ref{lem:k=r} gives
$$\Gamma_\vartheta \nabla|_{\ker \tilde J_r} = (-1)^r \left(\Lien-\id\right), \quad \Gamma_\vartheta \nabla|_{\mathrm{ran} \tilde J_r} = (-1)^r \Lien.$$
As before, we obtain
\begin{equation}\label{eq:detr}
\det\left(\Gamma_\vartheta \nabla |_{C^r_+}\right) = (-1)^{rm_r^+}(-1)^{m_r^+}{s_0}^{m_r^+-m_{r-1}^0}(1+s_0)^{m_{r-1}^0}.
\end{equation}

Combining (\ref{eq:product}) with (\ref{eq:detk}) and (\ref{eq:detr}) we finally obtain
\begin{equation}\label{eq:tauform}
\tau(C^\bullet, \Gamma_\vartheta) = (-1)^{J}{s_0}^K(1+s_0)^L
\end{equation}
where 
$$J = \sum_{k=0}^r(-1)^k m_k^+, \quad K = \sum_{k=0}^{r} (-1)^k(m_k^+ - m_{k-1}^0), \quad L =\sum_{k=0}^{r-1}(-1)^k (m_k^+ - m_k^0).$$
Note that for $0 \leq k \leq r-1$ one has by acyclicity and because $\Gamma_\vartheta$ induces isomorphisms $C^k_+ \simeq C^{n-k}_-$ (since $B_\vartheta$ is invertible),
$$
\begin{aligned}
m_k^+ = m_{n-k}^- &= \dim \ker \left(\nabla|_{C^{n-k}}\right) = \dim \mathrm{ran}\left(\nabla|_{C^{n-k-1}}\right) = m_{n-k-1} - m_{n-k-1}^{-} = m_{k+1}-m_{k+1}^{+}.
\end{aligned}
$$
Therefore 
\begin{equation}\label{eq:m_k}
m_k^+ + m_{k+1}^+ = m_{k+1}, \quad 0 \leq k \leq r-1,
\end{equation} which leads to $m_k^+ + m_{k+1}^+= m_k^0 + m_{k+1}^0$. As a consequence, since $m_0^+ = m_0 = m_0^0$, we get
$$m_r^+ - m_r^0  = -(m_{r-1}^+ - m_{r-1}^0) = \cdots = (-1)^r(m_0^+ - m_0^0) = 0.$$
This implies 
\begin{equation}\label{eq:0=+}
m_k^0 = m_k^+, \quad 0 \leq k \leq r,
\end{equation}
which leads to $L = 0$. Moreover, since $m_k^0 = m_{2r-k}^0$, we get
$$K = \sum_{k=0}^{r} (-1)^k(m_k^0 - m_{k-1}^0) = \sum_{k=0}^{2r} (-1)^k m_k^0 = -\sum_{k=0}^n (-1)^k k m_k = (-1)^qm(s_0),$$
where we used (\ref{eq:orderres}) in the last equality. Finally, again because $m_k^0 = m_{2r-k}^0$,
$$2J = (-1)^rm_r^0 + \sum_{k=0}^{2r} (-1)^km_k^0 = (-1)^rm_r^0 - \sum_{k=0}^n(-1)^kkm_k.$$
We have
$$
\begin{aligned}
(-1)^rm_r^0 &= \sum_{k=0}^r (-1)^km_k, \\ \sum_{k=0}^n(-1)^kkm_k &=  \sum_{k=0}^r(-1)^k(2k-n)m_k,
\end{aligned}
$$
where the first equality comes from (\ref{eq:m_k}) and (\ref{eq:0=+}) and the second from the fact that $m_k = m_{n-k}$. We thus obtained
$$J = \sum_{k=0}^r(-1)^k(r+1-k)m_k = Q_{s_0},$$
and finally by (\ref{eq:tauform})
$$
\tau(C^\bullet, \Gamma_\vartheta) = (-1)^{Q_{s_0}} (-s_0)^{(-1)^qm(s_0)}
$$
But now recall from (\ref{eq:detlie}) that $\grdet{C^\bullet}\left(\Lien\right)^{(-1)^{q+1}} = (-s_0)^{m(s_0)}.$ This completes the proof.

\section{Invariance of the dynamical torsion under small perturbations of the contact form}\label{sec:invariance}
In this section, we are interested in the behaviour of the dynamical torsion when we deform the contact form. Namely, we prove here the
\begin{theo}\label{theo:invariance}
Assume that $(\vartheta_t)_{tÊ\in (-\delta, \delta)}$ is a smooth family of contact forms such that their Reeb vector fields $X_t$ generate a contact Anosov flow for each $t$. Let $(E,\nabla)$ be an acyclic flat vector bundle. Then the map $t \mapsto \tau_{\vartheta_t}(\nabla)$ is real differentiable and we have
$$\frac{\dd}{\dd t} \tau_{\vartheta_t}(\nabla)= 0.$$
\end{theo}

\begin{rem}\label{rem:notacyclic2}
In view of Remark \ref{rem:notacyclic}, if $\nabla$ is not assumed acyclic, then it is not hard to see that the proof (given below) of Theorem \ref{theo:invariance} is still valid and we have that $\partial_t \tau_{\vartheta_t}(\nabla) = 0$ in $\det H^\bullet(M,\nabla)$.
\end{rem}

We will thus consider a family of contact forms and set $\vartheta = \vartheta_0$ and $X = X_0$. We also fix an acyclic flat vector bundle $(E,\nabla)$.

\subsection{Anisotropic spaces for a family of vector fields}\label{subsec:anisotropicuniform}
To study the dynamical torsion when the dynamics is perturbed, we construct with the help of \cite{bonthonneau2018flow} some anisotropic Sobolev spaces on which each $X_t$ has nice spectral properties. We refer to Appendix \ref{sec:continuityruelle} where we briefly recall the construction of these spaces.

By \S\ref{subsec:continuityruelle}, the set
$$
\left\{(t,s), \ s \notin \Res(\Lie_{X_t}^\nabla) \right\}
$$
is open in $(-\delta, \delta) \times \mathbb{C}$. Fix $\lambda \in (0,1)$ such that 
\begin{equation}\label{eq:circle}
\mathrm{Res}(\Lien) \cap \{|s| \leq \lambda\} \subset \{0\}.
\end{equation} Then for $t$ close enough to $0$, we have $\Res(\Lient) \cap \{|s|=\lambda\} = \emptyset$ so that the spectral projectors
\begin{equation}\label{eq:pit}
\Pi_t = \frac{1}{2i\pi}\int_{|s| = \lambda}(\Lient +s)^{-1} \dd s : \Omega^\bullet(M,E) \to \mathcal{D}^{'\bullet}(M,E)
\end{equation}
are well defined. The next proposition is a brief summary of the results from Appendix \ref{sec:continuityruelle}. We will denote for any $C,\rho > 0$,
\begin{equation}\label{eq:omegacrho}
\Omega(c, \rho)=\{\operatorname{Re}(s)>c\} \cup\{|s| \leq \rho\} \subset \mathbb{C}.
\end{equation}
\begin{prop}\label{prop:anisotropicuniform}
There is $c, \varepsilon_0 > 0$ such that for any $\rho > 0$ there exists anisotropic Sobolev spaces
$$
\Omega^\bullet(M,E) \subset \H^\bullet_{1} \subset \H^\bullet \subset \mathcal{D}^{'\bullet}(M,E),
$$
each inclusion being continuous with dense image, such that the following holds.
\begin{enumerate}
\item For each $t \in [-\varepsilon_0, \varepsilon_0]$, the family $s \mapsto \Lient + s$ is a holomorphic family of (unbounded) Fredholm operators $\H_{1}^\bullet \to \H^\bullet_1$ and $\H^\bullet \to \H^\bullet$ of index $0$ in the region $\Omega(c, \rho)$. Moreover
$$
\Lient \in \mathcal{C}^1\Bigl([-\varepsilon_0, \varepsilon_0], \Lie(\H^\bullet_1, \H^\bullet)\Bigr).
$$
\item For every relatively compact open region $\mathcal{Z} \subset \mathrm{int}~ \Omega(c,\rho)$ such that $\Res(\Lien) \cap \overline{\mathcal{Z}} = \emptyset$, there exists $t_\mathcal{Z} > 0$  such that
$$
\left(\Lient + s\right)^{-1} \in \mathcal{C}^0\Bigl([-t_\mathcal{Z}, t_\mathcal{Z}]_t, \mathrm{Hol}\bigl(\mathcal{Z}_s, \Lie(\H^\bullet_{1}, \H^\bullet)\bigr)\Bigr).
$$
\item $\Pi_t \in \mathcal{C}^1\Bigl([-\varepsilon_0, \varepsilon_0]_t, \Lie(\H^\bullet, \H^\bullet_1)\Bigr).$
\end{enumerate}
\end{prop}
We will thus fix such Hilbert spaces for some $\rho > c + 1$. We denote $C^\bullet_t = \mathrm{ran} \ \Pi_t \subset \mathcal{H}^\bullet,$
$\Pi = \Pi_{t=0}$ and $C^\bullet = \mathrm{ran} \ \Pi.$

\subsection{Variation of the torsion part}
Let $\Gamma_t : C^\bullet_t \to C^{n-\bullet}_t$ be the chirality operator associated to $X_t$, c.f. \S \ref{subsec:chirality}. The next lemma allows us to compute the variation of the finite dimensional torsion part of the dynamical torsion.
\begin{lemm}\label{lem:variationvectortorsion}
We have that $t \mapsto \tau(C^\bullet_t, \Gamma_t)$ is real differentiable and 
$$\frac{\dd }{\dd t}\tau(C^\bullet_t, \Gamma_t) =-\str{C^\bullet_t}\bigl( \Pi_t \vartheta_t  \iota_{\dot{X}_t}\bigr)\tau(C^\bullet_t, \Gamma_t),$$
where $\dot{X}_t = \displaystyle{\frac{\dd}{\dd t}X_t}.$ 
\end{lemm}

\begin{proof}
By Proposition \ref{prop:anisotropicuniform}, the operator $\Pi_t|_{C^\bullet} : C^\bullet \to C^\bullet_t$ is invertible for $t$ close enough to $0$ and we will denote by $Q_t$ its inverse. Then for $t$ close enough to $0$, one has 
$$\tau(C^\bullet_t, \Gamma_t) = \tau(C^\bullet, \tilde{\Gamma}_t),$$
where $\tilde{\Gamma}_t = Q_t \Gamma_t \Pi_t|_{C^\bullet} : C^\bullet \to C^\bullet$
because $\nabla$ and $\Pi_t$ commute and the image of a $\tilde{\Gamma}_t$ invariant basis of $C^\bullet$ by the projector $\Pi_t$ is a
$\Gamma_t$ invariant basis of $C^\bullet_t$.

Therefore \cite[Proposition 4.9]{braverman2007refined}
$$\frac{\dd }{\dd t}\tau(C^\bullet_t, \Gamma_t) = \frac{1}{2}\str{C^\bullet}\bigl(\dot{\tilde{\Gamma}}_t\tilde{\Gamma}_t\bigr)\tau(C^\bullet_t, \Gamma_t),$$
where $\dot{\tilde{\Gamma}}_t\ = \frac{\dd}{\dd t} \tilde{\Gamma}_t : C^\bullet \to C^\bullet$. Since $\Gamma_t$ and $\Pi_t$ commute, and by the two first points of Proposition \ref{prop:anisotropicuniform}, we can apply (\ref{eq:qtatpit}) to get
$$\tilde{\Gamma}_t = \Pi \Gamma_t \Pi|_{C^\bullet} + t\Pi\dot{\Gamma}\Pi + o_{C^\bullet \to C^\bullet}(t).$$
This leads to
$$
\dot{\tilde{\Gamma}}\tilde{\Gamma}= \Pi \dot \Gamma {\Gamma}|_{C^\bullet},
$$
where we removed the subscripts $t$ to signify that we take all the $t$-dependent objects at $t = 0.$
Therefore,
$$\frac{1}{2}\str{C^\bullet}\Bigl(\dot{\tilde{\Gamma}}{\tilde{\Gamma}}\Bigr) = \frac{1}{2}\str{C^\bullet}\Bigl(\Pi\dot\Gamma{\Gamma}\Bigr),$$ 
Now notice that $\Gamma_t^2 = 1$ implies $\Gamma \dot{\Gamma} + \dot{\Gamma}\Gamma = 0$. Therefore, for every $k \in \{0, \dots, r\}$,
$$\tr_{C^{n-k}}\Gamma \dot{\Gamma} = \tr_{C^k} \Gamma \Gamma \dot{\Gamma} \Gamma = \tr_{C^k}\dot{\Gamma}\Gamma = - \tr_{C^k}\Gamma \dot{\Gamma}.$$
 Therefore we only need to compute $\tr_{C^k}\left(\Gamma\dot{\Gamma} \right)$ for $k\in \{0,\dots,r\}$ to get the full
super trace $\str{C^\bullet}\left( \dot{\Gamma}\Gamma\right)$.
Since $n$ is odd we have
\begin{equation*}
\frac{1}{2}\str{C^\bullet}\Bigl(\dot{\tilde{\Gamma}}{\tilde{\Gamma}}\Bigr) = \frac{1}{2}\tr_{C^\bullet}\Bigl((-1)^{N+1}\Pi\Gamma\dot{\Gamma}\Bigr) = \sum_{k=0}^r (-1)^{k+1}\tr_{C^k} \left(\Pi\Gamma \dot{\Gamma}\right).
\end{equation*}
Let $k \in \{0, \dots, r\}$ and $\alpha \in \Omega^k(M)$. Using the decomposition
$$\alpha = (-1)^{k-1}\iota_{X_t}\alpha \wedge \vartheta_t + \bigl(\alpha + (-1)^k \iota_{X_t} \alpha \wedge \vartheta_t\bigr),$$
we get by definition of $\Gamma_t$
$$\Gamma_t \alpha = (-1)^{k-1}\iota_{X_t}\alpha \wedge (\dd \vartheta_t)^{r-k+1} + \bigl(\alpha + (-1)^k \iota_{X_t} \alpha \wedge \vartheta_t\bigr) \wedge (\dd \vartheta_t)^{r-k}\wedge \vartheta_t.$$
Therefore,
$$
\begin{aligned}
\dot{\Gamma}_t \alpha &= (-1)^{k-1}\iota_{\dot{X}_t} \alpha \wedge (\dd \vartheta_t)^{r-k+1} \\
& \quad+ (r-k+1)(-1)^{k-1}\iota_{X_t}\alpha \wedge \dd \dot{\vartheta}_t \wedge (\dd \vartheta_t)^{r-k} \\
& \quad + (-1)^k\Bigl(\iota_{\dot{X}_t} \alpha \wedge \vartheta_t + \iota_{X_t} \alpha \wedge \dot{\vartheta}_t\Bigr) \wedge (\dd \vartheta_t)^{r-k} \wedge\vartheta_t \\
& \quad + \Bigl(\alpha +(-1)^k \iota_{X_t} \alpha \wedge \vartheta_t \Bigr)\wedge (\dd \vartheta_t)^{r-k} \wedge \dot{\vartheta}_t \\
& \quad + (r-k)\Bigl(\alpha +(-1)^k \iota_{X_t} \alpha \wedge \vartheta_t \Bigr) \wedge \dd \dot{\vartheta}_t \wedge (\dd \vartheta_t)^{r-k-1} \wedge \vartheta_t
\end{aligned}
$$
Now we use the decompositions 
$$
\begin{aligned}
\dd \dot{\vartheta}_t &= -\iota_{X_t} \dd \dot{\vartheta}_t\wedge \vartheta_t + \bigl(\dd \dot{\vartheta}_t + \iota_{X_t} \dd \dot{\vartheta}_t\wedge \vartheta_t \bigr), \\
\dot{\vartheta}_t &= \dot{\vartheta}_t(X_t)\vartheta + \bigl(\dot{\vartheta}_t - \dot{\vartheta}_t(X_t) \vartheta\bigr), \\
\iota_{\dot{X}_t}\alpha &= (-1)^k \iota_{X_t} \iota_{\dot{X}_t} \alpha \wedge \vartheta_t + \bigl(\iota_{\dot{X}_t} \alpha + (-1)^{k+1} \iota_{X_t} \iota_{\dot{X}_t}\alpha \wedge \vartheta_t\bigr)
\end{aligned}
$$
to get, again by definition,
\begin{equation}\label{eq:calculusgammadot}
\begin{aligned}
\Gamma \dot{\Gamma} \alpha &= (-1)^{k-1} \bigl(\iota_{\dot{X}} \alpha + (-1)^{k+1} \iota_{X} \iota_{\dot{X}}\alpha \wedge \vartheta\bigr) \wedge \vartheta \\
& \quad \quad \quad \quad  + (-1)^{k-1} \left(\mathscr{L}^{r-k}\right)^{-1} \Bigl((-1)^k\iota_{X}\iota_{\dot{X}}\alpha \wedge (\dd \vartheta)^{r-k+1}\Bigr)\\
& \quad +   (r-k+1)\left(\mathscr{L}^{r-k+1}\right)^{-1}\Bigl((-1)^{k-1}\iota_X\alpha \wedge \bigl(\dd \dot{\vartheta} + \iota_{X} \dd \dot{\vartheta}\wedge \vartheta\bigr) \wedge (\dd \vartheta)^{r-k}\Bigr) \wedge \vartheta \\
& \quad \quad \quad \quad - (r-k+1)\left((-1)^{k-1}\iota_X \alpha\right) \wedge \iota_X \dd \dot{\vartheta} \\
& \quad + (-1)^k\iota_X\alpha \wedge \bigl(\dot{\vartheta}-\dot{\vartheta}(X) \vartheta\bigr) \\
& \quad   + \left(\mathscr{L}^{r-k+1}\right)^{-1}\Bigl(\left(\alpha +(-1)^k \iota_X \alpha \wedge \vartheta\right) \wedge (\dd \vartheta)^{r-k}  \wedge \bigl(\dot{\vartheta}-\dot{\vartheta}(X) \vartheta\bigr) \Bigr) \wedge \vartheta \\
& \quad \quad \quad \quad  + \left(\alpha +(-1)^k \iota_X \alpha \wedge \vartheta\right) \dot{\vartheta}(X) \\
& \quad + (r-k)\left(\mathscr{L}^{r-k}\right)^{-1}\Bigl(\left(\alpha +(-1)^k \iota_X \alpha \wedge \vartheta\right) \wedge \bigl(\dd \dot{\vartheta} + \iota_{X} \dd \dot{\vartheta}\wedge \vartheta\bigr) \wedge (\dd\vartheta)^{r-k-1} \Bigr),
\end{aligned} 
\end{equation}
where again we removed the subscripts $t$ to signify that we take everything at $t = 0.$
Now let $A_k : C^k_0 \to C^k_{0}$ (note that here $C^k_0$ is $C^k \cap \ker \iota_X$, cf \S \ref{subsec:chirality}, and not $C^k_t$ at $t=0$) defined by 
$$A_k u = (r-k)\left(\mathscr{L}^{r-k}\right)^{-1}\Bigl(u \wedge \bigl(\dd \dot{\vartheta} + \iota_{X} \dd \dot{\vartheta}\bigr) \wedge (\dd\vartheta)^{r-k-1} \Bigr).$$
Note that the maps defined by the second, the fourth, the fifth and the sixth terms of the right hand side of (\ref{eq:calculusgammadot}) are anti-diagonal, that is they have the form $\begin{pmatrix} 0 & \star \\ \star & 0 \end{pmatrix}$ in the decomposition $C^\bullet = C^{\bullet-1}_0 \wedge \vartheta \oplus C^\bullet_0$. Therefore, since $A_r = 0$ (we also set $A_{-1} = 0$),
\begin{equation}\label{eq:tracegammadot}
\begin{aligned}
\sum_{k=0}^r (-1)^{k+1}\tr_{C^k} \left(\Pi\Gamma \dot{\Gamma}\right) &= \sum_{k=0}^r (-1)^{k+1} \Bigl(\tr_{C^k} \Pi\vartheta \iota_{\dot{X}} + \tr_{C^k_0} \Pi \dot{\vartheta}(X)\Bigr) \\
& \quad + \sum_{k=0}^r (-1)^{k+1} \Bigl(\tr_{C^{k-1}_0} \Pi A_{k-1} + \tr_{C^k_0} \Pi A_k \Bigr) \\
&= \sum_{k=0}^r (-1)^{k+1} \Bigl(\tr_{C^k} \Pi \vartheta\iota_{\dot{X}} + \tr_{C^k_0} \Pi \dot{\vartheta}(X)\Bigr).
\end{aligned}
\end{equation}
But now note that if $\alpha = f \wedge \vartheta + g \in C^{k-1}_0 \wedge \vartheta \oplus C^k_0$ then 
$$\vartheta \wedge \iota_{\dot{X}} \alpha = \vartheta(\dot{X}) (f\wedge \vartheta) + \vartheta \wedge \iota_{\dot{X}} g.$$
This shows that for every $k \in \{0, \dots, n\}$ one has
\begin{equation}\label{eq:trck}
\tr_{C^k} \Pi\vartheta \iota_{\dot{X}} = \tr_{C^{k-1}_0} \Pi\vartheta(\dot{X}).
\end{equation}
Injecting this relation in (\ref{eq:tracegammadot}) we obtain, with $\vartheta(\dot{X}) = - \dot{\vartheta}(X)$ and the formula $\dot{\vartheta}(X)|_{C^{2r-k}_0} \mathscr{L}^{r-k} = \mathscr{L}^{r-k} \dot{\vartheta}(X)|_{C^k_0}$,
$$
\sum_{k=0}^r (-1)^{k+1}\tr_{C^k} \left(\Pi\Gamma \dot{\Gamma}\right) = \sum_{k=0}^r (-1)^{k+1} \left(\tr_{C^{k-1}_0} \Pi \vartheta(\dot X) - \tr_{C^{k}_0} \Pi \vartheta(\dot X)\right)=\sum_{k=0}^{2r} (-1)^k \tr_{C^k_0} \Pi\vartheta(\dot X).
$$
But this concludes since by (\ref{eq:trck}) we have
$$\sum_{k=0}^{2r}(-1)^{k} \tr_{C^k_0} \Pi{\vartheta}(\dot X) = \tr_{C^\bullet}\Bigl((-1)^{N+1} \Pi\vartheta \iota_{\dot{X}}\Bigr).$$
\end{proof}

\subsection{Variation of the rest}
Let us now interest ourselves in the variation of $t \mapsto \zeta_{X_t, \nabla}^{(\lambda, \infty)}(0)$, cf. \S\ref{subsec:cut}. For $t$ close enough to $0$, let $P_t : TM \to TM$ be defined by
$$
\begin{matrix}
P_t :  & \ker \vartheta & \oplus & \mathbb{R}X & \to & \ker \vartheta & \oplus & \mathbb{R}X_t, \\
& v & + & \mu X & \mapsto & v & + & \mu X_t.
\end{matrix}
$$
For simplicity, we will still denote $\Lambda^k ({}^TP_t) : \Lambda^k T^*M \to \Lambda^k T^*M$ by $P_t$. Then formula (5.4) of \cite{dang2018fried} gives that for $\Re(s)$ big enough, $t \mapsto \zeta_{X_t, \nabla}(s)$ is differentiable and we have for every $\varepsilon > 0$ small enough 
$$
\left.\frac{\dd}{\dd t}\right|_{t=0} \log \zeta_{X,\nabla}(s) =  (-1)^{q} s \ \strf \Bigl( \dot P (\Lie_{X}^\nabla+s)^{-1}\e^{-\varepsilon (\Lie_{X}^\nabla+s)}\Bigr),
$$
where $\dot P = \left.\frac{\dd}{\dd t}\right|_{t=0} P_t$. One can show that for every $k \in \{0, \dots, n\}$ and $\beta \in \Lambda^kT^*M$ one has
\begin{equation}\label{eq:A_t}
\dot{P} \beta = \vartheta \wedge \iota_{\dot{X}}\beta.
\end{equation}
Therefore (we differentiated at $t=0$ but we can do the same for small $t$)
\begin{equation}\label{eq:variationzeta}
\frac{\dd}{\dd t} \log \zeta_{X_t,\nabla}(s) = (-1)^q s \  \strf \Bigl(  \vartheta_t \iota_{\dot{X}_t} (\Lie_{X_t}^\nabla+s)^{-1}\e^{-\varepsilon (\Lie_{X_t}^\nabla+s)}\Bigr).
\end{equation}
Now let us compute the variation of the $[0, \lambda]$ part of $\zeta^{(\lambda, \infty)}(s)$.
\begin{lemm}
We have 
$$\frac{\dd}{\dd t} \log \grdet{C^\bullet_t} \left(\Lient + s\right)^{(-1)^{q+1}} = (-1)^{q+1} \str{C^\bullet_t} \Bigl(  \vartheta_t \iota_{\dot{X}_t} \Lie_{X_t}^\nabla (\Lie_{X_t}^\nabla + s)^{-1}\Bigr).$$
\end{lemm}

\begin{proof}
We are in a position to apply Lemma \ref{lem:projector2} which gives
$$\frac{\dd}{\dd t} \log \grdet{C^\bullet_t} \left(\Lient + s\right)^{(-1)^{q+1}} =(-1)^{q+1} \grtr{C^\bullet_t} \Bigl( \Pi_t \Lie_{\dot{X}_t}^\nabla (\Lie_{X_t}^\nabla + s)^{-1}\Bigr).$$
Denote $A_t = P_t^{-1}\dot{P}_t.$ Then one can verify that
$$\iota_{X_t} = P_t^{-1} \iota_X P_t,$$
which leads to 
$$\Lie_{\dot{X}_t}^\nabla = -\nabla A_t \iota_{X_t} + \nabla \iota_{X_t}A_t - A_t \iota_{X_t}\nabla + \iota_{X_t}A_t \nabla.$$
Using
$$
\begin{aligned}
(-1)^NN \nabla &= \nabla (-1)^{N+1}(N+1), \\
(-1)^N N \iota_{X_t} &= \iota_{X_t} (-1)^{N-1}(N-1),
\end{aligned}
$$
and the cyclicity of the trace, we get since $(\Lie_{X_t}^\nabla+s)^{-1}$ commute with $\iota_{X_t}$ and $\nabla$,
$$
\begin{aligned}
&\begin{split}
\tr_{C^\bullet_t} \Bigl((-1)^{N+q+1} N\Pi_t \Lie_{\dot{X}_t}^\nabla &(\Lie_{X_t}^\nabla + s)^{-1}\Bigr) 
\\&= {(-1)^{q+1}\tr_{C^\bullet_t} \biggl(\Pi_t A_t \Bigl((-1)^{N}(N+1)\iota_{X_t}\nabla + (-1)^NN\nabla \iota_{X_t} }\\
&{\hspace{115pt} - (-1)^NN\iota_{X_t}\nabla -(-1)^N(N-1)\nabla\iota_{X_t}\Bigr)(\Lie_{X_t}^\nabla+s)^{-1}\biggr)} \\
&= (-1)^{q+1}\tr_{C_t^\bullet}\Bigl((-1)^N\Pi_t A_t \Lie_{X_t}^\nabla(\Lie_{X_t}^\nabla+s)^{-1}\Bigr) \\
\end{split}
\end{aligned}
$$
Therefore using (\ref{eq:A_t}) again this concludes, because $P_{t=0} = \id$.
\end{proof}

\subsection{Proof of Theorem \ref{theo:invariance}}\label{subsec:proofinvariance}

Combining this lemma and (\ref{eq:variationzeta}) we obtain that for $\Re(s)$ big enough and $t$ small enough
\begin{equation}\label{eq:fraczetalambda}
\begin{aligned}
\displaystyle{ \frac{\zeta^{(\lambda, \infty)}_{X_t, \nabla}(s)}{\zeta^{(\lambda, \infty)}_{X_0, \nabla}(s)}} &=  \exp\Biggl( -s \int_0^t \strf \left(\vartheta_\tau \iota_{\dot{X}_\tau}(\Lie_{X_\tau}^\nabla+s)^{-1}\e^{-\varepsilon(\Lie_{X_\tau} + s)}\right) \dd \tau \\
& \hspace{80pt} - \int_0^t\str{C^\bullet_\tau}\left(\Pi_\tau \vartheta_\tau\iota_{\dot{X}_\tau}\Lie_{X_\tau}^\nabla(\Lie_{X_\tau}^\nabla+s)^{-1}\right) \dd \tau \Biggr)^{(-1)^{q+1}}.
\end{aligned}
\end{equation}
Note that for every $s \notin \Res(\Lie_{X_t}^\nabla)$ we have
$$
\Lie_{X_t}^\nabla (\Lie_{X_t}^\nabla + s)^{-1} = \id - s(\Lie_{X_t}^\nabla+s)^{-1},
$$
so that
\begin{equation}\label{eq:xpluss}
\str{C^\bullet_t}\Pi_t \vartheta_t\iota_{\dot{X}_t}\Lie_{X_t}^\nabla(\Lie_{X_t}^\nabla+s)^{-1} = \str{C^\bullet_t}\Pi_t \vartheta_t \iota_{\dot X_t} - s\str{C^\bullet_t} \Pi_t \vartheta_t \iota_{\dot{X}_t} (\Lient+s)^{-1}.
\end{equation}
We now fix $s_0 \in \mathbb{C}$ with $\Re(s_0)$ big enough so that (\ref{eq:fraczetalambda}) is valid and a smooth path $c : [0,1] \to \mathbb{C}$ with $c(0) = 0$, $c(1) = s_0$ and
$$
c(u) \notin \Res(\Lien), \quad u \in (0,1].
$$
Let $\delta, t_0 > 0$ small enough so that
\begin{equation}\label{eq:2delta}
\mathrm{dist}\Bigl(\{|s|= \lambda\} \cup (V_\delta \cap \{|s|Ê\geq \lambda\}), \ \Res(\Lient)\Bigr) \geq 2\delta, \quad |t|Ê\leq t_0,
\end{equation}
where $V_\delta$ is the open $\delta$-neighborhood of $\mathrm{Im} \ c$. We moreover ask that
$$
(\Res(\Lient) \cap \{|s|Ê\leq \lambda\})Ê\subset \{|s|Ê\leq \delta\} \ \text{  and  } \ (V_\delta \cap \{|s| \geq \lambda\}) \cap \Res(\Lient) = \emptyset.
$$
For $t \in [-t_0, t_0]$ and $s \notin \Res(\Lient)$ we define
\begin{equation}\label{eq:ytresolvent}
Y_t(s) = \left(\Lient + s\right)^{-1}(\id - \Pi_t).
\end{equation}
Then by (\ref{eq:laurent}), we have that $s \mapsto Y_t(s)$ is holomorphic on a neighborhood of $\{|s|Ê\leq \lambda\}$ for each fixed $t$. 
This implies
\begin{equation}\label{eq:yt}
\displaystyle{Y_t(s)} = \sum_{n=0}^\infty Y_{t,n} s^n, \quad |s| < \lambda, \quad |t|Ê\leq t_0,
\end{equation}
with
\begin{equation}\label{eq:ytn}
Y_{t,n} = \frac{1}{2i \pi}\int_{|s| = \lambda} Y_t(s) s^{-n-1} \dd s.
\end{equation}
Therefore, for every $|t|Ê\leq t_0$ one has $\|Y_{t,n}\|_{\H \to \H} \leq 2\delta\lambda^{-n-1}$ by (\ref{eq:2delta}) and Proposition \ref{prop:anisotropicuniform}. 

Let $\mathcal{Q}_t(s)$ denote the Schwartz Kernel of the operator $Q_t(s) = \left(\Lient+s\right)^{-1} \e^{-\varepsilon\left(\Lient+s\right)}$. Then \cite[Proposition 6.3]{dang2018fried} gives that the map
$$
[-t_0, t_0] \times \{|s| = \lambda\} \ni (t,s) \mapsto \mathcal{Q}_t(s) \in \mathcal{D}^{'n}_\Gamma(M \times M, E^\vee \boxtimes E)
$$
is bounded for some closed conic subset $\Gamma \subset T^*\left(M \times M\right)$ not intersecting the conormal of the diagonal. Moreover by \S\ref{subsec:wfsetprojectors}, we have that $[-t_0, t_0] \ni t \mapsto \Pi_t$ is bounded in $\mathcal{D}^{'n}_{W_s \times W_u}(M \times M, E^\vee \boxtimes E)$, and so is the map $[-t_0, t_0] \times \{|s| = \lambda\} \mapsto \left(\Lient + s\right)^{-1}\Pi_t$. As a consequence (\ref{eq:ytresolvent}), (\ref{eq:yt}) and (\ref{eq:ytn}) imply that the map
\begin{equation}\label{eq:boundedwf}
[-t_0, t_0] \times \{|s|Ê\leq 3\delta/2\} \ni (t,s) \mapsto \mathcal{Y}_t(s) \in \mathcal{D}^{'n}_\Gamma(M \times M, E^\vee \boxtimes E),
\end{equation}
is bounded, where $\mathcal{Y}_t(s)$ is the Schwartz kernel of the operator $Y_t(s)\e^{-\varepsilon\left(\Lient+s\right)}.$ We also know that this map is continuous when it is seen as a map valued in $\mathcal{D}^{'n}$ thanks to the last point of Proposition \ref{prop:anisotropicuniform}; therefore this map is continuous when valued in $\mathcal{D}^{'n}_\Gamma(M \times M, E^\vee \boxtimes E$, cf. \cite[\S8.4]{hor1}. Therefore we obtain with \S\ref{subsec:flattrace} that
\begin{equation}\label{eq:strfyt}
\strf \vartheta \iota_{\dot X_t} Y_t(s) \in \mathcal{C}^0\Bigl([-t_0, t_0], \mathrm{Hol}\bigl( \{|s|Ê\leq 3\delta / 2\}\bigr)\Bigr).
\end{equation}
But now apply \cite[Theorem 4]{dang2018fried} to obtain that
\begin{equation}\label{eq:strfqt}
\strf \vartheta \iota_{\dot X_t} Q_t(s) \in \mathcal{C}^0\Bigl([-t_0, t_0], \mathrm{Hol}\bigl(V_\delta \cap \{|s|Ê\geq 5\delta / 4\}\bigr)\Bigr).
\end{equation}
Since the flat trace coincides with the usual trace for operators of finite rank,
\begin{equation*}
\begin{aligned}
\strf \vartheta_t \iota_{\dot X_t} Q_t(s) -\str{C^\bullet}\Pi_t \vartheta_t \iota_{\dot X_t}(\Lient +s)^{-1} &= \strf \vartheta_t \iota_{\dot X_t} \left(\Lient+s\right)^{-1}(\id - \Pi_t)\e^{-\varepsilon(\Lient + s)}  \\
& \ \quad   + \str{C^\bullet}\Pi_t \vartheta_t \iota_{\dot X_t}(\Lient +s)^{-1}\left(\e^{-\varepsilon(\Lient + s)}- \id\right).
\end{aligned}
\end{equation*}
Then (\ref{eq:strfyt}), (\ref{eq:strfqt}) and (\ref{eq:ytresolvent}) imply that the right hand side of the last equation is continuous with respect to $t$ with values in holomorphic functions on $(V_\delta \cap \{|s|Ê\geq 5\delta / 4\}) \cup \{|s|Ê\leq 3\delta /2\}$ (indeed $s \mapsto (\Lient +s)^{-1}\left(\e^{-\varepsilon(\Lient + s)}- \id\right)$ is holomorphic of $C^\bullet_t$), and so is the left hand side. As a consequence, (\ref{eq:xpluss}) shows that both members of (\ref{eq:fraczetalambda}) are holomorphic on this region and
$$
\zeta^{(\lambda, \infty)}_{X_t, \nabla}(0) = \zeta^{(\lambda, \infty)}_{X_0, \nabla}(0) \exp\left(-\int_0^t \str{C^\bullet_\tau} \Pi_\tau \vartheta \iota_{\dot X_\tau} \dd \tau\right)^{(-1)^{q+1}}.
$$
Comparing this with Lemma \ref{lem:variationvectortorsion} we obtain Theorem \ref{theo:invariance} by definition of the dynamical torsion, cf \S\ref{subsec:defdyn}.

\section{Variation of the connection}\label{sec:variationconnexion}\label{sec:variationconnexion}
In this section we compute the variation of the dynamical torsion when the connection is perturbed. This formula will be crucial to compare the dynamical torsion and Turaev's refined combinatorial torsion.

\subsection{Real-differentiable families of flat connections}\label{subsec:familyconnexion}
Let $U \subset \mathbb{C}$ be some open set and consider $\nabla(z), ~z \in U$, a family of flat 
connections on $E$. We will assume that the map $z \mapsto \nabla(z)$ is $\mathcal{C}^1$, that is, there exists continuous maps $z \mapsto \mu_z, \nu_z \inÊ\Omega^1(M, \mathrm{End}(E))$ such that for any $z_0 \in U$ one has
\begin{equation}\label{eq:connexionholomorphic}
\nabla(z) = \nabla(z_0) + \Re(z-z_0) \mu_{z_0} + \mathrm{Im}(z-z_0) \nu_{z_0} + o(z-z_0),
\end{equation}
where $o(z-z_0)$ is understood in the Fr\'echet topology of $\Omega^1(M, \mathrm{End}(E))$. We will denote for any $\sigma \in \mathbb{C}$
\begin{equation}\label{eq:alpha_z}
\alpha_{z_0}(\sigma) = \Re(\sigma) \mu_{z_0} + \Im(\sigma) \nu_{z_0} \in \Omega^1(M,\mathrm{End}(E)).
\end{equation}
Note that since the connections $\nabla(z)$ are assumed to be flat, we have
\begin{equation}\label{eq:nablaalpha}
[\nabla(z), \alpha_z(\sigma)] = \nabla(z) \alpha_z(\sigma) + \alpha_z(\sigma) \nabla(z) = 0.
\end{equation}

\subsection{A cochain contraction induced by the Anosov flow}\label{subsec:cochain}
For $z \in U$ let
\begin{equation}\label{eq:devs=0}
\bigl(\Lienz+s\bigr)^{-1} =  \sum_{j=1}^{J(0)} \frac{\left(-\Lienz\right)^{j-1}\Pi_0(z)}{s^j} + Y(z) + \dom(s)
\end{equation}
be the development (\ref{eq:laurent}) for the resonance $s_0 = 0$. Let $C^\bullet(0 ; z) = \mathrm{ran} ~\Pi_0(z)$. Recall from \S\ref{subsec:topologyruelle} that since $\nabla(z)$ is acyclic, the complex $(C^\bullet(0; z), \nabla(z))$ is acyclic. Therefore there exists a cochain contraction $k(z) : C^\bullet(0; z) \to C^\bullet(0; z)$, i.e. a map of degree $-1$ such that 
\begin{equation}\label{eq:cochain0}
\nabla(z) k(z) + k(z) \nabla(z) = \id_{C^\bullet(0;z)}.
\end{equation}
We now define
\begin{equation}\label{eq:K}
K(z) = \iota_XY(z)(\id - \Pi_0(z)) + k(z)\Pi_0(z) ~ : ~ \Omega^\bullet(M,E) \to \mathcal{D}^{'\bullet}(M, E).
\end{equation}
A crucial property of the operator $K$ is that it satisfies
the chain homotopy equation
\begin{equation}\label{eq:cochainanosov}
\nabla(z) K(z) + K(z) \nabla(z) = \id_{\Omega^{\bullet}(M,E)},
\end{equation}
as follows from the development (\ref{eq:devs=0}).

\subsection{The variation formula}\label{subsec:variationformula}
For simplicity, we will set for every $z \in U$
$$
\tau(z) = \tau_\vartheta(\nabla(z)).
$$
The operators $K(z)$ defined above are involved in the variation formula of the dynamical torsion, as follows.
\begin{prop}\label{prop:variationconnexion}
The map $z \mapsto \tau(z)$ is real differentiable; we have for every $z \in U$ and $\varepsilon > 0$ small enough
\begin{equation}\label{eq:formulavariation}
\dd (\log \tau)_{z} \sigma =  -\strf \left(\alpha_{z}(\sigma) K(z) \e^{-\varepsilon \Lie_X^{\nabla(z)}}\right), \quad \sigma \in \mathbb{C}.
\end{equation}
\end{prop}
The proof of the previous proposition is similar of that of the last subsection, i.e. we compute the variation of each part of the dynamical torsion. The rest of this section is devoted to the proof of Proposition \ref{prop:variationconnexion}.

\subsection{Anisotropic Sobolev spaces for a family of connections}\label{subsec:anisotropicuniformconnexion}
Fix some $z_0 \in U$. Recall from \S\ref{subsec:anisotropicuniform} that we chose some anisotropic Sobolev spaces $\H^\bullet_{1} \subset \H^\bullet$. Notice that
\begin{equation}\label{eq:lienz}
\Lienz = \Lienzo + \beta(z)(X),
\end{equation}
where $\beta(z) \in \Omega^1(M,\mathrm{End}(E))$ is defined by
$$\nabla(z) = \nabla(z_0) + \beta(z).$$
Therefore (\ref{eq:connexionholomorphic}) implies that $\Lienz - \Lienzo$ is a $\mathcal{C}^1$ family of pseudo-differential operators of order $0$, and thus forms a $\mathcal{C}^1$ family of bounded operators $\H^\bullet \to \H^\bullet$ and $\H^\bullet_{1} \to \H^\bullet_{1}$ by construction of the anisotropic spaces and standard rules of pseudo-differential calculus (see for example \cite{faure2011upper}). As a consequence and thanks to Proposition \ref{prop:anisotropicuniform}, we are in position to apply \cite[Theorem 3.11]{kato1976perturbation}; thus if $\delta$ is small enough we have that
\begin{equation}\label{eq:continuousspectrum}
R_\rho = \left\{(z,s) \in \mathbb{C}^2, \ |z-z_0| < \delta, \ s \in \Omega(c,\rho), \ s \notin \sigma_{\H^\bullet}(\Lienz)\right\} \text{ is open},
\end{equation}
where $\sigma_{\H^\bullet}(\Lienz)$ denotes the resolvent set of $\Lienz$ on $\H^\bullet$, and $\Omega(c,\rho)$ is defined in (\ref{eq:omegacrho}). Moreover (\ref{eq:connexionholomorphic}) and (\ref{eq:lienz}) imply that for any open set $\mathcal{Z} \subset \Omega(c,\rho)$ such that $\Res\left(\Lie^{\nabla(z_0)}_X\right) \cap \overline{\mathcal{Z}} = \emptyset$, there exists $\delta_\mathcal{Z} > 0$ such that for any $j \in \{0,1\}$,
\begin{equation}\label{eq:biresolvent}
\left(\Lienz+s\right)^{-1} \in \mathcal{C}^1\Bigl( \bigl\{|z - z_0| < \delta_{\mathcal{Z}} \bigr\}, ~\mathrm{Hol}\bigl(\mathcal{Z}_s, \Lie\left(\H^\bullet_j, \H^{\bullet}_j\right)\bigr)\Bigr).
\end{equation}
For all $z$, the map $s \mapsto \left(\Lienz+s\right)^{-1}$ is meromorphic in the region $\Omega(c,\rho)$ with poles (of finite multiplicity) which coincide with the resonances of $\Lienz$ in this region.

 Moreover, the arguments from the proof of \cite[Proposition 3.4]{dyatlov2013dynamical} can be made uniformly for the family $z \mapsto \left(\mathcal{L}_X^{\nabla(z)} +s\right)^{-1}$ to obtain that for some closed conic set $\Gamma\subset T^*\left(M\times M\right)$ not intersecting the conormal to the diagonal and any $\varepsilon>0$ small enough, the map $(s,z) \mapsto \mathcal{K}(s,z)$ is bounded from $\mathcal{Z} \times \{|z-z_0|Ê< \delta_\mathcal{Z}\}$ with values $\mathcal{D}^\prime_\Gamma(M\times M,\pi_1^*E^\vee \otimes \pi_2^* E)$, where $\mathcal{K}(s,z)$ is the Schwartz kernel of the shifted resolvent $\left(\mathcal{L}_X^{\nabla(z)} +s\right)^{-1}\e^{-\varepsilon\mathcal{L}^{\nabla(z)}_X}$.

\subsection{A family of spectral projectors}
Fix $\lambda \in (0,1)$ such that 
\begin{equation}\label{eq:only0}
\{s \in \mathbb{C}, \ |s| \leq \lambda\} \cap \Res\left(\Lie_{X}^{\nabla(z_0)}\right) \subset \{0\}.
\end{equation}
Thanks to (\ref{eq:continuousspectrum}), if $z$ is close enough to $z_0$,
\begin{equation}\label{eq:cutz}
\{s \in \mathbb{C}, \ |s| = \lambda\}\cap \Res\left(\Lienz\right) = \emptyset,
\end{equation}
by compacity of the circle.
For $z \in U$ we will denote by 
\begin{equation}\label{eq:pi(z)}
\Pi(z)  = \frac{1}{2i\pi} \int_{|s|=\lambda} \left(\Lienz+s\right)^{-1} \dd s
\end{equation}
the spectral projector of $\Lienz$ on generalized eigenvectors for resonances in $\{s \in \mathbb{C}, \ |s|\leq \lambda\}$, and $C^\bullet(z) = \mathrm{ran}~\Pi(z)$.
It follows from (\ref{eq:biresolvent}), (\ref{eq:cutz}) and (\ref{eq:pi(z)}) that the map
$$z \mapsto \Pi(z) \in \Lie(\H^\bullet_j, \H^{\bullet}_j)$$
is $\mathcal{C}^1$ for $j = 0,1$. We can therefore apply \ref{lem:projector3} to get, for $\delta$ small enough,
\begin{equation}\label{eq:projectoranalytic}
\Pi(z) \in \mathcal{C}^1\Bigl(\{|z-z_0| < \delta\}_z, ~\Lie(\H^\bullet, \H^\bullet_1)\Bigr).
\end{equation}

\subsection{Variation of the finite dimensional part}\label{subsec:variationconnexionfinitedimensional}
Because $(C^\bullet(z_0), \nabla(z_0))$ is acyclic, there exists a cochain contraction $k(z_0) : C^\bullet(z_0) \to C^{\bullet-1}(z_0)$, cf \S\ref{subsec:analyticfamiliesofdifferentials}. The next lemma computes the variation of the finite dimensional part of the dynamical torsion.
\begin{lemm}\label{lem:variationconnexiontorsion}
The map $z \mapsto c(z)= \tau(C^\bullet(z), \Gamma)$ is real differentiable at $z=z_0$ and
$$
\dd (\log c)_{z_0}\sigma = - \str{C^\bullet} \Pi(z_0) \alpha_{z_0}(\sigma) k(z_0), \quad \sigma \in \mathbb{C}.
$$
\end{lemm}
\begin{proof}
By continuity of the family $z \mapsto \Pi(z)$, we have that $\Pi(z)|_{C^\bullet(z_0)} : C^\bullet(z_0) \to C^\bullet(z)$ is an isomorphism for $|z-z_0|$ small enough, of inverse denoted by $Q(z)$. For those $z$ we denote by $\widehat{C}^\bullet(z)$ the graded vector space $C^\bullet(z_0)$ endowed with the differential
$$
\widehat{\nabla}(z) = Q(z) \nabla(z) \Pi(z) : C^\bullet(z_0) \to C^{\bullet}(z_0).
$$
Then because $\Gamma$ commutes with every $\Pi(z)$ one has
\begin{equation}\label{eq:conjugation}
\tau(\widehat{C}^\bullet(z), \Gamma) = \tau(C^\bullet(z), \Gamma)
\end{equation}
By (\ref{eq:projectoranalytic}) we can apply (\ref{eq:qtatpit}) in the proof of Lemma \ref{lem:projector2} which gives for any $h$ small enough
$$
\widehat{\nabla}(z_0 + \sigma)\Pi(z_0) = \Pi(z_0) \nabla(z_0) \Pi(z_0) +  \Pi(z_0) \alpha_{z_0}(\sigma) \Pi(z_0) + o_{C^\bullet(z_0) \to C^\bullet(z_0)}(\sigma).
$$
Therefore the real differentiable version of Lemma \ref{lem:5.1} implies the desired result.
\end{proof}

\subsection{Variation of the zeta part}
We give a first 
Proposition which computes the variation of the Ruelle zeta function in its convergence region.
\begin{prop}[Variation of the dynamical zeta function]\label{prop:variationzetaconnection}
For $\Re(s)$ big enough, the map $z \mapsto g_s(z) = \zeta_{X,\nabla(z)}(s)$ is $\mathcal{C}^1$ near $z=z_0$ and we have for every $\varepsilon > 0$ small enough
$$\dd (\log g_s)_{z_0} \sigma = (-1)^{q+1}\e^{-\varepsilon s}\strf\left(\alpha_{z_0}(\sigma) \iota_X\left(\Lienzo+s\right)^{-1} \e^{-\varepsilon \Lienzo}\right).$$
\end{prop}

\begin{proof}
Let $\varphi^t$ denote the flow of $X$. For $\gamma \in \mathcal{G}_X$, $\dd \varphi^{-\ell(\gamma)}|\gamma$ will denote $\dd \varphi^{-\ell(\gamma)}$ taken at any point of the image of $\gamma$; this ambiguity will not stand long since another choice of base point will lead to a conjugated linear map, and we aim to take traces. We have the standard factorization, for $\Re(s)$ big enough and any $z$ near $z_0$,
\begin{equation}\label{eq:zetanabla}
g_s(z) = \exp \sum_{k=0}^n (-1)^k k \sum_{\gamma \in \mathcal{G}_X} \frac{\ell^\#(\gamma)}{\ell(\gamma)} \tr \rho_{\nabla(z)}(\gamma) \e^{-s\ell(\gamma)} \frac{\tr \Lambda^k (\dd \varphi^{-\ell(\gamma)})|_\gamma}{\det(I-P_\gamma)},
\end{equation}
where $P_\gamma = \left.\dd \varphi^{-\ell(\gamma)}\right|_{E_u \oplus E_s}$ is the linearized Poincar\'e map of $\gamma$, and $\ell^\#(\gamma)$ is the primitive period of $\gamma$. 
Now (\ref{lem:holonomy}) implies
$$
 \tr \rho_{\nabla(z_0+\sigma)}(\gamma) = \tr \rho_{\nabla(z_0)}(\gamma) -\tr \left(\rho_{\nabla(z_0)}(\gamma) \int_{\gamma} \alpha_{z_0}(\sigma)(X)\right) + o(\sigma)\ell(\gamma).
$$
As a consequence, the sum in (\ref{eq:zetanabla}) is $\mathcal{C}^1$ near $z=z_0$ for $\Re(s)$ big enough, and
$$\dd (\log g_s)_{z_0}\sigma = -\sum_{k=0}^n (-1)^k k \sum_{\gamma \in \mathcal{G}_X} \frac{\ell^\#(\gamma)}{\ell(\gamma)} \tr \left(\rho_{\nabla(z_0)}(\gamma) \int_\gamma \alpha_{z_0}(\sigma)(X)\right) \e^{-s\ell(\gamma)} \frac{\tr \Lambda^k (\dd \varphi^{-\ell(\gamma)})|_\gamma}{\det(I-P_\gamma)}.$$
Now a slight extension of Guillemin trace formula \cite{guillemin1977lectures} gives, in $\mathcal{D}'(\mathbb{R}_{>0})$,
$$\tr^\flat \left.\alpha_{z_0}(\sigma)(X) \e^{-t\Lien}\right|_{\Omega^k(M,E)} = \sum_\gamma \frac{\ell^\#(\gamma)}{\ell(\gamma)}\tr \left(\rho_{\nabla(z_0)}(\gamma)\int_\gamma\alpha_{z_0}(\sigma)(X)\right) \frac{\tr \Lambda^k \dd \varphi^{-\ell(\gamma)}}{\left|\det(I-P_\gamma)\right|}\delta(t-\ell({\gamma})),$$
where $\delta$ is the Dirac distribution. But now recall from \S\ref{subsec:theruelle} that $|\det(I-P_\gamma)| = (-1)^q \det(I-P_\gamma)$. Therefore, if $\varepsilon > 0$ satisfies $\varepsilon < \ell(\gamma)$ for all $\gamma$, arguing exactly as in \cite[\S 4]{dyatlov2013dynamical}, with (\ref{eq:resolvent}) in mind,
$$
\dd (\log g_s)_{z_0}\sigma= \e^{-\varepsilon s} (-1)^{q+1}\grtrf \left(\alpha_{z_0}(\sigma)(X)\left(\Lienzo + s\right)^{-1}\e^{-\varepsilon \Lienzo}\right).
$$

Now it remains to turn the graded trace $\grtrf $ into a super trace $\strf $ keeping in mind the relation $\grtrf=\strf \left(N \cdot \right) $ where $N$ is the number operator, cf. \S\ref{subsec:flattrace}. 
Note that $\alpha_{z_0}(\sigma)(X) = [\alpha_{z_0}(\sigma), \iota_X] = \alpha_{z_0}(\sigma) \circ \iota_X + \iota_X \circ \alpha_{z_0}(\sigma)$. We therefore have
$$
\begin{aligned}
N  \alpha_{z_0}(\sigma)(X)&= N [\alpha_{z_0}(\sigma), \iota_X] \\
&= N \alpha_{z_0}(\sigma) \iota_X + \iota_X(N-1)\alpha_{z_0}(\sigma) \\
&= N \alpha_{z_0}(\sigma) \iota_X - (N-1)\alpha_{z_0}\iota_X + [(N-1)\alpha, \iota_X].
\end{aligned}
$$
Since $\iota_X$ commutes with $\left(\Lienzo + s\right)^{-1}\e^{-\varepsilon \Lienzo}$ one finally obtains
$$
\begin{aligned}
N  \alpha_{z_0}(\sigma)(X) = \alpha_{z_0}(\sigma) \iota_X &\left(\Lienzo + s\right)^{-1}\e^{-\varepsilon \Lienzo} \\
&+ \left[(N-1) \alpha_{z_{0}}(\sigma)\left(\mathcal{L}_{X}^{\nabla\left(z_{0}\right)}+s\right)^{-1} \mathrm{e}^{-\varepsilon \mathcal{L}_{X}^{\nabla\left(z_{0}\right)}}, ~\iota_{X}\right].
\end{aligned}
$$
This concludes by cyclicity of the flat trace.
\end{proof}

The following lemma is a direct consequence of Lemma \ref{lem:projector2} and the fact that $\Pi_0(z_0) = \Pi(z_0)$ by (\ref{eq:only0}).
\begin{lemm}
For $\Re(s)$ big enough, the map $z \mapsto h_s(z) = \grdet{C^\bullet(z)}\left(\Lienz+s \right)^{(-1)^{q+1}}$ is $\mathcal{C}^1$ near $z=z_0$ and
$$\dd (\log h_s)_{z_0}\sigma = (-1)^{q+1}\str{C^\bullet(z_0)}\left(\Pi_0(z_0) \alpha_{z_0}(\sigma)  \iota_X \left(\Lienzo + s\right)^{-1}\right).$$
\end{lemm}

\subsection{Proof of Proposition \ref{prop:variationconnexion}}
Combining the two lemmas of the preceding subsection we obtain for $\Re(s)$ big enough, the map $z \mapsto \zeta^{(\lambda, \infty)}_{X,\nabla(z)}(s) = g_s(z) / h_s(z)$ is real differentiable at $z=z_0$ (and therefore on $U$ since we may vary $z_0$). Moreover for every $\varepsilon > 0$ small enough
\begin{equation}\label{eq:variationzetabigs}
\begin{aligned}
\dd \left(\log \frac{g_s}{h_s}\right)_z \sigma = (-1)^{q+1}\Biggl(&\e^{-\varepsilon s}\strf\alpha_{z}(\sigma) \iota_X\left(\Lienz+s\right)^{-1} \e^{-\varepsilon \Lienz}  \\ 
& \quad \quad  \quad - \str{C^\bullet(z)}\Pi_0(z) \alpha_{z}(\sigma)  \iota_X \left(\Lienz + s\right)^{-1}\Biggr).
\end{aligned}
\end{equation}
This gives the variation of $\zeta_{X,\nabla(z)}^{(\lambda, \infty)}(s)$ for $\Re(s)$ big enough. To obtain the variation of $b(z) = \zeta_{X,\nabla(z)}^{(\lambda, \infty)}(0)$, we can reproduce the arguments made in \S\ref{subsec:proofinvariance} to obtain
$$
(-1)^{q+1} \dd \left(\log b \right)_{z} \sigma = \strf \left(\alpha_z(\sigma) \iota_X Y(z) (\id - \Pi_0(z)) \e^{-\varepsilon \Lienz} \right)+ \str{C^\bullet(z)} \Bigl(\Pi_0(z) \alpha_z(\sigma) \iota_X Q_z(\varepsilon)\Bigr),
$$
where
$$
Q_z(\varepsilon) = \sum_{n \geq 1}\frac{(-\varepsilon)^n}{n!}\left(\Lienz\right)^{n-1} : C^\bullet(z) \to C^\bullet(z).
$$
Recall that if $c(z) = \tau(C^\bullet(z),\Gamma)$ one has $\tau(z) = c(z)b(z)^{(-1)^{q}}.$ Therefore Lemma \ref{lem:variationconnexiontorsion} gives, with what precedes,
\begin{equation}\label{eq:finalvariation}
\begin{aligned}
\dd (\log \tau)_z \sigma = - \strf &\Bigl(\alpha_{z}(\sigma) K(z) \e^{-\varepsilon \Lie_X^{\nabla(z)}}\Bigr) \\  &- \str{C^\bullet(z)}\left(\Pi_0(z) \alpha_z(\sigma) \Bigl(k(z) \left(\id - \e^{-\varepsilon \Lienz}\right) + \iota_XQ_z(\varepsilon)\Bigr)\right).
\end{aligned}
\end{equation}
We have $\id-\e^{-\varepsilon\Lienz} = -\Lienz Q_z(\varepsilon)$, which leads to
$$
\iota_XQ_z(\varepsilon) +k(z) \left(\id - \e^{-\varepsilon \Lienz}\right) = \Bigl( \iota_X - k(z) \Lienz \Bigr)Q_z(\varepsilon).
$$
But now since $k(z)$ is a cochain contraction, we get
$$
\iota_X - k(z) \Lienz = [\nabla(z), k(z)\iota_X].
$$
Because $\nabla(z)$ commutes with $\Pi_0(z)$ and $\Lienz$, we obtain with (\ref{eq:nablaalpha})
$$
\Bigl[\nabla(z), \Pi_0(z) \alpha_z(\sigma) k(z) \iota_X Q_z(\varepsilon) \Bigr] = \Pi_0(z) \alpha \Bigl( \iota_XQ_z(\varepsilon) +k(z) \left(\id - \e^{-\varepsilon \Lienz}\right)\Bigr).
$$
This concludes by (\ref{eq:finalvariation}) and the cyclicity of the trace.

\section{Euler structures, Chern-Simons classes}\label{sec:turaevtorsion}
The Turaev torsion is defined using \textit{Euler structures}, introduced by Turaev~\cite{turaev1990euler}, whose purpose is to fix sign ambiguities of combinatorial torsions. We shall use however the representation in terms of vector fields used by Burghelea--Haller~\cite{burghelea2006euler}. The goal of the present section is to introduce these Euler structures, in view of the definition of the Turaev torsion.

\subsection{The Chern-Simons class of a pair of vector fields}\label{subsec:chernsimons}
If $X \in \Cinf(M,TM)$ is a vector field with isolated non degenerate zeros, we define the singular $0$-chain
$$
\div(X) = -\sum_{x \in \crit(X)} \mathrm{ind}_X(x) [x]Ê~\in ~C_0(M, \mathbb{Z}),
$$
where $\crit(X)$ is the set of critical points of $X$ 
and $\ind_X(x)$ denotes the Poincar\'e-Hopf index of $x$ as a critical point of $X$~\footnote{$\mathrm{ind}_X(x)=(-1)^{\dim E_s(x)}$ if $x$ is hyperbolic and $E_s(x) \subset T_xM$ is the stable subspace of $x$.}. 
Note also that $\div\left(-X\right)=-\div(X)$ since $M$ is odd dimensional.

Let $X_0, X_1$ be two vector fields with isolated non degenerate zeros. Let $p :  M \times [0,1] \to M$ be
the projection over the first factor and choose a smooth section $H$ of the bundle $p^*TM \to M \times [0,1]$, transversal to the zero section, such that $H$ restricts to $X_i$ on $\{i\} \times M$ for $i=0,1$. Then the set $H^{-1}(0) \subset M \times [0,1]$ is an oriented smooth submanifold of dimension $1$ with boundary (it is oriented because $M$ and $[0,1]$ are), and we denote by $[H^{-1}(0)]$ its fundamental class.

\begin{defi} The class
$$
p_* [H^{-1}(0)] \in C_1(M, \mathbb{Z}) / \partial C_2(M,\mathbb{Z}),
$$
where $p_*$ is the pushforward by $p$, does not depend on the choice of the homotopy $H$ relating $X_0$ and $X_1$, cf. \cite[\S2.2]{burghelea2006euler}. 
This is the \emph{Chern-Simons class} of the pair $(X_0, X_1)$, denoted by $\cs(X_0, X_1)$.
\end{defi}
We have the fundamental formulae
\begin{equation}\label{eq:propcs}
\begin{split}
\partial \cs(X_0,X_1) = \div(X_1) - \div(X_0), \\
\cs(X_0, X_1) + \cs(X_1, X_2) = \cs(X_0, X_2),
\end{split}
\end{equation}
for any other vector field with non degenerate zeros $X_2$. Notice also that if $X_0$ and $X_1$ are nonsingular vector fields, then $\cs(X_0, X_1)$ defines a homology class in $H_1(M, \mathbb{Z})$.

\subsection{Euler structures.}\label{subsec:eulerstruct}
Let $X$ be a smooth vector field on $M$ with non degenerate zeros. An \textit{Euler chain} for $X$ is a singular one-chain $e \in C_1(M, \mathbb{Z})$ such that $\partial e = \div(X)$.  Euler chains for $X$ always exist because $M$ is odd-dimensional and thus $\chi(M) = 0$. 

Two pairs $(X_0, e_0)$ and $(X_1, e_1)$, with $X_i$ a vector field with non degenerate zeros and $e_i$ an Euler chain for $X_i$, $i=0,1$, will be said to be equivalent if
\begin{equation}\label{eq:equivalencerelation}
e_1 = e_0 + \cs(X_0, X_1) ~Ê\in ~C_1(M,\mathbb{C}) / \partial C_2(M, \mathbb{Z}).
\end{equation}
\begin{defi}
An \emph{Euler structure} is an equivalence class $[X,e]$ for the relation (\ref{eq:equivalencerelation}). We will denote by $\mathrm{Eul}(M)$ the set of Euler structures.
\end{defi}
There is a free and transitive action of $H_1(M, \mathbb{Z})$ on $\mathrm{Eul}(M)$ given by 
$$[X, e] + h =  [X, e + h], \quad h \in H_1(M, \mathbb{Z}).$$

\subsection{Homotopy formula relating flows}\label{subsec:homotopy}
Let $X_0, X_1$ be two vector fields with non degenerate zeros. Let $H$ be a smooth homotopy between $X_0$ and $X_1$ as in \S\ref{subsec:chernsimons} and set $X_t = H(t, \cdot) \in \Cinf(M,TM)$.
For $\varepsilon > 0$ we define $\Phi_\varepsilon : M\times[0,1] \to M \times M \times [0,1]$ via
$$
\Phi_\varepsilon(x,t) = \left(\e^{-\varepsilon X_t}(x),x,t\right), \quad x \in M, \quad t \in [0,1].
$$
Set also, with notations of \S\ref{subsec:integrationcurrents}, $H_\varepsilon = \mathrm{Gr}(\Phi_\varepsilon) \subset M \times M \times \mathbb{R}$. Then $H_\varepsilon$ is a submanifold with boundary of $M \times M \times \mathbb{R}$ which is oriented (since $M$ and $\mathbb{R}$ are). Define 
$$[H_\varepsilon] = (\Phi_\varepsilon)_* \left([M]Ê\times [[0,1]] \right)\in \mathcal{D}^{'n}(M \times M \times \mathbb{R})$$
to be the associated integration current, cf. \S\ref{subsec:integrationcurrents}. Let $g$ be any metric on $M$ and let $\rho > 0$ be smaller than its injectivity radius. Then for any $x,y \in M$ with $\mathrm{dist}(x,y) \leq \rho$, we denote
by $P(x,y) \in \mathrm{Hom}(E_x, E_y)$ the parallel transport by $\nabla$ along the minimizing geodesic joining $x$ to $y$. Then $P \in \Cinf(M\times M, \pi_1^*E^\vee \otimes \pi_2^*E)$ and we can define
$$
\mathcal{R}_\varepsilon = - \pi_*[H_\varepsilon] \otimes P \in \mathcal{D}^{'n-1}(M \times M, \pi_1^*E^\vee \otimes \pi_2^*E),
$$
where $\pi : M\times M \times \mathbb{R}Ê\to MÊ\times M$ is the projection over the two first factors. Note that $\mathcal{R}_\varepsilon$ is well defined if $\varepsilon$ is small enough so that
\begin{equation}\label{eq:distleqrho}
\mathrm{dist}\left(x, \e^{-sX_t}(x)\right) \leq \rho, \quad s \in [0, \varepsilon], \quad t \in [0,1], \quad x \in M,
\end{equation}
which implies $\mathrm{supp}~ \pi_*[H_\varepsilon] \subset \{(x,y), ~\mathrm{dist}(x,y) \leq \rho\}$. Now, let
$$R_\varepsilon : \Omega^\bullet(M,E) \to \mathcal{D}^{'\bullet-1}(M,E)$$
be the operator of degree $-1$ whose Schwartz kernel is $\mathcal{R}_\varepsilon$.

\begin{lemm}\label{lem:homotopy}
We have the following homotopy formula
\begin{equation}\label{eq:homotopy}
[\nabla, R_\varepsilon] = \nabla R_\varepsilon + R_\varepsilon \nabla = \e^{-\varepsilon \Lie_{X_1}^\nabla} -  \e^{-\varepsilon \Lie_{X_0}^\nabla}.
\end{equation}
\end{lemm}

\begin{proof}
First note that because $M$ is odd dimensional, the boundary (computed with orientations) of the manifold $H_\varepsilon$ is 
$$\partial H_\varepsilon = \mathrm{Gr}(\e^{-\varepsilon X_0})\times \{0\} - \mathrm{Gr}(\e^{-\varepsilon X_1})\times \{1\}.$$
Therefore we have, cf. (\ref{eq:boundary}), 
 $$(-1)^n\dd^{M\times M} \pi_* [H_\varepsilon] = \pi_* [\partial H_\varepsilon] = \left[\mathrm{Gr}(\e^{-\varepsilon X_0})\right] - \left[\mathrm{Gr}(\e^{-\varepsilon X_1})\right]$$
 where $\left[\mathrm{Gr}(\e^{-\varepsilon X_i})\right]$ denotes the integration current on the manifold $\mathrm{Gr}(\e^{-\varepsilon X_i})$ for $i=0,1$. Now note that we have by construction $\nabla^{E^\vee \boxtimes E} P = 0$. Therefore
$$
\nabla^{E^\vee \boxtimes E} \mathcal{R}_\varepsilon = (-1)^n\Bigl(\left[\mathrm{Gr}(\e^{-\varepsilon X_1})\right] - \left[\mathrm{Gr}(\e^{-\varepsilon X_0})\right]\Bigr) \otimes P.
 $$
Note that by definition of $\e^{-\Lie_{X_i}^\nabla}$ (cf \S\ref{subsec:policott}), the formula (\ref{eq:distleqrho}) and the flatness of $\nabla$ imply that the Schwartz kernel of $\e^{-\varepsilon \Lie_{X_i}^\nabla}$ is $\left[\mathrm{Gr}(\e^{-\varepsilon X_i})\right] \otimes P$. This concludes because the Schwartz kernel of $[\nabla, R_\varepsilon]$ is $(-1)^n\nabla^{E^\vee \boxtimes E} \mathcal{R}_\varepsilon$, cf. \cite[Lemma 2.2]{harvey2001finite}.
\end{proof}

The next formula follows from the definition of the flat trace and the Chern-Simons classes. It will be crucial for the topological interpretation of the variation formula obtained in \S\ref{sec:variationconnexion}.
\begin{lemm}\label{lem:homotopycs}
We have for any $\alpha \in \Omega^\bullet(M, \mathrm{End}(E))$ such that $\tr \alpha$ is closed and $\varepsilon > 0$ small enough
\begin{equation}\label{eq:homotopycs}
\tr^\flat_{\mathrm{s}} \alpha R_\varepsilon = \bigl\langle \tr\alpha, ~ \cs(X_0, X_1) \bigr\rangle.
\end{equation}
\end{lemm}
Note that because $H$ is transverse to the zero section, we have
\begin{equation}\label{eq:wfinterpolator}
\WF(\mathcal{R_\varepsilon}) \cap N^*\Delta= \emptyset,
\end{equation}
where $N^*\Delta$ denotes the conormal to the diagonal $\Delta$ in $M \times M$, so that the above flat trace is well defined.
\begin{proof}
We denote by
$i: M\hookrightarrow M\times M$ the diagonal inclusion. Note that the Schwartz kernel of $\alpha R_\varepsilon$ is $(-1)^n \pi_2^*\alpha \wedge \mathcal{R}_\varepsilon = - \pi_2^* \alpha \wedge \mathcal{R}_\varepsilon$ since $n$ is odd. From the definition of the
super flat trace $\strf$, we find that
\begin{equation}\label{eq:trsalphar}
\strf \alpha R_\varepsilon=\Bigl\langle\tr i^*\left(\pi_2^*\alpha\wedge \pi_*[H_\varepsilon]\otimes P \right), ~1\Bigr\rangle,
\end{equation}
where $\pi_2 : M\times M \to M$ is the projection over the second factor. Of course we have $i^* P = \id_E \in \Cinf(M, \mathrm{End}(E))$. We therefore have
$$
\tr i^*\left(\pi_2^*\alpha\wedge \pi_*[H_\varepsilon]\otimes P \right) = \tr \alpha \wedge i^*\pi_* [H_\varepsilon] =  \tr \alpha \wedge p_* j^* [H_\varepsilon]
$$
where $j : M \times [0,1] \hookrightarrow M \times M \times [0,1], ~ (x,t) \mapsto (x,x,t)$. This leads to
$$
\strf \alpha R_\varepsilon = \bigl\langle \tr \alpha \wedge p_* j^* [H_\varepsilon], 1 \bigr \rangle = \bigl \langle p^* \tr \alpha, j^*[H_\varepsilon]\bigr\rangle.
$$
Now if $\varepsilon$ is small enough, we can see that $j^*[H_\varepsilon] = [H^{-1}(0)]$. Therefore
$$
\strf \alpha R_\varepsilon = \bigl \langle \tr \alpha, p_* [H^{-1}(0)] \bigr \rangle = \bigl \langle \tr \alpha, \cs(X_0, X_1) \bigr \rangle.
$$
\end{proof}

\section{Morse theory and variation of Turaev torsion.}\label{sec:turaevvariation}

We introduce here the Turaev torsion which is defined in terms of CW decompositions.
In the spirit of the seminal work of Bismut--Zhang~\cite{bismut1992extension} based on 
geometric constructions of Laudenbach~\cite{laudenbach1992thom}, 
we use a CW decomposition which comes from the unstable cells of a 
Morse-Smale gradient flow induced by a Morse function. This allows us to interpret the variation of the Turaev torsion as a supertrace on the space of generalized resonant states for the Morse-Smale flow. 
This interpretation will be convenient for the comparison of the Turaev torsion 
with the dynamical torsion.

\subsection{Morse theory and \textrm{CW}-decompositions}\label{subsec:morsesmale}

Let $f$ be a Morse function on $M$ and $\widetilde{X} = -\grad_g f$ be its associated gradient vector field with respect to some Riemannian metric $g$ (the tilde notation is used to make the difference with the Anosov flows we studied until now). For any $a \in \crit(f)$, 
we denote by
$$W^s(a) = \left\{y \in M, ~\lim_{t\to \infty} \e^{t\widetilde{X}}y = a \right\}, \quad W^u(a) = \left\{y \in M, ~\lim_{t\to \infty} \e^{-t\widetilde{X}}y = a \right\},$$
the stable and unstable manifolds of $a$. Then it is well known that $W^s(a)$ (resp. $W^u(x)$) is a smooth embedded open disk of dimension $n-\ind_f(a)$ (resp. $\ind_f(a)$), where $\ind_f(a)$ is the index of $a$ as a critical point of $f$, that is, in a Morse chart $(z_1, \dots, z_n)$ near $a$,
$$
f(z_1, \dots, z_n) = f(a) - z_1^2 - \dots - z_{\ind_f(a)}^2 + z_{\ind_f(a)+1}^2 + \dots + z_n^2.
$$
For simplicity, we will denote 
$$|a| = \ind_f(a) = \dim W^u(a),$$
and we fix an orientation of every $W^u(a).$

We assume that $\widetilde{X}$ satisfies the Morse-Smale condition, that is, for any $a,b \in \crit(f)$, the manifolds $W^s(a)$ and $W^u(b)$ are transverse. Also, we assume that for every $a \in \crit(f)$, the metric $g$ is flat near $a$. Let us summarize some results 
from \cite[Theorems 3.2,3.8,3.9]{qin2010moduli} ensured by the unstable manifolds of $f$. We would like to mention that such results 
can be found in slightly different form in the work
of Laudenbach~\cite{laudenbach1992thom} and are used in~\cite{bismut1992extension}~\footnote{A difference is that Laudenbach only needs to compactify the unstable cells as $C^1$--manifolds with conical singularities whereas Qin proves smooth compactification as manifolds with corners.}.

First, $W^u(a)$ admits a compactification to a smooth $|a|$-dimensional manifold with corner $\overline{W}^u(a)$, endowed with a smooth map $e_a : \overline{W}^u(a) \to M$ that extends the inclusion $W^u(a) \subset M$. Then the collection $W = \left\{\overline{W}^u(a)\right\}_{a \in \crit(f)}$ and the applications $e_a$ induce a CW-decomposition on $M$. Moreover, the boundary operator of the cellular chain complex is given by
$$
\partial\overline{W}^u(a)  = \sum_{ |b| = |a| - 1} \# \Lie(a,b) Ê\overline{W}^u(b),
$$ 
where $\Lie(a,b)$ is the moduli space of gradient lines joining $a$ to $b$ and $\#\Lie(a,b)$ is the sum of the orientations induced by the  orientations of the unstable manifolds of $(a,b)$, see \cite[Theorem 3.9]{qin2010moduli}.

\subsection{The Thom-Smale complex}\label{subsec:combinatorialcomplex}
We set
$
C_\bullet(W, E^\vee) = \bigoplus_{k=0}^n C_k(W, E^\vee)$
where
$$
C_k(W, E^\vee) = \bigoplus_{\substack{a \in \crit(f) \\ |a| = k}} E_a^\vee, \quad k = 0, \dots, n.
$$
We endow the complex $C_\bullet(W, E^\vee)$ with the boundary operator $\partial^{\nabla^\vee}$ defined by
$$
\partial^{\nabla^\vee} u = \sum_{|b| = |a| - 1} \sum_{\gamma \in \Lie(a,b)} \varepsilon_\gamma P_{\gamma}(u), \quad a \in \crit(f), \quad u \in E_a^\vee,
$$
where for $\gamma \in \Lie(a,b)$, $P_{\gamma} \in \mathrm{End}(E_a^\vee, E_b^\vee)$ is the parallel transport of $\nabla^\vee$ along the curve $\gamma$ and $\varepsilon_\gamma = \pm 1$ is the orientation number of $\gamma \in \Lie(a,b)$.

Then by \cite{laudenbach1992thom} (see also \cite{dang2017topology} for a different approach),
there is a canonical isomorphism
$$
H_\bullet(M, \nabla^\vee) \simeq H_\bullet(W, \nabla^\vee),
$$
where $H_\bullet(M,\nabla^\vee)$ is the singular homology of flat sections of $(E^\vee, \nabla^\vee)$ and $H_\bullet(W, \nabla^\vee)$ denotes the homology of the complex $C_\bullet(W, E^\vee)$ endowed with the boundary map $\partial^{\nabla^\vee}$. Therefore this complex is acyclic since $\nabla$ (and thus $\nabla^\vee$) is.

\subsection{The Turaev torsion}\label{subsec:turaevtorsion}
Fix some base point $x_\star \in M$ and for every $a \in \crit(f)$, let $\gamma_a$ be some path in $M$ joining $x_\star$ to $a$. Define
\begin{equation}\label{eq:spider}
e = \sum_{a\in \crit(f)} (-1)^{|a|}\gamma_a \in C_1(M, \mathbb{Z}).
\end{equation}
Note that the Poincar\'e-Hopf index of $\widetilde{X}$ near $a \in \crit(f)$ is $-(-1)^{|a|}$ so that
\begin{equation}\label{eq:partiale}
\partial e = \div(\widetilde{X})
\end{equation}
because $\sum_{a\in \crit(f)} (-1)^{|a|} = \chi(M) = 0$ by the Poincar\'e-Hopf index theorem. Therefore $e$ is an Euler chain for $\widetilde{X}$ and 
$$\frak{e} = [\widetilde{X}, e]$$
defines an Euler structure. Choose some basis $u_1, \dots, u_d$ of $E_{x_\star}^\vee$. For each $a \in \crit(f)$, we propagate this basis via the parallel transport of $\nabla$ along $\gamma_a$ to obtain a basis $u_{1,a}, \dots, u_{d,a}$ of $E_{a}$. We choose an ordering of the cells $\left\{\overline{W}^u(a)\right\}$; this gives us a cohomology orientation $\frak{o}$ (see \cite[\S6.3]{turaev1990euler}). Moreover this ordering and the chosen basis of $E_a^\vee$ give us (using the wedge product) an element $c_k \in \det C_k(W, E^\vee)$ for each $k$, and thus an element $c \in \det C_\bullet(W, E^\vee)$. 

The \textit{Turaev torsion} of $\nabla$ with respect to the choices $\frak{e,o}$ is then defined by \cite[\S9.2 p.~218]{farber2000poincare}
$$
\tau_{\frak{e,o}}(\nabla)^{-1} = \varphi_{C_\bullet(W,\nabla^\vee)}(c) \in \mathbb{C} \setminus 0,
$$
where $\varphi_{C_\bullet(W,\nabla^\vee)}$ is the homology version of the isomorphism (\ref{eq:isocoh}). Note that $\nabla^\vee$ (and not $\nabla$) is involved in the definition of $\tau_\frak{e,o}(\nabla)$. Indeed, we use here the cohomological version of Turaev's torsion, which is more convenient for our purposes, and which is consistent with \cite{braverman2007ray}, \cite[p.~252]{braverman2008refined}.

\subsection{Resonant states of the Morse-Smale flow}
In \cite{dang2017topology}, it has been shown that we can define 
Ruelle resonances for the Morse-Smale gradient flow $\Lientilde$ as described in \S\ref{sec:policott} in the context of Anosov flows. More precisely, we have that the resolvent
$$
\left(\Lientilde +s\right)^{-1} : ~\Omega^\bullet(M,E) \to \mathcal{D}^{'\bullet}(M,E),
$$
is
well defined for $\Re(s) \gg 0$, has a meromorphic continuation to all $s \in \mathbb{C}$. The poles of this continuation are the Ruelle resonances of $\Lientilde$ and the set of those will be denoted by $\Res(\Lientilde)$. In fact, the set $\Res(\Lientilde)$ does not depend on the flat vector bundle $(E, \nabla)$. 
Let $\lambda > 0$ be such that $\Res(\Lientilde) \cap \{|s|Ê\leq \lambda\} \subset \{0\}$; set 
\begin{equation}\label{eq:projectormorsesmale}
\widetilde{\Pi} = \frac{1}{2 \pi i} \int_{|s| = \lambda} \left(\Lientilde + s\right)^{-1} \dd s
\end{equation}
the spectral projector associated to the resonance $0$, and denote by
$$
\widetilde{C}^\bullet = \mathrm{ran}~ \widetilde{\Pi} \subset \mathcal{D}^{'\bullet}(M,E)
$$
the associated space of generalized eigenvectors for $\Lientilde$.  Since $\nabla$ and $\Lientilde$ commute, $\nabla$ induces a differential on the complex $\widetilde{C}^\bullet$. Moreover, $\widetilde{\Pi}$ maps $\mathcal{D}_\Gamma^{'\bullet}(M,E)$ to itself continuously where
$$
\Gamma = \bigcup_{a \in \crit(f)} \overline{{N}^*W^u(a)} \subset T^*M.
$$

\subsection{A variation formula for the Turaev torsion}
Assume that we are given a $\mathcal{C}^1$ family of acyclic connections $\nabla(z)$ on $E$ as in \S\ref{sec:variationconnexion}. We denote by $\widetilde{\Pi}_-(z)$ the spectral projector (\ref{eq:projectormorsesmale}) associated to $\nabla(z)$ and $-\widetilde X$, and set $\widetilde{C}^\bullet_-(z) = \mathrm{ran}~\widetilde{\Pi}_-(z)$. By \cite{dang2017topology} we have that all the complexes $(\widetilde{C}^\bullet(z), \nabla(z))$ are acyclic and there exists cochain contractions $\tilde{k}_-(z) : \widetilde{C}^\bullet_-(z) \to \widetilde{C}_-^{\bullet-1}(z)$. As in \S\ref{subsec:variationformula} we have a variation formula for the Turaev torsion.

\begin{prop}\label{prop:variationturaev}
The map $z \mapsto \tilde{\tau}(z) = \tau_{\frak{e,o}}(\nabla(z))$ is real differentiable on $U$ and for any $z \in U$
$$
\dd (\log \tilde \tau)_z \sigma = -\str{\widetilde{C}^\bullet(z)}\left( \widetilde{\Pi}_-(z)Ê\alpha_z(\sigma) \tilde{k}_-(z)\right) - \int_e \tr \alpha_z(\sigma), \quad \sigma \in \mathbb{C}
$$
where $\alpha_z(\sigma)$ is given by (\ref{eq:alpha_z}) and $e$ is given by (\ref{eq:spider}).
\end{prop}
The rest of this section is devoted to the proof of Proposition \ref{prop:variationturaev}. For convenience, we will first study the variation of $z \mapsto \tau_{\frak{e,o}}(\nabla(z)^\vee)$.

\subsection{A preferred basis}
Let $a \in \crit(f)$ and $k = |a|$. We denote by $[W^u(a)] \in \mathcal{D}^{'n-k}_\Gamma(M)$ the integration current over the
unstable manifold $W^u(a)$ of $\tilde{X}$, it is a well defined current far from $\partial W^u(a)$. We also pick a cut-off function $\chi_a \in \Cinf(M)$ valued in $[0,1]$ with $\chi_a \equiv 1$ near $a$ and $\chi_a$ is supported in a small neighborhood $\Omega_a$ of $a$, with $\overline{\Omega_a} \cap \partial W^u(a) = \emptyset.$ Recall from \ref{subsec:turaevtorsion} that we have a basis $u_{1,a}, \dots, u_{d,a}$ of $E_a$. Using the parallel transport of $\nabla$, we obtain flat sections of $E$ over $W^u(a)$
that we will still denote by $u_{1,a}, \dots, u_{d,a}.$ Define
\begin{equation}\label{eq:utilde}
\tilde{u}_{j,a} = \widetilde{\Pi} \Bigl( \chi_a [W^u(a)] \otimes u_{j,a}\Bigr) \in \widetilde{C}^{n-k}, \quad j=1, \dots, d.
\end{equation}
By \cite{dang2017spectral} we have that $\bigl\{\tilde{u}_{j,a}, ~a \in \crit(f), ~1\leq j \leq d\bigr\}$ is a basis of $\widetilde{C}^\bullet$. Adapting the proof of \cite[Theorem 2.6]{dangwitten} to the bundle case, we obtain the following proposition which will allow us to compute the Turaev torsion with the help of the complex $\widetilde{C}^\bullet$.
\begin{prop}
The map $\Phi : C_\bullet(W,\nabla) \to \widetilde{C}^{n-\bullet}$ defined by
$$
\Phi \bigl( u_{j,a}\bigr) = \tilde{u}_{j,a}, \quad a \in \crit(f), \quad j=1, \dots, d,
$$
is an isomorphism and satisfies~\footnote{$(-1)^\bullet$ comes from $\partial=(-1)^{\deg +1}\dd$ comparing the boundary $\partial$ and De Rham differential $\dd$}
$$
\Phi \circ \partial^\nabla = (-1)^{\bullet}\nabla \circ \Phi.
$$
\end{prop}

An immediate corollary is that (using the notation of \S\ref{subsec:torsionfinitedim})
\begin{equation}\label{eq:turaevmorse}
\tau_{\frak{e,o}}(\nabla^\vee) = \varphi_{C_\bullet(W,\nabla)}(u)^{-1} = \tau(\widetilde{C}^\bullet, \tilde{u}),
\end{equation}
where $u \in \det C_\bullet(W,\nabla)$ (resp. $\tilde{u}Ê\in \det \widetilde{C}^\bullet$) is the element given by the basis $\{u_{j,a}\}$ (resp. $\{\tilde{u}_{j,a}\})$ and the ordering of the cells $W^u(a)$.

\subsection{Proof of Proposition \ref{prop:variationturaev}}
For any $a \in \crit(f)$ we denote by $P_{\gamma_a}(z) \in \mathrm{Hom}(E_{x_\star}, E_a)$ the parallel transport of $\nabla(z)$ along $\gamma_a$. We set
$$
u_{j,a}(z) = P_{\gamma_a}(z) P_{\gamma_a}(z_0)^{-1} u_{j,a}
$$
and
$$
\tilde{u}_{j,a}(z) = \widetilde{\Pi}(z)\Bigl(\chi_a [W^u(a)] \otimes u_{j,a}(z) \Bigr),
$$
where again we consider $u_{j,a}(z)$ as a $\nabla(z)$-flat section of $E$ over $W^u(a)$ using the parallel transport of $\nabla(z)$.
The construction of Ruelle resonances for Morse-Smale gradient flow follows from the construction of anisotropic Sobolev spaces 
$$\Omega^\bullet(M,E) \subset \widetilde{\H}^\bullet_{1} \subset \widetilde{\H}^\bullet \subset \mathcal{D}^{'\bullet}(M,E),$$
see \cite{dang2016spectral}, on which $\Lie_{\widetilde{X}}^\nabla + s$ is a holomorphic family of Fredholm operators of index $0$ in the region $\{\Re(s) > -2\}$, and such that $\nabla(z)$ is bounded 
$\widetilde{\H}^\bullet_{1} \to \widetilde{\H}^\bullet$. 
Every argument made in \S\ref{subsec:anisotropicuniformconnexion} also stand here and $z\mapsto \widetilde{\Pi}(z)$ is a $\mathcal{C}^1$ family of bounded operators $\widetilde{\H}^\bullet \to \widetilde{\H}^\bullet_{1}$.

Note that by continuity, $\widetilde{\Pi}(z)$ induces an isomorphism $\widetilde{C}^\bullet(z_0) \to \widetilde{C}^\bullet(z)$ for $z$ close enough to zero. Let $\tilde{u}(z) \in \det \widetilde{C}^\bullet(z)$ be the element given by the basis $\{\tilde{u}_{j,a}(z)\}$ and the ordering of the cells $W ^u(a)$. Then by (\ref{eq:turaevmorse}) and (\ref{eq:changeofbasis}) we have
\begin{equation}\label{eq:turaevmorse2}
\tau_{\frak{e,o}}(\nabla(z)^\vee)= \tau\left(\widetilde{C}^\bullet(z), \tilde{u}(z)\right) = \bigl[\tilde{u}(z) : \widetilde{\Pi}(z) \tilde{u}(z_0)Ê\bigr]Ê\tau\left(\widetilde{C}^\bullet(z), \widetilde{\Pi}(z)\tilde{u}(z_0)\right),
\end{equation}
where $\widetilde{\Pi}(z) \tilde{u}(z_0) \in \det \widetilde{C}^\bullet(z)$ is the image of $\tilde{u}$ by the isomorphism $\det \widetilde{C}^\bullet(z_0) \to \det \widetilde{C}^\bullet(z)$ induced by $\widetilde{\Pi}(z)$, and $\tilde{u}(z) = \bigl[\tilde{u}(z) : \widetilde{\Pi}(z) \tilde{u}(z_0)Ê\bigr] \widetilde{\Pi}(z) \tilde{u}(z_0).$ Doing exactly as in \S\ref{subsec:variationconnexionfinitedimensional}, we obtain that $z \mapsto \hat{\tau}(z) = \tau\left(\widetilde{C}^\bullet(z), \widetilde{\Pi}(z)\tilde{u}\right) $ is $\mathcal{C}^1$ and
\begin{equation}\label{eq:variationconnexionmorse}
\dd (\log \hat{\tau})_{z_0} \sigma = - \str{\widetilde{C}^\bullet} \widetilde{\Pi}(z_0)Ê\alpha_{z_0}(\sigma) \tilde{k}(z_0).
\end{equation}
Therefore it remains to compute the variation of $\bigl[\tilde{u}(z) : \widetilde{\Pi}(z) \tilde{u}(z_0)Ê\bigr].$ This is the purpose of the next formula.
\begin{lemm}\label{lem:changebasis}
We have
$$
\bigl[\tilde{u}(z) : \widetilde{\Pi}(z) \tilde{u}(z_0)Ê\bigr] = \prod_{a \in \crit(f)} \det \Bigl(P_{\gamma_a}(z) P_{\gamma_a}(z_0)^{-1}\Bigr)^{(-1)^{n-|a|}}.
$$
\end{lemm}
\begin{proof}
By definition of the basis $\{u_{a,j}\}$ in \S\ref{subsec:turaevtorsion} it suffices to show that for $z$ small enough
\begin{equation}\label{eq:identitybasis}
\widetilde{\Pi}(z) \tilde{u}_{a,i} = \sum_{j=1}^d A_{a,i}^j(z) \tilde{u}_{a,j}(z), \quad aÊ\in \crit(f), \quad 1\leq i,j \leq d, 
\end{equation}
where the coefficients $A_{a,i}^j(z)$ are defined by  $\displaystyle{u_{a,i}(z_0)(a) =  \sum_{j=1}^d A_{a,i}^j(z) u_{a,j}}(z)(a)$.

Consider the dual operator $\Lie_{-\widetilde{X}}^{\nabla(z)^{\vee}} : \Omega^\bullet(M,E^\vee) \to \Omega^\bullet(M,E^\vee).$ The above constructions, starting from a dual basis $s_1, \dots, s_d \in E_{x_\star}^\vee$ of $u_1, \dots, u_d$, give a basis $\{{s}_{a,i}(z)\}$ of each $\Gamma(W^s(a), \nabla(z)^\vee)$ (the space of flat section of $\nabla(z)^\vee$ over $W^s(a)$), since the unstable manifolds of $-\widetilde{X}$ are the stable ones of $\widetilde{X}$. Let $\widetilde{C}^\bullet_\vee(z)$ be the range of the spectral projector $\widetilde{\Pi}^\vee(z)$ from (\ref{eq:projectormorsesmale}) associated to the vector field $-\widetilde{X}$ and the connection $\nabla(z)^\vee$. We have a basis $\{\tilde{s}_{a,i}(z)\}$ of $\widetilde{C}^\bullet_\vee(z)$ given by
$$
\tilde{s}_{a,i}(z) = \widetilde{\Pi}^\vee(z)\Bigl(\chi_a [W^s(a)] \otimes s_{a,i}(z)\Bigr).
$$
We will prove that for any $a,b \in \crit(f)$ with same Morse index we have for any $1\leq i,j \leq d$,
\begin{equation}\label{eq:identity}
\bigr\langle  \tilde{s}_{a,j}(z), ~\tilde{u}_{a,i}(z_0)\bigr\rangle = \left\{ \begin{matrix} \bigl\langle s_{a,j}(z)(a), u_{a,i}(z_0)(a) \bigr\rangle_{E_a^\vee, E_a} &\text{ if }a=b, \\
0 &\text{ if }a \neq b\end{matrix}\right..
\end{equation}

First assume that $a \neq b$. Then $W^u(a) \cap W^s(b) = \emptyset$ by the transversality condition, since $a$ and $b$ have same Morse index. Therefore for any $t_1, t_2 \geq 0$, we have
\begin{equation}\label{eq:bigpairing}
\biggl\langle \e^{-t_1 \Lie_{-\widetilde{X}}^{\nabla(z)^\vee}} \Bigl(\chi_b [W^s(b)] \otimes s_{b,j}(z)\Bigr), ~ \e^{-t_2 \Lie_{\widetilde{X}}^{\nabla(z_0)}} \Bigl(\chi_a [W^u(a)] \otimes u_{a,i}(z)\Bigr) \biggr\rangle = 0,
\end{equation}
since the currents in the pairing have disjoint support because they are respectively contained in $W^s(b)$ and $W^u(a)$. Now notice that for $\Re(s)$ big enough, one has
$$
\left(\Lie_{-\widetilde{X}}^{\nabla(z)^\vee} + s\right)^{-1} = \int_0^\infty \e^{-t \Lie_{-\widetilde{X}}^{\nabla(z)^\vee}} \e^{-ts} \dd t \quad \text{ and } \quad \left(\Lie_{\widetilde{X}}^{\nabla(z_0)} + s\right)^{-1} = \int_0^\infty \e^{-t \Lie_{\widetilde{X}}^{\nabla(z_0)}} \e^{-ts} \dd t.
$$
Therefore the representation (\ref{eq:projectormorsesmale}) of the spectral projectors and the analytic continuation of the above resolvents imply with (\ref{eq:bigpairing}) that $\bigr\langle  \tilde{s}_{b,j}(z), ~\tilde{u}_{a,i}\bigr\rangle = 0$.

Next assume that $a=b$. Then $W^u(a) \cap W^s(a) = \{a\}$. Since the support of $\tilde{s}_{a,i}(z)$ (resp. $\tilde{u}_{a,i}(z_0)$) is contained in the closure of $W^s(a)$ (resp. $W^u(a)$), we can compute
$$
\begin{aligned}
\biggl\langle \widetilde{\Pi}^\vee(z) \Bigl(\chi_a [W^s(a)] \otimes s_{a,j}(z)\Bigr),~ \widetilde{\Pi}\Bigl(&\chi_a [W^u(a)] \otimes u_{a,i}(z_0)\Bigr) \biggr\rangle \\
&=  \Bigl\langle \chi_a [W^s(a)] \otimes s_{a,j}(z), ~\chi_a [W^u(a)] \otimes u_{a,i}(z_0)\Bigr\rangle \\
&= \Bigl\langle [a], \langle s_{a,j}(z), u_{a,i}(z_0)Ê\rangle_{E^\vee, E}\Bigr \rangle,
\end{aligned}
$$
where the first equality stands because $\tilde{s}_a(z) = [W^s(a)] \otimes s_{a,j}(z)$ near $a$ by \cite[Proposition 7.1]{dang2017spectral}. This gives (\ref{eq:identity}).

This identity immediately yields (\ref{eq:identitybasis}) with $A_{a,i}^j(z) = \bigl\langle s_{a,j}(z)(a), u_{a,i}(z_0)(a) \bigr\rangle_{E_a^\vee, E_a}$ since we have
\begin{equation}\label{eq:pimorse}
\widetilde{\Pi}(z) = \sum_{a,i} \bigl\langle \tilde{s}_{a,j}(z), \cdot \bigr\rangle \tilde{u}_{a,j}(z)
\end{equation}
\end{proof}
Using the lemma, we obtain, if $\mu(z) = \bigl[\tilde{u}(z) : \widetilde{\Pi}(z) \tilde{u}(z_0)Ê\bigr]$,
$$
\dd (\log \mu)_{z_0} \sigma = \sum_{a \in \crit(f)} (-1)^{n-|a|} \tr \Bigl(A_{\gamma_a}(z_0, \sigma) P_{\gamma_a}(z_0)^{-1}\Bigr)
$$
where $A_{\gamma_a}(z_0, \sigma) = \dd \left(P_{\gamma_a}\right)_{z_0}\sigma$. Since $n$ is odd, 
we obtain by definition of $e$ and (\ref{lem:holonomy})
$$
\dd (\log \mu)_{z_0}\sigma = \sum_{a \in \crit(f)} (-1)^{|a|} \int_{\gamma_a}  \tr \alpha_{z_0}(\sigma) = \int_e \tr \alpha_{z_0}(\sigma).
$$
This equation combined with (\ref{eq:turaevmorse2}) and (\ref{eq:variationconnexionmorse}) yields, if $\tilde{\tau}^\vee(z) = \tau_{\frak{e,o}}(\nabla(z)^\vee)$
$$
\dd (\log \tilde\tau^\vee)_{z_0}\sigma = -\str{\widetilde{C}^\bullet}{\widetilde{\Pi}(z_0) \alpha_{z_0}(\sigma) \tilde{k}(z_0)} + \int_e \tr \alpha_{z_0}(\sigma).
$$
The proof is almost finished. But since 
we need to formulate our results in terms of the cohomological torsion, we still have to make some tedious formal manipulations
to pass to the cohomological formalism.
The first step is to replace $\nabla$ by the dual connection $\nabla^\vee$ in the above formula.
We also introduce some notation. The operator $\widetilde{\Pi}$ was the spectral projector
on the kernel of $\mathcal{L}_{\widetilde{X}}^\nabla$. Now we need to work with the spectral projector on
$\ker\left(\mathcal{L}_{\widetilde{X}}^{\nabla(z_0)^\vee}\right)$ (resp. $\mathcal{L}_{-\widetilde{X}}^{\nabla(z_0)} $), which we denote by
$\widetilde{\Pi}^\vee_+(z_0)$ (resp $\widetilde{\Pi}_-(z_0)$) where the $+$ (resp $-$) sign emphasizes the fact that we deal with $+\widetilde{X}$ (resp $-\widetilde{X}$).
Now note that 
$$
\nabla(z)^\vee = \nabla(z_0)^\vee - {}^T\bigl(\alpha_{z_0}(z-z_0)\bigr) + o(z-z_0).
$$
Therefore, applying what precedes to $\tilde \tau(z)$ we get
\begin{equation}\label{eq:variationfinal}
\dd (\log \tilde \tau)_{z_0} \sigma = -\str{\widetilde{C}^\bullet_{\vee,+}} \left({\widetilde{\Pi}^\vee_{+}(z_0) \bigl(-{}^T\alpha_{z_0}(\sigma)\bigr) \tilde{k}^\vee_+(z_0)}\right) + \int_e \tr \left(-{}^T\alpha_{z_0}(\sigma)\right),
\end{equation}
where $\widetilde{\Pi}^\vee_+(z_0)$ is the spectral projector (\ref{eq:projectormorsesmale}) associated to $\nabla(z_0)^\vee$ and $+ \widetilde{X}$, $\widetilde{C}^\bullet_{\vee,+} = \mathrm{ran} ~\widetilde{\Pi}^\vee_+(z_0)$, and $\widetilde{k}^\vee_+(z_0)$ is any cochain contraction on the complex $(\widetilde{C}^\bullet_{\vee,+}, \nabla(z_0)^\vee)$. Now, we have the identification 
$$
\left(\widetilde{C}^k_{\vee,+}\right)^\vee \simeq \widetilde{C}^{n-k}_{-},
$$
where $\widetilde{C}^{\bullet}_{-}$ is the range of $\widetilde{\Pi}_-(z_0)$, the spectral projector (\ref{eq:projectormorsesmale}) associated to $\nabla(z_0)$ and $-\widetilde{X}$. It is easy to show that under this identification, one has
$$
\left(\widetilde{\Pi}^\vee_+ \bigl({}^T\alpha_{z_0}(\sigma)\bigr) \tilde{k}(z_0)\right)^\vee = \widetilde\Pi_-(z_0) \alpha_{z_0}(\sigma) k_-(z_0) + \left[\widetilde\Pi_-(z_0) \alpha_{z_0}(\sigma), k_-(z_0)\right],
$$
where for any $j \in \{0, \dots, n\}$, we set
$$
k_-(z_0)|_{\widetilde{C}_-^{n-j}} = (-1)^{j+1}\bigl(\tilde{k}^\vee_+(z_0)|_{\tilde C^{j+1}}\bigr)^\vee : \widetilde{C}_-^{n-j} \to \widetilde{C}_-^{n-j-1}.
$$
Then $k_-(z_0)$ is a cochain contraction on the complex $(\widetilde C_-^\bullet, \nabla(z_0))$.
As a consequence, since $n$ is odd,
$$
\str{\widetilde{C}^\bullet_{\vee,+}} \left({\widetilde{\Pi}^\vee_{+}(z_0) \bigl(-{}^T\alpha_{z_0}(\sigma)\bigr) \tilde{k}^\vee_+(z_0)}\right) = \str{\widetilde{C}^\bullet_-} \widetilde\Pi_-(z_0) \alpha_{z_0}(\sigma) k_-(z_0).
$$
This concludes by (\ref{eq:variationfinal}) since $\tr(-{}^T\beta) = -\tr \beta$ for any $\beta \in \Omega^1(M,\mathrm{End}(E)).$

\section{Comparison of the dynamical torsion with the Turaev torsion}\label{sec:comparison}
In this section we see the dynamical torsion and the Turaev torsion as functions on the space of acyclic representations. This is an open subset of a complex affine algebraic variety. Therefore we can compute the derivative of  $\tau_\vartheta/\tau_{\frak{e,o}}$ along holomorphic curves, using the variation formulae obtained in \S\S\ref{sec:variationconnexion},\ref{sec:turaevvariation}. From this computation we will deduce Theorem \ref{thm:main2}.

\subsection{The algebraic structure of the representation variety}
We describe here the analytic structure of the space
$$
\Rep(M,d) = \mathrm{Hom}(\pi_1(M), \mathrm{GL}(\mathbb{C}^d))
$$
of complex representations of degree $d$ of the fundamental group. Since $M$ is compact, $\pi_1(M)$ is generated by a finite number of elements $c_1, \dots, c_L \in \pi_1(M)$ which satisfy finitely many relations. A representation $\rho \in \Rep(M,d)$ is thus given by $2L$ invertible $d \times d$ matrices $\rho(c_1), \dots, \rho(c_L)$, $ \rho(c_1^{-1}), \dots \rho(c_L^{-1})$ with complex coefficients satisfying finitely many polynomial equations. Therefore the set $\Rep(M,d)$ has a natural structure of a complex affine algebraic set. We will denote the set of its singular points by $\Sigma(M,d).$ In what follows, we will only consider the classical topology of $\Rep(M,d)$.

We will say that a representation $\rho \in \Rep(M,d)$ is acyclic if $\nabla_\rho$ is acyclic. We denote by $\Rep_{\mathrm{ac}}(M,d) \subset \Rep(M,d)$ the space of acyclic representations. This is an open set (in the Zariski topology, thus in the classical one) in $\Rep(M,d)$, see \cite[\S4.1]{burghelea2006euler}. For any $\rho \in \Rep_{\mathrm{ac}}(M,d)$ we set
$$
\tau_\vartheta(\rho) = \tau_\vartheta(\nabla_\rho), \quad \tau_{\frak{e,o}}(\rho) = \tau_{\frak{e,o}}(\nabla_\rho),
$$
for any Euler structure $\frak{e}$ and any cohomological orientation $\frak{o}$.

\subsection{Holomorphic families of acyclic representations}
\label{ss:holoconnections}
Let $\rho_0 \in \Rep_{\mathrm{ac}}(M,d) \setminus \Sigma(M,d)$ be a regular point. Take $\delta > 0$ and $\rho(z), ~|z| < \delta,$ a holomorphic curve in $\Rep_{\mathrm{ac}}(M,d) \setminus \Sigma(M,d)$ such that $\rho(0) = \rho_0$. Theorems \ref{thm:main2} and \ref{thm:main3} will be a consequence of the following

\begin{prop}\label{prop:comparison}
Let $X$ be a contact Anosov vector field on $M$. Let $\frak{e} = [\widetilde{X}, e]$ be the Euler structure defined in \S\ref{subsec:turaevtorsion}. Note that $-\cs(-\widetilde{X}, X) + e$ is a cycle and defines a homology class $h \in H_1(M, \mathbb{Z})$. Then $z\mapsto \tau_\vartheta(\rho(z))/\tau_{\frak{e,o}}(\rho(z))$ is complex differentiable and
$$
\frac{\dd}{\dd z} \left(\frac{\tau_\vartheta(\rho(z))}{\tau_{\frak{e,o}}(\rho(z))}{\bigl\langle \det \rho(z),h \bigr\rangle}\right)= 0
$$
for any cohomological orientation $\frak{o}$.
\end{prop}
Proposition \ref{prop:comparison} relies on the variation formulae given by Propositions \ref{prop:variationconnexion} and \ref{prop:variationturaev}, and Lemma \ref{lem:homotopycs} which gives a topological interpretation of those. 

\subsection{An adapted family of connections}
Following \cite[\S4.1]{braverman2017refined}, there exists a flat vector bundle $E$ over $M$ and a $\mathcal{C}^1$ family of connections $\nabla(z), ~|z|Ê< \delta,$  in the sense of  \S\ref{subsec:familyconnexion}, such that \footnote{It is actually stated in \cite[\S4.1]{braverman2017refined} that one can find a $\mathcal{C}^1$ family of connections satisfying (\ref{eq:rhonabla}); however looking carefully at the proofs one can choose the family $\nabla(z)$ to be $\mathcal{C}^1$ in $z$.} 
\begin{equation}\label{eq:rhonabla}
\rho_{\nabla(z)} = \rho(z)
\end{equation}
for every $z$; we can moreover ask  the family $\nabla(z)$ to be complex differentiable at $z=0$, that is, 
\begin{equation}\label{eq:nabla(z)}
\nabla(z) = \nabla + z \alpha + o(z),
\end{equation}
where $\nabla = \nabla(0)$ and $\alpha \in \Omega^1(M,\mathrm{End}(E))$. Note that flatness of $\nabla(z)$ 
implies
$$
[\nabla, \alpha] = \nabla \alpha + \alpha \nabla = 0.
$$

\subsection{A cochain contraction induced by the Morse-Smale gradient flow}
Let
$$
\left(\Lientildem + s\right)^{-1} = \frac{\widetilde{\Pi}_-}{s} + \widetilde{Y} + \dom(s)
$$
be the Laurent expansion of $\left(\mathcal{L}^{\nabla}_{-\widetilde{X}} + s\right)^{-1}$ near $s=0$. The fact that $s=0$ is a simple pole comes from \cite{dang2016spectral}. As in \ref{subsec:cochain}, we consider the operator
$$
\widetilde{K} = \iota_{-\widetilde{X}}\widetilde{Y}(\id - \widetilde{\Pi}_-) + \tilde{k}_-Ê\widetilde{\Pi}_- ~:~ \Omega^\bullet(M,E) \to \mathcal{D}^{'\bullet}(M,E),
$$
where $\tilde{k}_-$ is any cochain contraction on $\widetilde{C}^\bullet_- = \mathrm{ran}~\widetilde{\Pi}_-$. Note that we have the identity
\begin{equation}\label{eq:cochainmorse}
[\nabla, \widetilde{K}] = \nabla \widetilde{K} + \widetilde{K}Ê\nabla = \id.
\end{equation}
The next proposition will allow us to interpret the term $\str{\widetilde{C}^\bullet} \widetilde{\Pi}_-(z)Ê\alpha_z(\sigma) \tilde{k}_-(z)$ appearing in Proposition \ref{prop:variationturaev} as a flat trace similar to the one appearing in Proposition \ref{prop:variationconnexion}. This will be crucial for the comparison between $\tau_\vartheta$ and $\tau_{\frak{e,o}}$.

\begin{prop}\label{prop:resolventmorse}
For $\varepsilon > 0$ small enough, the wavefront set of the Schwartz kernel of the operator $\iota_{-\widetilde{X}}\widetilde{Y}Ê(\id - \widetilde{\Pi}_-) \e^{-\varepsilon \Lientildem}$ does not meet the conormal to the diagonal in $M \times M$ and we have for any $\alpha \in \Omega^1(M, \mathrm{End}(E))$
$$
\strf \alpha \iota_{-\widetilde{X}} \widetilde{Y}Ê(\id - \widetilde{\Pi}_-)\e^{-\varepsilon \Lientildem} = 0.
$$
\end{prop}
We refer to appendix \ref{sec:wfmorse} for the proof. An immediate corollary is the formula
\begin{equation}\label{eq:tracetoflattrace}
\str{\widetilde{C}^\bullet_-}\widetilde{\Pi}_-\alpha \tilde{k}_- = \strf \alpha \widetilde{K} \e^{-\varepsilon \Lientildem}.
\end{equation}
Indeed, since $\Lientildem \widetilde{\Pi}_-Ê= 0$, we have $\widetilde{\Pi}_-\e^{-\varepsilon \Lientildem} = \widetilde{\Pi}_-$. Moreover, since the trace of finite rank operators coincides with the flat trace, we have $\str{\widetilde{C}^\bullet_-}\widetilde{\Pi}_- \alpha \tilde k_- = \str{\widetilde{C}^\bullet_-}\widetilde{\Pi}_- \alpha \tilde k_-\e^{-\varepsilon \Lientildem} = \strf \alpha \tilde k_- \widetilde{\Pi}_-\e^{-\varepsilon \Lientildem}.$ Therefore we obtain with Proposition \ref{prop:resolventmorse}
$$
\begin{aligned}
\str{\widetilde{C}^\bullet}\widetilde{\Pi}_-\alpha \tilde{k}_- &= \strf \alpha \iota_{-\widetilde{X}} \widetilde{Y}Ê(\id - \widetilde{\Pi}_-)\e^{-\varepsilon \Lientildem} + \strf \alpha \tilde k_- \widetilde{\Pi}_-\e^{-\varepsilon\Lientildem},
\end{aligned}
$$
which gives (\ref{eq:tracetoflattrace}).

\subsection{Proof of Proposition \ref{prop:comparison}}
Note that we have by (\ref{eq:rhonabla})
$$\tau_\vartheta(\rho(z)) = \tau_\vartheta(\nabla(z)), \quad \tau_{\frak{e,o}}(\rho(z)) = \tau_{\frak{e,o}}(\nabla(z)).$$
We will set $f(z) = \tau_\vartheta(\nabla(z)) / \tau_{\frak{e,o}}(\nabla(z))$ for simplicity. Now we apply Proposition \ref{prop:variationconnexion}, Proposition \ref{prop:variationturaev} to obtain that $z \mapsto f(z)$ is real differentiable (since $z \mapsto \nabla(z)$ is); moreover it is complex differentiable at $z=0$ by (\ref{eq:nabla(z)}) and for $\varepsilon > 0$ small enough we have
\begin{equation}\label{eq:comparison1}
\left.\frac{\dd}{\dd z}\right|_{z=0} \log f(z)= -\strf \alpha K\e^{-\varepsilon \Lien} + \strf\alpha \widetilde{K}\e^{-\varepsilon\Lientildem} + \bigl\langle \tr \alpha, e \bigr\rangle,
\end{equation}
where we used (\ref{eq:tracetoflattrace}). Let 
$$
\Delta = \nabla \nabla^\star + \nabla^\star \nabla ~:~ \Omega^\bullet(M,E) \to \Omega^\bullet(M,E)
$$
be the Hodge-Laplace operator induced by any metric on $M$ and any Hermitian product on $E$. Because $\nabla$ is acyclic, $\Delta$ is invertible and Hodge theory gives that its inverse $\Delta^{-1}$ is a pseudo-differential operator of order $-2$. Define
$$
J = \nabla^\star \Delta^{-1} ~:~ \mathcal{D}^{'\bullet}(M,E) \to \mathcal{D}^{'\bullet-1}(M,E).
$$
We have of course
\begin{equation}\label{eq:cochainlaplacian}
[\nabla, J] = \nabla J + J \nabla = \id_{\mathcal{D}^{'\bullet}(M,E)}.
\end{equation}
Let $R_\varepsilon$ be the interpolator at time $\varepsilon$ defined in \S\ref{subsec:homotopy} for the pair of vector fields 
$(-\widetilde{X}, X)$. This implies with (\ref{eq:homotopy})
\begin{equation}
[\nabla, R_\varepsilon] = \e^{-\varepsilon \Lien} - \e^{-\varepsilon \Lientildem}.
\end{equation}
Now define
$$
G_\varepsilon = J\left(K\e^{-\varepsilon \Lien} - \widetilde{K}\e^{-\varepsilon \Lientildem} -R_\varepsilon\right) : \Omega^\bullet(M,E) \to \mathcal{D}^{'\bullet-2}(M,E).
$$
Let us compute, having (\ref{eq:cochainlaplacian}) in mind,
$$
\begin{aligned}
\left[\nabla, G_\varepsilon \right] &=  \nabla J \left(K\e^{-\varepsilon \Lien} - \widetilde{K}\e^{-\varepsilon \Lientildem} -R_\varepsilon\right) Ê- J\left(K\e^{-\varepsilon \Lien} - \widetilde{K}\e^{-\varepsilon \Lientildem} -R_\varepsilon\right) \nabla \\
&= \left(\id - J\nabla\right)\left(K\e^{-\varepsilon \Lien} - \widetilde{K}\e^{-\varepsilon \Lientildem} -R_\varepsilon\right)  -J\left(K \nabla \e^{-\varepsilon \Lien} - \widetilde{K} \nabla \e^{-\varepsilon \Lientildem} -R_\varepsilon\nabla \right) \\
&= K  \e^{-\varepsilon \Lien} - \widetilde{K}  \e^{-\varepsilon \Lientildem} -R_\varepsilon - J
\left([\nabla, K]\e^{-\varepsilon \Lien} - [\nabla, \widetilde{K}]\e^{-\varepsilon\Lientildem} - [\nabla, R_\varepsilon]\right),
\end{aligned}
$$
where we used that $\e^{-\varepsilon \Lien}$ and $\e^{-\varepsilon\Lientildem}$ commute with $\nabla$. Now note that (\ref{eq:cochainanosov}), (\ref{eq:homotopy}) and (\ref{eq:cochainmorse}) imply
$$
[\nabla, K]\e^{-\varepsilon \Lien} - [\nabla, \widetilde{K}]\e^{-\varepsilon\Lientildem} - [\nabla, R_\varepsilon] = \e^{-\varepsilon \Lien} - \e^{-\varepsilon\Lientildem} - \left(\e^{-\varepsilon \Lien} - \e^{-\varepsilon\Lientildem}\right) = 0.
$$
Therefore we obtained
$$
[\nabla, G_\varepsilon] = K \e^{-\varepsilon \Lien} - \widetilde{K}  \e^{-\varepsilon \Lientildem} -R_\varepsilon.
$$
Because $[\nabla, \alpha] = 0$ we have
$$
[\nabla, \alpha G_\varepsilon] = -\alpha\left(K \e^{-\varepsilon \Lien} - \widetilde{K}  \e^{-\varepsilon \Lientildem} -R_\varepsilon\right).
$$ 
Using the notations of \S\ref{subsec:flattrace},
$\WF(J), \WF(\alpha),\WF(\nabla)$ are contained in the conormal bundle of the diagonal $N^*\Delta$ since $J,\alpha,\nabla$ are pseudodifferential operators; moreover, equation (\ref{eq:wfinterpolator}) shows that
$$\WF\left(K\e^{-\varepsilon \Lien} - \widetilde{K}\e^{-\varepsilon \Lientildem} -R_\varepsilon\right)\cap N^*\Delta=\emptyset.$$ 
It follows from wave front composition~\cite[Theorem 8.2.14]{hor1}
that $\WF(\alpha G_\varepsilon)\cap N^*\Delta=\emptyset$. The operators $\nabla,\alpha G_\varepsilon$ satisfy the assumptions of Proposition~\ref{p:flattracesupercomm}
which gives $\strf~ [\nabla, \alpha G_\varepsilon] = 0$ and therefore (\ref{eq:comparison1}) reads
\begin{equation}\label{eq:comparison2}
\left.\frac{\dd}{\dd z}\right|_{z=0} \log f(z) = - \strf \alpha R_\varepsilon + \bigl\langle \tr \alpha, e \bigr\rangle.
\end{equation}
The identity $[\nabla, \alpha] = 0$ also implies that $\dd \tr \alpha = \tr \nabla^{E \otimes E^\vee} \alpha = \tr [\nabla, \alpha] = 0.$ As a consequence we can apply (\ref{eq:homotopycs}) to obtain
$$
\begin{aligned}
\strf \alpha R_\varepsilon = \bigl \langle \tr \alpha,  ~\cs (-\widetilde{X},X) \bigr \rangle.
\end{aligned}
$$
Now note that $\partial \bigl(-\cs( -\widetilde{X}, X) + e\bigr) =-\bigl(\div(X) -\div(-\widetilde{X})\bigr) + \div(\widetilde{X}) = 0$ by (\ref{eq:propcs}) and (\ref{eq:partiale}) since $X$ is non singular. Therefore we obtain
$$
\left.\frac{\dd}{\dd z} \right|_{z=0} \log f(z)  = \bigl\langle \tr \alpha, h \bigr\rangle
$$
where $h = [-\cs(-\widetilde{X}, X) + e] \in H_1(M, \mathbb{Z})$. Finally, let us note that by (\ref{lem:holonomy}),
$$
\left.\frac{\dd}{\dd z}\right|_{z=0} \log \bigl \langle \det \rho(z), h \bigr \rangle = - \bigl\langle \tr \alpha, h \bigr \rangle,
$$
since $\rho(z) = \rho_{\nabla(z)}$. Therefore the proposition is proved for $z=0$. However the same argument holds for every $z$ close enough to $0$, which concludes.

\subsection{Proof of Theorems \ref{thm:main2} and \ref{thm:main3} }
By Hartog's theorem and Proposition \ref{prop:comparison}, we have that the map
\begin{equation}\label{eq:map}
\rho \mapsto \frac{\tau_\vartheta(\rho)}{\tau_{\frak{e,o}}(\rho)}{\langle \det \rho, h \rangle}
\end{equation}
is locally constant on $\Rep_{\mathrm{ac}}(M,d) \setminus \Sigma(M,d)$.

Moreover, we can reproduce all the arguments we made in the continuous category to obtain that $\rho \mapsto {\tau_\vartheta(\rho)}/{\tau_{\frak{e,o}}(\rho)}$ is actually continuous on $\Rep_{\mathrm{ac}}(M,d)$. Because $\Rep_{\mathrm{ac}}(M,d) \setminus \Sigma(M,d)$ is open and dense in $\Rep_{\mathrm{ac}}(M,d)$, we get that the map \ref{eq:map} is locally constant on $\Rep_{\mathrm{ac}}(M,d)$.

By \cite[p.~211]{farber2000poincare} we have, if $\frak{e}'$ is another Euler structure, $\tau_{\frak{e'},\frak{o}}(\rho) = \langle \det \rho, \frak{e'} -  \frak{e} \rangle \tau_{\frak{e,o}}(\rho)$. As a consequence, if we set $\frak{e}_\vartheta = [-X, 0]$ which defines an Euler structure since $X$ is nonsingular (see \S\ref{subsec:eulerstruct}), we have $\frak{e} - \frak{e}_\vartheta = h$ and we obtain that $\rho \mapsto \tau_\vartheta(\rho) / \tau_{\frak{e}_\vartheta,\frak{o}}(\rho)$ is locally constant on $\Rep_{\mathrm{ac}}(M,d)$. 

Now let $\eta$ be another contact form inducing an Anosov Reeb flow and denote by $X_\eta$ its Reeb flow. Then if $\frak{e_\eta} = [-X_\eta, 0],$ we have
$$
\frak{e}_\eta - \frak{e}_\vartheta = \cs(X,X_\eta)
$$
by definition. Therefore $\tau_{\frak{e}_\vartheta,\frak{o}}(\rho) = \tau_{\frak{e_\eta,o}}(\rho) \langle \det \rho, \frak{e_\vartheta}- \frak{e_\eta} \rangle =  \tau_{\frak{e}_\eta,\frak{o}}(\rho) \langle \det \rho, \cs(X_\eta,X)\rangle$ and we obtain that
$$
\rho \mapsto \frac{\tau_\vartheta(\rho)}{\tau_\eta(\rho)} \langle \det \rho, \cs(X,X_\eta) \rangle
$$
is locally constant on $\Rep_{\mathrm{ac}}(M,d)$. By Theorem \ref{theo:invariance} we thus obtain Theorem \ref{thm:main3}.

Finally assume that $\dim M = 3$ and $b_1(M) \neq 0$. Take $\mathcal{R}$ a connected component of $\Rep_{\mathrm{ac}}(M,d)$ and assume that it contains an acyclic and unitary representation $\rho_0$. We invoke \cite[Theorem 1]{dang2018fried} and the Cheeger-M\"uller theorem \cite{cheeger1979analytic,muller1978analytic} to obtain that $0 \notin \Res(\Lie_X^{\nabla_{\rho_0}})$ and
$$|\tau_\vartheta(\rho_0)| = |\zeta_{X, \nabla_{\rho_0}}(0)|^{-1} = \tau_{\mathrm{RS}}(\rho_0),$$
where the first equality comes from (\ref{eq:notresonance}) (we have $q = 1$ since $\dim M = 3$) and $\tau_\mathrm{RS}(\rho_0)$ is the Ray-Singer torsion of $(M,\rho_0)$, cf. \cite{ray1971r}. On the other hand, we have by \cite[Theorem 10.2]{farber2000poincare} that $\tau_\mathrm{RS}(\rho_0)Ê= |\tau_{\frak{e,o}}(\rho_0)|$ since $\rho_0$ is unitary. Therefore the map $\rho \mapsto \tau_\vartheta(\rho) / \tau_{\frak{e_\vartheta,o}}(\rho)$ is of modulus one on $\mathcal{R}$. This concludes the proof of Theorem \ref{thm:main2}.

\appendix

\section{Projectors of finite rank}\label{sec:tracefiniterank}

\subsection{Traces on variable finite dimensional spaces}
In what follows, we consider two Hilbert spaces $\G \subset \H$, the inclusion being dense and continuous. We will denote by $\Lie(\H,\G)$ the space of bounded linear operators $\H \to \G$ endowed with the operator norm. Let $\delta > 0$ and $\Pi_t, |t| \leq \delta$, be a family of finite rank projectors on $\H$ such that $\mathrm{ran} \ \Pi_t \subset \G.$ Assume that $t \mapsto \Pi_t$ is differentiable at $t=0$ as a family of bounded operators $\H \to \G$, that is,
\begin{equation}\label{eq:pidev}
\Pi_t = \Pi + tP + o_{\mathcal{H} \to \mathcal{G}}(t)
\end{equation}
for some $P \in \mathcal{L}(\mathcal{H}, \mathcal{G})$, where $\Pi = \Pi_0$. Denote $C_t = \mathrm{ran} \ \Pi_t $ and $C = \mathrm{ran} \ \Pi$. Note that by continuity, $\Pi_t|_C : C \to C_t$ is invertible for $|t|$ small enough; we denote by $Q_t : C_t \to C$ its inverse. 

\begin{lemm}\label{lem:projector}
We have
\begin{enumerate}[label=(\roman*)]
\item $P = \Pi P + P \Pi,$
\item $Q_t \Pi_t= \Pi \Pi_t + o_{\H \to \G}(z).$
\end{enumerate}
\end{lemm}
\begin{proof}
Using (\ref{eq:pidev}) and $\Pi_t^2 = \Pi_t$ we obtain \textit{(i)}.
This implies
$$
\begin{aligned}
\Pi_t \circ \Pi \circ  \Pi_t &= \Bigl(\Pi + tP + o(t)\Bigr)   \Pi  \Bigl(\Pi + tP + o(t)\Bigr) \\
&= \Pi + t\Bigl(P\Pi + \Pi P\Bigr) + o(t) \\
&= \Pi + tP + o(t) \\
&= \Pi_t + o(t),
\end{aligned}
$$
where all the $o(t)$ are taken in $\Lie(\H, \G)$.
Therefore $Q_t \circ \Pi_t \circ \Pi \circ \Pi_t = Q_t \Pi_t + o(t)$. Since $Q_t \circ \Pi_t \circ \Pi = \Pi$ by definition, one obtains 
$$
Q_t \circ \Pi_t = \Pi \circ \Pi_t + o(t),
$$
which proves the first part of the Lemma. The second part is very similar.
\end{proof}

\begin{lemm}\label{lem:projector2}
Let $A_t,~  |t| \leq \delta,$ be a $\mathcal{C}^1$ family of bounded operators $\G \to \H$ such that $A_t$ commutes with $\Pi_t$ for every $t$. Denote $A = A_0$. Then $ t \mapsto \tr_{C_t}( A_t)$ is real differentiable at $t=0$ and
$$
\left.\frac{\dd}{\dd t}\right|_{t=0} \tr_{C_t}( A_t) = \tr_{C}\bigl(\Pi \dot A),
$$
where $\dot A_t = \frac{\dd}{\dd t} A_t$. If moreover $A$ is invertible on $C$, we have
$$
\left.\frac{\dd}{\dd t}\right|_{t=0}\log {\det}_{C_t} (A_t) = \tr_{C} \left( \Pi \dot A (A|_C)^{-1} \right).
$$
\end{lemm}

\begin{proof}
We start from
$$
\tr_{C_t}( A_t) = \tr_{C}(Q_tA_t \Pi_t).
$$
Now since $A_t$ commutes with $\Pi_t$ we have by the second part Lemma \ref{lem:projector}
$$
\begin{aligned}
Q_tA_t \Pi_t \Pi &=\Pi \Pi_t A_t \Pi + o_{C\to C}(t)\\
&= \Pi A \Pi + t \Pi \Bigl(\dot A + PA\Pi + \Pi A P \Bigr)\Pi + o_{C \to C}(t).
\end{aligned}
$$
But now the first part of Lemma \ref{lem:projector} gives $\Pi P \Pi$ = 0. We therefore obtain, because $A$ and $\Pi$ commute,
\begin{equation}\label{eq:qtatpit}
Q_tA_t \Pi_t \Pi =  \Pi A \Pi + t \Pi \dot A \Pi + o_{C \to C}(t),
\end{equation}
which concludes.
\end{proof}

\subsection{Gain of regularity}
Assume that we are given four Hilbert spaces 
$\mathcal{E}Ê\subset \mathcal{F} \subset \mathcal{G} \subset \mathcal{H}$ 
with continuous and dense inclusions. 
Let $\Pi_t$, ~$|t|Ê< \delta$ be a family of finite rank projectors 
on $\H$ which is differentiable at $t=0$ as family of bounded operators 
$\G \to \H$ (note that this differs from the last subsection where we had $\H \to \G$ instead), 
that is
\begin{equation}\label{eq:pidev}
\Pi_t = \Pi + tP + o_{\mathcal{G} \to \mathcal{H}}(t)
\end{equation}
for some $P \in \Lie(\mathcal{G}, \H)$. We will denote $C_t = \mathrm{ran}(\Pi_t) \subset \H$ and $C = \mathrm{ran}(\Pi)$.
\begin{lemm}\label{lem:projector3}
Under the above assumptions,
assume that $\Pi_t$ is bounded $\mathcal{E} \to \mathcal{F}$ and that $\Pi_t$ is differentiable at $t=0$ as a family of $\Lie(\mathcal{E}, \mathcal{F})$. Assume also that $\rank \Pi_t$ does not depend on $t$. Then $P$ is actually bounded $\mathcal{G}\to \mathcal{F}$ and
$$
\Pi_t = \Pi +t P + o_{\mathcal{G} \to \mathcal{F}}(t).
$$
\end{lemm}
\begin{proof}
Because $\mathcal{E}$ is dense in $\mathcal{H}$ we know that $C \subset \mathcal{F}$. There exists $\varphi^1, \dots, \varphi^m \in \mathcal{E}$ such that $\varphi^1_t, \dots, \varphi^m_t$ is a basis of $C_t$ for $t$ small enough where we set $\varphi^j_t = \Pi_t(\varphi^j) \in \mathcal{F}$. Denote $\tilde{\varphi}^j_t = \Pi(\varphi^j_t) \in C$. This family $t \mapsto \tilde{\varphi}^j_t \in C$ is differentiable at $t=0$. Let $\nu^1_t, \dots, \nu^m_t \in C^*$ be the dual basis of $\tilde{\varphi}^1_t, \dots, \tilde{\varphi}^m_t$. Because $C$ is finite dimensional, $\Pi$ is actually bounded $\mathcal{H}Ê\to \mathcal{F}$. As a consequence the map
$$t\mapsto\ell^j_t = \nu^j_t \circ \Pi \circ \Pi_t \in \mathcal{G}'$$
is differentiable at $t=0$. Noting that
$$
\Pi_t = \sum_{j=1}^m \varphi^j_t \otimes \ell^j_t : \mathcal{G}Ê\to \mathcal{F},
$$
we finally obtain that $t \mapsto \Pi_t \in \Lie(\mathcal{G}, \mathcal{F})$ is differentiable at $t=0$.
\end{proof}

\section{Continuity of the Pollicott-Ruelle spectrum}\label{sec:continuityruelle}
We describe here the spaces used in \S\S\ref{sec:invariance},\ref{sec:variationconnexion}. In what follows, $M$ is a compact manifold, $(E, \nabla)$ a flat vector bundle on $M$ and $X_0$ is a vector field on $M$ generating an Anosov flow, cf. \S\ref{subsec:anosov}. We denote by $T^*M = E_{u,0}^* \oplus E_{s,0}^* \oplus E_{0,0}^*$ its Anosov decomposition of $T^*M$.

\subsection{Bonthonneau's uniform weight function}
We state here a lemma from Bonthonneau which is \cite[Lemma 3]{bonthonneau2018flow}. This gives us an escape function having uniform good properties for a family of vector fields. A consequence is that one can define some uniform anisotropic Sobolev spaces on which each vector field of the family has good spectral properties. In what follows, $|\cdot|$ is a smooth norm on $T^*M$.

\begin{lemm}\label{lem:bonthonneau}
There exists conical neighborhoods $N_u$ and $N_s$ of $E_{u,0}^*$ and $E_{s,0}^*$, some constants $C, \beta, T, \eta > 0$, and a weight function $m \in \Cinf(T^*M, [0,1])$ such that the following holds. Let $X$ be any vector field satisfying $\|X-X_0\|_{\mathcal{C}^1} < \eta$, and denote by $\Phi^t$ its induced flow on $T^*M$ and by $E_u^*$ and $E_s^*$ its (dual) unstable and stable bundles. Then
\begin{enumerate}
\item $E_\bullet^* \subset N_\bullet$, for $\bullet = s,u$ and for any $t > 0$, $\xi_u \in E_u^*$ and $\xi_s \in E_s^*$ one has
$$
|\Phi^t(\xi_u)| \geq \frac{1}{C}\e^{\beta t}|\xi_u|, \quad |\Phi^{-t}(\xi_s)| \geq \frac{1}{C}\e^{\beta t}|\xi_s|. 
$$
\item For every $t \geq T$ it holds
$$
\Phi^t\left(\complement N_s \cap X^\perp\right) \subset N_u, \quad \Phi^{-t}\left(\complement N_u \cap X^\perp\right) \subset N_s,
$$
where $X^\perp = \{\xi \in T^*M, ~Ê\xi \cdot X = 0\}.$
\item If $\bf{X}$ is the Lie derivative induced by $\Phi^t$, then
$$
m \equiv 1 \text{ near } N_s, \quad m \equiv -1 \text{ near } N_u, \quad \mathbf{X}.m \geq 0.
$$
\end{enumerate}
\end{lemm}

\subsection{Anisotropic Sobolev spaces}\label{subsec:appendixanisotropic}
Take the weight function $m$ of Lemma \ref{lem:bonthonneau}. Define the escape function $g$ by 
$$
g(x,\xi) = m(x,\xi) \log (1 + |\xi|), \quad (x,\xi) \in T^*M.
$$
We set $G = \Op(g) \in \Psi^{0+}(M)$ for any quantization procedure $\Op$. Then by \cite[\S\S8.3,9.3,14.2]{zworski} we have $\exp(\pm \mu G) \in \Psi^{\mu+}(M)$ for any $\mu> 0$. For any $\mu>0$ and $j \in \mathbb{Z}$ we define the spaces
$$
\H_{\mu G,j}^\bullet = \exp(-\mu G) H^j(M, \Lambda^\bullet \otimes E) \subset \mathcal{D}^{'\bullet}(M,E),
$$
where  $H^j(M, \Lambda^\bullet \otimes E)$ is the usual Sobolev space of order $j$ on $M$ with values in the bundle $\Lambda^\bullet \otimes E$. Note that any pseudo-differential operator of order $m$ is bounded $\H^\bullet_{\mu G,j} \to \H^\bullet_{\mu G,j-m}$ for any $\mu,m,j$.

\subsection{Uniform parametrices}
Let us consider a smooth family of vector fields $X_t, ~|t| < \varepsilon,$ perturbing $X_0$. For any $c, \rho > 0$ we will denote
$$
\Omega(c, \rho) = \{\Re(s) > c \}Ê\cup \{|s|Ê\leq \rho\} \subset \mathbb{C}.
$$
The spaces defined in the last subsection yields an uniform version of \cite[Proposition 3.4]{dyatlov2013dynamical}, as follows.
\begin{prop}{\cite[Lemma 9]{bonthonneau2018flow}}\label{lem:uniformparametrices}
Let $Q$ be a pseudo-differential operator micro-locally supported near the zero section in $T^*M$ and elliptic there. There exists $c, \varepsilon_0 > 0$ such that for any $\rho > 0$ and $J \in \mathbb{N}$, there is $\mu_0, h_0 > 0$ such that  the following holds. For each $\mu \geq \mu_0$, $0 < h < h_0$, $j \in \mathbb{Z}$ such that $|j|Ê\leq J$ and $s \in \Omega(c, \rho)$ the operator
$$
\Lient - h^{-1} Q + s : \H^\bullet_{\mu G, j + 1} \to \H^\bullet_{\mu G,j}
$$
is invertible for $|t|Ê\leq \varepsilon_0$ and the inverse is bounded $\H^\bullet_{\mu G,j} \to \H^\bullet_{\mu G,j}$ independently of $t$.
\end{prop}

\subsection{Continuity of the Pollicott-Ruelle spectrum}\label{subsec:continuityruelle}
We fix $\rho, J \geq 4$ and $\mu_0, \mu, h_0, h, j$ as in Proposition \ref{lem:uniformparametrices}. 
We first observe that
\begin{equation}\label{eq:fredholm}
\left(\Lient+s\right)\left(\Lient - h^{-1} Q + s\right)^{-1} = \id + h^{-1} Q \left(\Lient - h^{-1} Q + s\right)^{-1}.
\end{equation}
Since $Q$ is supported near $0$ in $T^*M$, it is smoothing and thus trace class on any $\H^\bullet_{\mu G,j}$. By analytic Fredholm theory, the family $s \mapsto K(t,s) = h^{-1} Q \left(\Lient - h^{-1} Q + s\right)^{-1}$ is a holomorphic family of trace class operators on $\H^\bullet_{\mu G,j}$ in the region $\Omega(c,\rho)$. We can therefore consider the Fredholm determinant
$$
D(t,s) = {\det}_{\H^\bullet_{\mu G,j}} \bigl(\id + K(t,s)\bigr).
$$
It follows from \cite[Corollary 2.5]{simon2005trace} that for each $t$, $s \mapsto D(t,s)$ is holomorphic on $\Omega(c,\rho)$. Moreover (\ref{eq:fredholm}) shows that its zeros coincide, on $\Omega(c,\rho)$, with the 
Pollicott-Ruelle resonances of $\Lient$. In addition, we have for any $s \in \Omega(c, \rho)$,
\begin{equation}\label{eq:resolventformula}
\begin{aligned}
\Bigl(\Lient - h^{-1}& Q + s\Bigr)^{-1}  - \left(\Lie_{X_{t'}}^\nabla - h^{-1} Q + s\right)^{-1}  \\
&=  -\left(\Lient - h^{-1} Q + s\right)^{-1} \left( \Lient - \Lie_{X_{t'}}^\nabla \right) \left(\Lie_{X_{t'}}^\nabla - h^{-1} Q + s\right)^{-1}.
\end{aligned}
\end{equation}
We have
\begin{equation}\label{eq:partialX}
\frac{\Lient - \Lie_{X_{t'}}^\nabla}{t-t'} \underset{t\to t'}{\longrightarrow} \Lie_{\dot{X}_t}^\nabla \text{ in }\Lie(\H^\bullet_{\mu G,j+1}, \H^\bullet_{\mu G,j}).
\end{equation}
where $\displaystyle{\dot{X}_t = \frac{\dd}{\dd t}X_t}$ and $\Lie(\H^\bullet_{\mu G,j+1}, \H^\bullet_{\mu G,j})$ is the space of bounded linear operators $\H^\bullet_{\mu G,j+1} \to \H^\bullet_{\mu G,j}$ endowed with the operator norm. We therefore obtain by Proposition \ref{lem:uniformparametrices} and because $Q$ is smoothing (and thus trace class $\H^\bullet_{\mu G,j} \to \H^\bullet_{\mu G,j'}$ for any $\mu,j,j'$) that $K(t',s) \to K(t,s)$ as $t'\to t$ in $\mathcal{L}^1(\H^\bullet_{\mu G,0})$ locally uniformly in $s$, where $\mathcal{L}^1(\H^\bullet_{\mu G,0})$ is the space of trace class operators on $\H^\bullet_{\mu G,0}$ endowed with its usual norm. As a consequence, we obtain with \cite[Corollary 2.5]{simon2005trace}
\begin{equation}\label{eq:rouchŽ}
D(t,s) \in \mathcal{C}^0 \bigl([-\varepsilon_0, \varepsilon_0]_t, \mathrm{Hol}\bigl(\Omega(c,\rho)_s\bigr)\bigr).
\end{equation}

\subsection{Regularity of the resolvent}\label{subsec:regularityresolvent}
Let $\mathcal{Z}$ be an open set of $\mathbb{C}$ whose closure is contained in the interior of $\Omega(c,\rho)$. We assume that $\overline{\mathcal{Z}} \cap \Res(\Lie_{X_0}^\nabla) = \emptyset$. Up to taking $\varepsilon_0$ smaller, Rouch\'e's theorem and (\ref{eq:rouchŽ}) imply that there exists $\delta > 0$ such that $\mathrm{dist}\left(\mathcal{Z}, ~\Res(\Lient)\right) > \delta$ for any $|t|Ê\leq \varepsilon_0$. As a consequence, we obtain that for every $|j| \leq J$, the map $\left(\Lient + s\right)^{-1} : {\H^\bullet_{\mu G,j} \to \H^\bullet_{\mu G,j}}$  is bounded independently of $(t,s) \in [-\varepsilon_0, \varepsilon_0] \times \mathcal{Z}$. Noting that
\begin{equation}\label{eq:resolventformula2}
\frac{\Bigl(\Lient + s\Bigr)^{-1} - \left(\Lie_{X_{t'}}^\nabla + s\right)^{-1}}{t-t'} = -\Bigl(\Lient + s\Bigr)^{-1} \frac{\Lient - \Lie_{X_{t'}}^\nabla }{t-t'}\left(\Lie_{X_{t'}}^\nabla + s\right)^{-1},
\end{equation}
we obtain by (\ref{eq:partialX}) that $t' \mapsto \left(\Lie_{X_{t'}}^\nabla + s\right)^{-1}$ is continuous in $\Lie(\H^\bullet_{\mu G,j+1}, \H^\bullet_{\mu G,j})$. Therefore, applying (\ref{eq:resolventformula2}) again, we get that
\begin{equation}\label{eq:regularityresolvent}
\left(\Lient +s \right)^{-1} \in \mathcal{C}^1\bigl([-\varepsilon_0, \varepsilon_0]_t, \mathrm{Hol}(\mathcal{Z}_s,~ \Lie(\H^\bullet_{\mu G,j+1}, \H^\bullet_{\mu G,j-2})\bigr).
\end{equation}
Note that here we need $|j-2|, |j+1|Ê\leq J$.

\subsection{Regularity of the spectral projectors}\label{subsec:regularityprojectors}
Let $0 < \lambda < 1$ such that $\{|s| = \lambda\}Ê\cap \Res(\Lie_{X_0}^\nabla) = \emptyset.$ Applying the last subsection with $\mathcal{Z} = \{|s| = \lambda\}$, we get $\{|s| = \lambda\}Ê\cap \Res(\Lie_{X_t}^\nabla) = \emptyset$ for any $|t|Ê\leq \varepsilon_0$. We can therefore define for those $t$
$$
\Pi_t = \frac{1}{2\pi i} \int_{|s| = \lambda} \left(\Lient + s\right)^{-1} \dd s : \H^\bullet_{\mu G,j} \to \H^\bullet_{\mu G,j}.
$$
Then (\ref{eq:regularityresolvent}) gives that $\Pi_t \in \mathcal{C}^1\bigl([-\varepsilon_0, \varepsilon_0]_t, \mathcal{Z}_s,~ \Lie(\H^\bullet_{\mu G,j+1}, \H^\bullet_{\mu G,j-2}\bigr).$ This is true for $j=3$ and $j=-1$ because $J \geq 4$. Moreover by Rouch\'e's theorem, the number $m$ of zeros of $s \mapsto D(t,s)$ does not depend on $t$. Noting that
$$
\partial_sK(t,s)(1 + K(t,s))^{-1} = -K(t,s)\left(\Lient - h^{-1}Q +s)^{-1}(1+K(t,s)\right)^{-1},
$$
we obtain by \cite[Theorem C.11]{dyatlov2019mathematical} and the cyclicity of the trace that $m$ is equal to
$$
\begin{aligned}
\frac{1}{2 \pi i} \tr \int_{|s| = \lambda} \partial_sK(t,s) &(1+K(t,s))^{-1} \dd s \\
&= - \frac{1}{2\pi i} \tr \int_{|s| = \lambda} \left(\Lient -h^{-1}Q + s\right)^{-1}(1 + K(t,s))^{-1} K(t,s) \dd s \\
&= \frac{1}{2\pi i} \tr \int_{|s| = \lambda} \left(\Lient -h^{-1}Q + s\right)^{-1}(1 + K(t,s))^{-1},
\end{aligned}
$$
where we used that $s \mapsto \left(\Lient - h^{-1}Q + s\right)^{-1}$ is holomorphic on $\{|s|Ê\leq \lambda\}$.
The last integral is equal to $\tr \Pi_t = \rank \Pi_t$ by (\ref{eq:fredholm}). As a consequence we can apply Lemma \ref{lem:projector3} to obtain that
\begin{equation}\label{eq:regularityprojector}
\Pi_t \in \mathcal{C}^1\bigl([-\varepsilon_0, \varepsilon_0]_t, \Lie(\H^\bullet_{\mu G,0}, \H^\bullet_{\mu G,1}\bigr).
\end{equation}

\subsection{Wavefront set of the spectral projectors}\label{subsec:wfsetprojectors}
Let $(E, \nabla^\vee)$ be the dual bundle of $(E, \nabla)$. Then (\ref{eq:resolvent}) implies, for any $\Re(s) \gg 0$,
\begin{equation}\label{eq:dualresolvent}
\left\langle \Bigl(\Lient+s\Bigr)^{-1}u, v \right \rangle = \left \langle u, \left(\Lie_{-X_t}^{\nabla^\vee} +s \right)^{-1}v \right\rangle, \quad u \in \Omega^k(M,E), \quad v \in \Omega^{n-k}(M, E^\vee),
\end{equation}
where $\langle \cdot, \cdot \rangle$ is the pairing from \S \ref{subsec:currents}. This shows that $\Res(\Lie_{-X_t}^{\nabla^\vee}) = \Res(\Lient)$. Therefore we can apply the preceding construction with the escape function $g$ replaced by $-g$ (the unstable bundle of $-X_t$ is the stable one of $X_t$ and reciprocally) and we obtain that 
$$
\displaystyle{\Pi_t^\vee = \frac{1}{2\pi i} \int_{|s| = \lambda} \left(\Lie_{-X_t}^{\nabla^\vee} + s\right)^{-1} \dd s}  ~\in~ \mathcal{C}^1\bigl([-\varepsilon_0, \varepsilon_0]_t, \Lie(\H^\bullet_{-\mu G,0}, \H^\bullet_{-\mu G,1})\bigr).
$$
Note that (\ref{eq:dualresolvent}) implies 
\begin{equation}\label{eq:dualprojector}
\left \langle \Pi_t u, v \right \rangle = \left \langle u, \Pi_t^\vee v \right \rangle, \quad u \in \Omega^k(M,E), \quad v \in \Omega^{n-k}(M,E^\vee).
\end{equation}
 We denote $C^\bullet_t = \mathrm{ran}~\Pi_t$,  $C^{\vee \bullet}_t = \mathrm{ran}~\Pi_t^\vee$ and $m = \mathrm{rank}~\Pi_t = \mathrm{rank}~\Pi_t^\vee $. Take $\varphi^1, \dots, \varphi^m, \psi^1, \dots, \psi^m \in \Omega^\bullet(M,E)$ such that $\Pi_0(\varphi^1), \dots \Pi_0(\varphi^m)$ is a basis of $C^\bullet_0$ and $\langle \Pi_0 \varphi^i, \psi^j \rangle =  0$ if $i \neq j$ and $\langle \Pi_0 \varphi^i, \psi_j \rangle = 1$ otherwise. For $t$ small enough we set
$$
\varphi^i_t = \Pi_t \varphi^i, \quad \psi^t_j = \Pi_t^\vee \psi^j.
$$
Like in the proof of Lemma \ref{lem:projector3}, (\ref{eq:dualprojector}) implies that
\begin{equation}\label{eq:formulaprojectordual}
\Pi_t = \sum_{i=1}^m m_{ij}(t) \varphi^i_t \langle \psi^j_t, \cdot \rangle,
\end{equation}
where $t \mapsto m_{ij}(t)$ is continuous near $t=0$ and $m_{ij}(0) = \delta_{ij}$.

Next we show that there exists open conic neighborhoods of $N_u$ and $N_s$ such that, uniformly in $t \in [-\varepsilon_0, \varepsilon_0]$,
\begin{equation}\label{eq:wfsetprojector}
\WF(\varphi^i_t) \subset W_u, \quad \WF(\psi^i_t) \subset W_s, \quad W_u \cap W_s = \emptyset, \quad i =1, \dots, m.
\end{equation}
This means that the map $[-\varepsilon_0, \varepsilon_0] \ni t \mapsto \varphi^i_t$ (resp. $\psi^i_t$) is bounded in $\mathcal{D}^{'\bullet}_{W_u}(M,E)$ (resp. $\mathcal{D}^{'\bullet}_{W_s}(M,E^\vee)$). To proceed, we note that we can construct two weight functions $m_u, m_s$ satisfying the properties of Lemma \ref{lem:bonthonneau} such that $\{m_u \leq 0\}Ê\cap \{m_s \geq 0\} = \emptyset$ (for example by choosing well the $\chi$ from \cite[p. 6]{bonthonneau2018flow}). Let $G_u , G_s \in \Psi^{0+}(M)$ be the associated operators from \S\ref{subsec:appendixanisotropic}. 
Up to choosing $\varepsilon_0$ smaller, we obtain with (\ref{eq:regularityprojector}) that the map $t \mapsto \varphi^i_t$ is bounded in $\H^\bullet_{\mu G_u, 0}$ for $\mu > 0$ big enough. For any $\chi \in \Cinf(T^*M, [0,1])$  such that $\supp\chi \subset \{m_u \geq \delta\}$ for some $\delta > 0$, we have by classical rules of pseudo-differential 
calculus
$$
\|\Op(\chi)\varphi^i_t\|_{H^{\delta\mu}(M,\Lambda^\bullet \otimes E)} \leq C_\mu \| \varphi^i_t \|_{\H^\bullet_{\mu G_u, 0}} \leq C'_\mu, \quad t \in [-\varepsilon_0, \varepsilon_0],
$$
for some constants $C_\mu, C'_\mu$ independent of $t$. As a consequence, we obtain (for example using \cite[Lemma 7.4]{dangwitten}) that $[-\varepsilon_0, \varepsilon_0]\ni t \mapsto \varphi^i_t$ is bounded in $\mathcal{D}^{'\bullet}_{W_u}(M,E)$ where $W_u = \{m_u \leq 0\}$. Doing exactly the same with $-m_s$ and $-X_t$ we obtain that $[-\varepsilon_0, \varepsilon_0]\ni t \mapsto \psi^i_t$ is bounded in $\mathcal{D}^{'\bullet}_{W_s}(M,E^\vee)$ with $W_s = \{-m_s \geq 0\}$. This shows (\ref{eq:wfsetprojector}).

\section{The wave front set of the Morse-Smale resolvent}\label{sec:wfmorse}
The purpose of this section is to prove Proposition \ref{prop:resolventmorse}. 
For simplicity we prove it for $\widetilde{X}$ instead of $-\widetilde{X}$. 
We will denote by $\widehat{\Pi}$ the spectral projector (\ref{eq:projectormorsesmale}) 
for the trivial bundle $(\mathbb{C}, \dd)$. Recall that $\mathcal{D}^\prime_{\Gamma}(M\times M)$ denotes
distributions whose wave front set is contained in the closed conic set $\Gamma\subset T^\bullet(M\times M)$.
A family $(f_t)_{t\geq 0}$ of distributions will be $\mathcal{O}_{\mathcal{D}^\prime_\Gamma}(1)$ if it is bounded in
$\mathcal{D}^\prime_\Gamma$ in the sense of~\cite[p.~31]{dangthesis}.
We will need the following
\begin{lemm}\label{lem:wfsetpropagator}
Let $\varepsilon > 0$ and $a \in \crit(f)$. There exists $c > 0$, a closed conic set $\Gamma \subset T^*(M\times M)$ with $\Gamma \cap N^*\Delta(T^*M) = \emptyset$ and $\chi \in \Cinf(M,[0,1])$ such that $\chi \equiv 1$ near $a$ such that
$$
\mathcal{K}_{\chi, t + \varepsilon} = \dom_{\mathcal{D}^{'n}_\Gamma(M \times M)}(\e^{-tc}),
$$
where for $t \geq 0$, $\mathcal{K}_{\chi, t}$ is the Schwartz kernel of the operator $\chi \e^{-t\Lie_{\widetilde{X}}}\left(\id - \widehat{\Pi}\right) \chi$.
\end{lemm}

\begin{proof}
Because $\widetilde{X}$ is $\Cinf$-linearizable, we can take $U \subset \mathbb{R}^n$ to be a coordinate patch centered in $a$ so that, in those coordinates, $\e^{-t\widetilde{X}}(x) = \e^{-t A}(x)$ where $A$ is a matrix whose eigenvalues have nonvanishing real parts. Denoting $(x^1, \dots, x^n)$ the coordinates of the patch, $\widetilde{X}$ reads
$$
\widetilde{X} = \sum_{1\leq i,j \leq n}A_i^j x^i \partial_j.
$$
We have a decomposition $\mathbb{R}^n = W^u \oplus W^s$ stable by $A$ such that $A|_{W^u}$ (resp. $A|_{W^s}$) have eigenvalues with positive (resp. negative) real parts, $d_{u/s}=\dim W^{u/s}$, this induces a decomposition of the coordinates
$x=(x_s,x_u)$. 
We will denote by $A_u = A|_{W^u} \oplus 0_{W^s}$, $A_s = 0_{W^u} \oplus A|_{W^s}$ and $c > 0$ such that
$$c < \inf_{\lambda \in \mathrm{sp}(A)} |\Re(\lambda)|$$
 where $\mathrm{sp}(A)$ is the spectrum of $A$.

Let $\chi_1, \chi_2 \in \Omega^\bullet(M)$ such that $\supp \chi_i \subset \supp \chi$ for $i=1,2$. For simplicity, we identify $\e^{-tA}$ and its action on differential forms and currents given by the pull-back, $\delta^d(x)$ denotes the Dirac $\delta$ distribution at $0 \in \mathbb{R}^d$,
$\pi_1,\pi_2$ are the projections $M\times M\mapsto M$ on the first and second factor respectively.
$$
\begin{aligned}
\langle \mathcal{K}_{\chi,t}, \pi_1^*\chi_1  \wedge \pi_2^*\chi_2 \rangle &= \langle \chi_2, ~\e^{-tA}(\id - \widehat{\Pi}) \chi_1 \rangle \\
&= \left\langle \chi_2,~ \e^{-tA} \left(\chi_1 - \delta^{d_u}(x_u) \dd x_u \int_{W^s} \pi_{s,0}^* \chi_1\right)\right \rangle \\
&= \left \langle \e^{tA_s }\chi_2, ~\e^{-tA_u} \chi_1 \right \rangle - \left(\int_{W^u} \pi_{u,0}^*\chi_2\right)\left( \int_{W^s}\pi_{s,0}^*\chi_1\right) \\
&= \int_0^1  \int_U \partial_\tau  \left(\e^{tA_s}\pi_{u,\tau}^* \chi_2 \wedge \e^{-tA_u }\pi_{s,\tau}^* \chi_1\right) \dd \tau,
\end{aligned}
$$
where $\pi_{u,\tau}, \pi_{s,\tau} : U \to U$ are defined by $\pi_{u,\tau}(x_u,x_s) = (x_u, \tau x_s)$ and $\pi_{s,\tau}(x_u, x_s) = (\tau x_u, x_s)$.
Now write $\chi_2 = \sum_{|I|= k} \beta_I \dd x_s^{I_s} \wedge \dd x_u^{I_u}.$ We have
$$
\begin{aligned}
\partial_\tau \pi_{u,\tau}^* \chi_2(x_u, x_s) &= \partial_\tau \sum_I \tau^{|I_s|} \beta_I(x_u, \tau x_s) \dd x_u^{I_u} \wedge \dd x_s^{I_s} \\
&= \sum_I |I_s| \tau^{|I_s| - 1}\beta_I(x_u, \tau x_s) \dd x_u^{I_u} \wedge \dd x_s^{I_s} \\
& \quad \quad \quad \quad + \sum_I \tau^{|I_s|} \left(\partial_{x_s}\beta_I\right)_{(x_u, \tau x_s)} (x_s) \dd x_u^{I_u} \wedge \dd x_s^{I_s}.
\end{aligned}
$$
Therefore 
$$
\partial_\tau \e^{tA_s}\pi_{u,\tau}^* \chi_2 = \sum_I \left(|I_s| \tau^{|I_s|-1}\beta_I(x_u, \tau \e^{tA_s}x_s) + \tau^{|I_s|} \left(\partial_{x_s}\beta_I\right)_{(x_u, \tau x_s)} (\e^{tA_s}x_s) \right) \e^{tA_s} \dd x^I.
$$
Because $|\e^{tA_s}x_s| = \dom(\e^{-tc})$ and $ \e^{tA_s} \dd x^I = \dom(\e^{-c t|I_s|})$, $I=(I_s,I_u)$ is a multi--index and repeating the
same argument for $\partial_\tau \e^{-tA_u} \pi_{s,\tau}^*\chi_1$, we obtain the bound~:
\begin{equation}\label{eq:partialtau}
\partial_\tau \left(\e^{tA_s} \pi_{u, \tau}^* \chi_2 \wedge  \e^{-tA_u}\pi_{s,\tau}^*\chi_1\right) = \dom_{\chi_1, \chi_2} (\e^{-tc}).
\end{equation}
Replacing $\chi_1$ and $\chi_2$ by $\chi_1 \e^{i \langle\xi, \cdot \rangle}$  and $\chi_2 \e^{i\langle \eta, \cdot \rangle}$ with $\xi, \eta \in \mathbb{R}^n$, one gets
$$
\begin{aligned}
\Bigl\langle \mathcal{K}_{\chi,t}&, \pi_1^*\left(\chi_1\e^{i \langle\xi, \cdot \rangle}\right)  \wedge \pi_2^*\left(\chi_2 \e^{i \langle\eta, \cdot \rangle}\right)\Bigr\rangle \\
&= \int_0^1 \int_U \partial_\tau \left(\e^{tA_s} \pi_{u, \tau}^* \chi_2 \wedge  \e^{-tA_u}\pi_{s,\tau}^*\chi_1\right) \e^{i\langle \e^{tA_s}(x_u, \tau x_s), \eta \rangle} \e^{i\langle \e^{-tA_u}(\tau x_u, x_s), \xi \rangle} \dd \tau\\
&\quad + \int_0^1 \int_U e^{tA_s} \pi_{u, \tau}^* \chi_2 \wedge  \e^{-tA_u}\pi_{s,\tau}^*\chi_1 \partial_\tau \left(\e^{i\langle \e^{tA_s}(x_u, \tau x_s), \eta \rangle} \e^{i\langle \e^{-tA_u}(\tau x_u, x_s), \xi \rangle}\right) \dd \tau.
\end{aligned}
$$
Denoting $g(\tau, x_u, x_s) = \e^{i\langle \e^{tA_s}(x_u, \tau x_s), \eta \rangle} \e^{i\langle \e^{-tA_u}(\tau x_u, x_s), \xi \rangle}$ we have
$$\partial_\tau g(\tau, x_u, x_s) = i\left(\langle \e^{tA_s}x_s, \eta_s\rangle + \langle\e^{-tA_u} x_u, \xi_u\rangle\right) g(\tau, x_u, x_s) = \dom_{\Cinf(M)}(\e^{-tc}),$$
because $|\e^{tA_s}x_s|, |\e^{-tA_u}x_u|Ê= \dom(\e^{-tc})$.
Repeating the process that led to (\ref{eq:partialtau}) 
but for derivatives of $\chi_1,\chi_2$ as test forms with successive integration by parts, 
we therefore obtain for any $N \in \mathbb{N}$:
$$
\begin{aligned}
\Big\vert\Bigl\langle \mathcal{K}_{\chi,t}&, \pi_1^*\left(\chi_1\e^{i \langle\xi_1, \cdot \rangle}\right)  \wedge \pi_2^*\left(\chi_2 \e^{i \langle\xi_2, \cdot \rangle}\right)\Bigr\rangle\Bigr\vert \\
& \leq C_{N, \chi_1, \chi_2} \e^{-tc} \left(1 + |\e^{tA_s}\eta_s| + |\e^{-tA_u} \xi_u|\right) \int_0^1 \Bigl(1 +|\tau \e^{tA_s} \eta_s + \xi_s| + |\tau \e^{-tA_u} \xi_u + \eta_u|\Bigr)^{-N} \dd \tau,
\end{aligned}
$$
where $\xi = (\xi_u, \xi_s)$ and $\eta = (\eta_u, \eta_s)$. Now assume $(\xi, \eta)$ is close to $N^*\Delta(T^*M)$, say
$$
\left|\frac{\xi}{|\xi|}+\frac{\eta}{|\eta|}\right|<\nu \text { and } 1-\nu<\frac{|\xi|}{|\eta|}<1+\nu
$$
for some $\nu > 0$. Then we have for any $\tau \in [0,1]$:
$$
|\tau \e^{tA_s} \eta_s + \xi_s| + |\tau \e^{-tA_u} \xi_u + \eta_u| \geq \left(1 - \e^{-tc}(1+\nu)\right)(|\xi_s| + |\eta_u|).
$$
As a consequence, if $\nu > 0$ is small enough so that $(1+\nu)\e^{-(t+\varepsilon)c} < 1$, for every $t \geq 0$, we obtain
$$
\left|\Bigl\langle \mathcal{K}_{\chi,t+\varepsilon}, \pi_1^*\left(\chi_1\e^{i \langle\xi, \cdot \rangle}\right)  \wedge \pi_2^*\left(\chi_2 \e^{i \langle\eta, \cdot \rangle}\right)\Bigr\rangle\right| \leq C'_{N, \chi_1, \chi_2} (1 + |\xi| + |\eta|)^{-N},
$$
which concludes.
\end{proof}

\begin{proof}[Proof of Proposition \ref{prop:resolventmorse}]
Fix $\varepsilon > 0$. For $a \in \crit(f)$, take $c_a, \Gamma_a, \chi_a$ as in Lemma \ref{lem:wfsetpropagator}. The proof of Lemma \ref{lem:wfsetpropagator} actually shows that for $\Re(s) > - c_a$, the integral
$$
G_{\chi_a, \varepsilon, s} = \int_0^\infty \e^{-ts} \chi_a \e^{-(t+\varepsilon){\widetilde{X}}}(\id-\widehat{\Pi})\chi_a \dd t
$$
converges as an operator $\Omega^\bullet(M) \to \mathcal{D}^{'\bullet}(M)$. Moreover, its Schwartz kernel $\mathcal{G}_{\chi_a, \varepsilon, s}$ is locally bounded in $\mathcal{D}^{'n}_{\Gamma_a}(M\times M)$ in the region $\{\Re(s) > -c_a\}$. We will need the following lemma.

\begin{lemm}\label{lem:farfromcrit} For any $\mu >0$, there is $\nu > 0$ with the following property. For every $x \in M$ such that $\mathrm{dist}(x, \crit(f)) \geq \mu$, it holds
$$\mathrm{dist}\left(x, \e^{-(t+\varepsilon){\widetilde{X}}}(x)\right) \geq \nu, \quad t \geq 0.$$
\end{lemm}
\begin{proof}
We proceed by contradiction. Suppose that there is $\mu > 0$ and sequences $x_m \in M$ and $t_m \geq \varepsilon$ such that $\mathrm{dist}\left(x_m, \e^{-t_m{\widetilde{X}}}(x_m)\right)Ê\to 0$ as $m \to \infty$ and $\mathrm{dist}(x_m, \crit(f)) \geq \mu$. Extracting a subsequence we may assume that $x_m \to x$, $t_m \to \infty$ (indeed if $t_m \to t_\infty < \infty$  then $x$ is a periodic point for $\widetilde{X}$, which does not exist) and for any $m$,
$$\e^{-t\widetilde{X}}(x_m) \to a \text{ and } \e^{t \widetilde{X}}(x_m) \to b  \text{ as } t \to \infty,$$
for some $a,b \in \crit(f)$. Since the space of broken curves $\overline{\Lie}(a,b)$ is compact (see \cite{audin2014morse}), we may assume that the sequence of curves $\gamma_m = \left\{\e^{t\widetilde{X}}(x_m),~t \in \mathbb{R}\right\}$ converges to a broken curve $\ell = (\ell^1, \dots, \ell^q) \in \overline{\Lie}(a,b)$ with $\ell^j \in \Lie(c_{j-1}, c_j)$ for some $c_0, \dots, c_q \in \crit(f)$ with $c_0 = a$ and $c_q = b$.  Because $x_m \to x$, the proof of \cite[Theorem 3.2.2]{audin2014morse} implies $x \in \ell^j$ for some $j$ so that $\e^{-t\widetilde{X}} x \to c_{j-1}$ as $t \to \infty$. Therefore replacing $x$ by $\e^{-t\widetilde{X}}(x)$ for $t$ big enough, we may assume that $x$ is contained in a Morse chart $\Omega(c_{j-1})$ near $c_{j-1}$. Then $c_{j-1} \neq a$. Indeed if it was not the case then we would have $\e^{-t_m\widetilde{X}}x_m \to a$ as $m \to \infty$ (since $x_m$ would be contained in $\Omega(a) \cap W^u(a)$ for big enough $m$ and $t_m \to \infty$), which is not the case since $\mathrm{dist}(x, \crit(f)) \geq \mu \implies  x \neq a$ and $\mathrm{dist}\left(x_m, \e^{-t_m{\widetilde{X}}}(x_m)\right)Ê\to 0$ as $m \to \infty$. Therefore the flow line of $x_m$ exits $\Omega(c_{j-1})$ in the past. We therefore obtain, since $\e^{-t_m\widetilde{X}} x_m \to x$, that there is $i < j-1$ so that $c_i = c_{j-1}$. This is absurd since the sequence $\bigl(\ind_f(c_i)\bigr)_{i=0,\dots,q}$ is strictly decreasing.
\end{proof}

By (\ref{eq:pimorse}) we have $\supp \mathcal{K}_{\widehat{\Pi}} \cap \Delta = \crit(f)$, where $\mathcal{K}_{\widehat{\Pi}}$ is the Schwartz kernel of $\widehat{\Pi}$ and $\Delta$ is the diagonal in $M\times M$; the same holds for $\e^{-(t+\varepsilon)\widetilde{X}}\widehat{\Pi} = \widehat{\Pi}$ (see \cite{dangwitten}). Moreover, Lemma \ref{lem:farfromcrit} implies that if $\chi \in \Cinf(M, [0,1])$ satisfies $\chi \equiv 1$ near $\Delta$ and has support close enough to $\Delta$, we have 
$$\chi \e^{-(t+\varepsilon)\widetilde{X}} \chi = \sum_a \chi_a \e^{-(t+\varepsilon)\widetilde{X}}\chi_a.$$
Let $c = \min_{a \in \crit(f)}  c_a$. For $\Re(s) > -c$,
$$
G_{\chi, \varepsilon, s} =\int_0^\infty \e^{-ts} \chi \e^{-(t+\varepsilon){\widetilde{X}}}(\id - \widehat{\Pi})\chi \dd t 
$$
defines an operator $\Omega^\bullet(M) \to \mathcal{D}^{'\bullet}(M)$, whose Schwartz kernel $\mathcal{G}_{\chi, \varepsilon, s}$ is locally bounded in $\mathcal{D}^{'n}_{\Gamma}(M\times M)$ in the region $\{\Re(s) > -c\}$, where $\Gamma = \bigcup_{a \in \crit(f)} \Gamma_a$.

Now for $\Re(s) \gg 0$, we have as a consequence of the Hille--Yosida Theorem applied to $\Lie_{\widetilde{X}}$ acting on suitable anisotropic spaces~\cite[3.2.3]{dangwitten}:
$$
\left(\Lie_{\widetilde{X}}+s\right)^{-1} = \int_0^\infty \e^{-ts} \e^{-t {\widetilde{X}}} \dd t: \Omega^\bullet(M)\mapsto \mathcal{D}^{\prime\bullet}(M) .
$$
Therefore for $\Re(s) \gg 0$, it holds
$$
G_{\chi, \varepsilon, s} = \chi \left(\Lie_{\widetilde{X}}+s\right)^{-1} (\id - \widehat{\Pi})\e^{-\varepsilon {\widetilde{X}}} \chi.
$$
Since both members are holomorphic in the region $\{\Re(s) > -c\}$ and coincide for $\Re(s) \gg 0$, they coincide in the region $\Re(s) > -c$. Let $\beta \in \Omega^1(M)$. We can compute for $\Re(s) \gg 0$, since $\iota_{\widetilde{X}} \widehat{\Pi} = 0$ by \cite{dangwitten},
$$
\begin{aligned}
\strf ~Ê\beta \iota_{\widetilde{X}} \left(\Lie_{\widetilde{X}}+s\right)^{-1} (\id - \widehat{\Pi})\e^{-\varepsilon \Lie_{\widetilde{X}}} &= \strf~Ê\beta \iota_{\widetilde{X}} G_{\chi, \varepsilon, s} \\
&= \int_0^\infty \e^{-ts} \strf~\beta \iota_{\widetilde{X}} \e^{-(t+\varepsilon) \widetilde{X}} (\id - \widehat{\Pi}) \\
&= \int_0^\infty \e^{-ts} \strf \beta \iota_{\widetilde{X}}\e^{-(t+\varepsilon) \widetilde{X}},
\end{aligned}
$$
where we could interchange the integral and the flat trace thanks to the bound obtained in Lemma \ref{lem:wfsetpropagator}. Now the Atiyah-Bott trace formula \cite{atiyah1967lefschetz} gives
$$
\strf \beta \iota_{\widetilde{X}} \e^{-(t+\varepsilon) \widetilde{X}} = 0
$$
since $\widetilde{X}$ vanishes at its critical points. By holomorphy this holds true for any $s$ such that $\Re(s) > -c$. In particular if $\lambda > 0$ is small enough
$$
\strf \beta \iota_{\widetilde{X}} \widehat{Y}(\id - \widehat{\Pi}) \e^{-\varepsilon \widetilde{X}} = \frac{1}{2i\pi} \int_{|s| = \lambda}\strf \beta \iota_{\widetilde{X}} \frac{\left(\Lie_{\widetilde{X}}+s\right)^{-1}}{s} (\id - \widehat{\Pi})\e^{-\varepsilon \Lie_{\widetilde{X}}} \dd s = 0,
$$
where $\displaystyle{\left(\Lie_{\widetilde{X}} + s\right)^{-1} = \widehat{Y} + \frac{\widehat{\Pi}}{s} + \dom(s)}.$ Therefore Proposition \ref{prop:resolventmorse} is proved in the case where $(E,\nabla)$ is the trivial bundle. The general case is handled similarly.
\end{proof}

\newcommand{\etalchar}[1]{$^{#1}$}

\end{document}